\documentclass[final,12pt,a4paper]{amsart} 

\usepackage{graphicx}
\usepackage[all]{xy}
\usepackage{placeins}
\usepackage{enumitem}
\usepackage{amssymb}
\usepackage{latexsym}
\usepackage{amsmath}
\usepackage{mathrsfs}
\usepackage{array,booktabs,longtable}
\usepackage{etoolbox}

\usepackage[notref, notcite]{showkeys}  
\usepackage[usenames,dvipsnames]{xcolor}
\usepackage{hyperref}
\hypersetup{
    colorlinks,
    linkcolor={red!50!black},
    citecolor={blue!50!black},
    urlcolor={blue!80!black}
}

\newcommand{\harxiv}[1]{\href{http://arxiv.org/abs/#1}{\texttt{arXiv:#1}}}
\newcommand{\hyref}[2]{\hyperref[#2]{#1~\ref*{#2}} }

\newcommand{\coloneqq}{\mathrel{\mathop:}=}

\usepackage{fullpage}

\makeatletter
\def\Ddots{\mathinner{\mkern1mu\raise\p@
\vbox{\kern7\p@\hbox{.}}\mkern2mu
\raise4\p@\hbox{.}\mkern2mu\raise7\p@\hbox{.}\mkern1mu}}
\makeatother

\newcommand{\spvdots}[1]{\text{\raisebox{#1ex}{$\vdots$}}}

\newcommand{\ccdots}{\cdot\!\cdot\!\cdot} 

\newcommand{\arf}[1]{\scriptstyle(#1)}    
\newcommand{\arft}[1]{ \arf{#1} \ar[dr] }         
\newcommand{\arfc}[1]{ \arf{#1} \ar[dr] \ar[ur] } 
\newcommand{\arfb}[1]{ \arf{#1} \ar[ur] }         

\newcommand{\lend}{\delta}                

\theoremstyle{plain}
\newtheorem{theorem}{Theorem}[section]
\newtheorem{conjecture}[theorem]{Conjecture}
\newtheorem{lemma}[theorem]{Lemma}
\newtheorem{corollary}[theorem]{Corollary}
\newtheorem{proposition}[theorem]{Proposition}

\theoremstyle{definition}
\newtheorem{remark}[theorem]{Remark}
\newtheorem{example}[theorem]{Example}
\newtheorem*{naive-algorithm}{Na\"ive algorithm}
\newtheorem*{refined-algorithm}{Refined algorithm}
\newtheorem{definition}[theorem]{Definition}
\newtheorem*{question}{Question}

\newtheorem{properties}[theorem]{Properties}

\newcommand{\Step}[1]{\medskip\noindent\emph{Step #1:}}
\newcommand{\Case}[1]{\medskip\noindent\emph{Case #1:}}

\newcommand{\RemainingCases}{\medskip\noindent\emph{Remaining cases: }}

\newcommand{\LLambda}{\Lambda(r,n,m)}

\newcommand{\catD}{\mathsf D}

\newcommand{\Db}{\sD^b}

\newcommand{\cX}{\mathcal{X}}
\newcommand{\cY}{\mathcal{Y}}
\newcommand{\cZ}{\mathcal{Z}}

\newcommand{\ray}[1]{\mathsf{ray}(#1)}                
\newcommand{\rayfrom}[1]{\mathsf{ray}_{\! +}(#1)}
\newcommand{\rayto}[1]{\mathsf{ray}_{\! -}(#1)}
\newcommand{\raythrough}[1]{\mathsf{ray}_{\! \pm}(#1)}
\newcommand{\coray}[1]{\mathsf{coray}(#1)}            
\newcommand{\corayfrom}[1]{\mathsf{coray}_{\! +}(#1)}
\newcommand{\corayto}[1]{\mathsf{coray}_{\! -}(#1)}
\newcommand{\coraythrough}[1]{\mathsf{coray}_{\! \pm}(#1)}

\newcommand{\begintabularhammock}{ \smallskip\noindent\hspace{0.05\textwidth}
                                   \begin{tabular}{@{} p{0.15\textwidth} @{} p{0.80\textwidth} @{}} }

\newcommand{\rk}[1]{\mathrm{rk}\, #1}

\newcommand{\listskip}{\\[0.7ex]}

\newcommand{\AAA}{{\mathbb{A}}}

\newcommand{\IZ}{{\mathbb{Z}}}

\newcommand{\IN}{{\mathbb{N}}}

\newcommand{\ko}{{\mathcal O}}

\newcommand{\kl}{{\mathcal L}}

\newcommand{\kd}{{\mathcal D}}
\newcommand{\kr}{{\mathcal R}}

\newcommand{\ku}{{\mathcal U}}

\newcommand{\kk}{{\mathbf{k}}}

\newcommand{\orth}{^\perp}

\newcommand{\inv}{^{-1}}
\newcommand{\blank}{-}

\DeclareMathOperator{\susp}{\mathsf{susp}}
\DeclareMathOperator{\cosusp}{\mathsf{cosusp}}

\DeclareMathOperator{\stHom}{\underline{\mathsf{Hom}}} 
\DeclareMathOperator{\Aut}{\mathsf{Aut}}
\DeclareMathOperator{\Out}{\mathsf{Out}}
\DeclareMathOperator{\Inn}{\mathsf{Inn}}

\DeclareMathOperator{\id}{{\mathsf{id}}}

\newcommand{\ind}[1]{\mathsf{ind}(#1)}
\renewcommand{\mod}[1]{\mathsf{mod}(#1)}
\newcommand{\stmod}[1]{\underline{\mathsf{mod}}(#1)}
\newcommand{\proj}[1]{\mathsf{proj}(#1)}

\DeclareMathOperator{\add}{\mathsf{add}}

\newcommand{\hi}[1]{h(#1)}

\newcommand{\thick}[2]{\mathsf{thick}_{#1}(#2)}

\DeclareMathOperator{\Hom}{\mathrm{Hom}}

\newcommand{\TTT}{\mathsf{T}\!}   
\newcommand{\FFF}{\mathsf{F}\!}   
\newcommand{\FF}{\mathsf{F}\!\!\:}

\newcommand{\SSS}{\mathsf{S}}

\renewcommand{\setminus}{\backslash}

\newcommand{\sA}{\mathsf{A}}
\newcommand{\sB}{\mathsf{B}}
\newcommand{\sC}{\mathsf{C}}
\newcommand{\sD}{\mathsf{D}}
\newcommand{\sE}{\mathsf{E}}

\newcommand{\sH}{\mathsf{H}}

\newcommand{\sK}{\mathsf{K}}

\newcommand{\sM}{\mathsf{M}}
\newcommand{\sN}{\mathsf{N}}

\newcommand{\sP}{\mathsf{P}}

\newcommand{\sT}{\mathsf{T}}
\newcommand{\sU}{\mathsf{U}}

\newcommand{\sX}{\mathsf{X}}
\newcommand{\sY}{\mathsf{Y}}
\newcommand{\sZ}{\mathsf{Z}}

\DeclareMathAlphabet{\mathpzc}{OT1}{pzc}{m}{it}

\newcommand{\cC}{\mathscr{C}}

\newcommand{\sod}[1]{{\langle #1\rangle}}      
\newcommand{\clext}[1]{{\langle #1\rangle}}    
\newcommand{\generate}[1]{{\langle #1\rangle}} 

\newcommand{\isom}{ \text{{\hspace{0.48em}\raisebox{0.8ex}{${\scriptscriptstyle\sim}$}}}
                    \hspace{-0.65em}{\rightarrow}\hspace{0.3em}} 

\newcommand{\embed}{\hookrightarrow}
\newcommand{\onto}{\twoheadrightarrow}
\newcommand{\arrd}{ \ar@{-}[r] \ar@{=}[d] }
\newcommand{\too}{\longrightarrow}
\renewcommand{\iff}{\Longleftrightarrow}

\newcommand{\rightlabel}[1]{\stackrel{#1}{\longrightarrow}}
\newcommand{\leftlabel}[1]{\stackrel{#1}{\longleftarrow}}

\newcommand{\xxrightarrow}[1]{\xrightarrow{\raisebox{-0.2ex}{\ensuremath{{\scriptscriptstyle #1}}}}} 

\newcommand{\tri}[3]{#1\rightarrow #2\rightarrow #3\rightarrow \Sigma #1}

\newcommand{\trilabels}[6]{#1\stackrel{#4}{\longrightarrow} #2\stackrel{#5}{\longrightarrow} #3\stackrel{#6}{\longrightarrow} \Sigma #1}

\newcommand{\paths}{\mathsf{Pa}}
\newcommand{\rpaths}{\widehat{\mathsf{Pa}}}

\newcommand{\strings}{\mathsf{St}}

\newcommand{\rLambda}{\text{\scalebox{0.9}{$\hat\Lambda$}}}
\newcommand{\rLLambda}{\rLambda(r,n,m)}
\newcommand{\rQ}{\hat{Q}}
\newcommand{\rrho}{\hat{\rho}}
\newcommand{\rp}{\hat{\bp}}
\renewcommand{\rq}{\hat{\mathbf{q}}}
\newcommand{\rS}{S}
\newcommand{\bp}{\mathbf{p}}

\newcommand{\ip}{\bar{p}}
\newcommand{\iv}{\bar{v}}
\newcommand{\iw}{\bar{w}}
\newcommand{\ia}{\bar{a}}
\newcommand{\ib}{\bar{b}}
\newcommand{\ic}{\bar{c}}
\newcommand{\invd}{\bar{d}}
\newcommand{\ix}{\bar{x}}

\newcommand{\Fac}[1]{\mathsf{Fac}(#1)}
\newcommand{\Sub}[1]{\mathsf{Sub}(#1)}

\newcommand{\sqmat}[4]{{\big(\genfrac{.}{.}{0pt}{1}{#1}{#3} \,
                        \genfrac{.}{.}{0pt}{1}{#2}{#4}\big) }}
\newcommand{\colmat}[2]{{\big(\genfrac{.}{.}{0pt}{1}{#1}{#2}\big) }}

\renewcommand{\phi}{\varphi}
\renewcommand{\epsilon}{\varepsilon}

\newcommand{\smxy}[1]{{\text{\tiny$#1$}}}

\newcommand{\arr}{\ar@{~}[r]}
\newcommand{\arrr}{\ar@{~}[rr]}

\newcommand{\arudd}{\ar[ur] \ar@{.}[dr]}
\newcommand{\aruu}{\ar@{.>}[uuurrr]}
\newcommand{\arurr}{\ar@{.>}[ur] \ar@{.}[rr]}

\newcommand{\bib}[6]{{\bibitem{#2} #3: {\emph{#4},} #5#6.}}

\newcommand{\bibno}[1]{}

\usepackage{tikz}
\usetikzlibrary{decorations.pathmorphing}
\usetikzlibrary{decorations.pathreplacing}
\usetikzlibrary{positioning}
\usetikzlibrary{calc}
\usetikzlibrary{backgrounds}
\usetikzlibrary{shapes.misc}

\usepackage{pifont} 
\newcommand{\dnr}[1]{ \node[shape=circle,draw,inner sep=0.75pt,black,thick] at (-1.4,1) {\text{\textsf{\scriptsize #1}}}; }

\newcommand{\tns}[1]{ \tikz[baseline=(char.base)]{ \node[shape=circle,draw,semithick,inner sep=0.75pt] (char) {\textsf{\scriptsize #1}}; } }

\newcommand{\Hht}[1]{ \scalebox{1.5}{$#1$} }
   
\newcommand{\Homhammocktemplate}[1]{
   \tikzset{ every path/.style = {->, >=latex, shorten <=6pt, shorten >=6pt, very thick} }
   
     \begin{scope}[gray!50]
   
       \foreach \x in {0,...,9} \foreach \y in {0,...,9} 
         { \node (\x\y) at (\x,\y) {};    
           \ifnumless{\x}{9}{ \ifnumless{\y}{9}{ \ifnumgreater{\x}{0}{ \ifnumgreater{\y}{#1} {
             \ifnumodd{\x+\y}{ \fill (\x,\y) circle (0.16); } {}
           } {} }{} } {} }  
         }
   
       \foreach \x in {0,...,8} \foreach \y in {9,...,1}
           { \ifnumodd{\x+\y}{ \draw (\x,\y) -- (\x+1,\y-1); }{ \draw (\x,\y-1) -- (\x+1,\y); } }
   
     \end{scope}
}   

  \tikzstyle{dot}   = [fill=gray!20, circle, inner sep=0pt, minimum size=3pt]
  \tikzstyle{box}   = [fill=black, rectangle, inner sep=0pt, minimum height=4pt, minimum width=4pt]
  \tikzstyle{minus} = [fill=white, rectangle, inner sep=0pt, minimum height=1pt, minimum width=2.5pt]
  \tikzstyle{plus}  = [draw=white, cross out, rotate=45, thick, minimum size=2.5pt, inner sep=0pt, outer sep=0pt]

  \tikzstyle{Xaisle} = [fill, gray!95, line width=8pt]
  \tikzstyle{Yaisle} = [fill, gray!30, line width=8pt]
  \newcommand{\Xdot}[1]{ \fill[Xaisle] (#1) circle (16pt); }
  \newcommand{\Ydot}[1]{ \fill[Yaisle] (#1) circle (16pt); }
  
  \newcommand{\boxo}[1]{ \node[box] at (#1) {}; }
  \newcommand{\boxp}[1]{ \node[box] at (#1) {}; \node[plus] at (#1) {}; }
  \newcommand{\boxm}[1]{ \node[box] at (#1) {}; \node[minus] at (#1) {}; }

\begin{document}

\title[Discrete derived categories I]{Discrete derived categories I \\[0.5ex] 
                      \scalebox{0.86}{Homomorphisms, autoequivalences and t-structures}}

\author{Nathan Broomhead}
\author{David Pauksztello}
\author{David Ploog}



\begin{abstract}
Discrete derived categories were studied initially by Vossieck \cite{Vossieck} and later by Bobi\'nski, Gei\ss, Skowro\'nski \cite{BGS}. In this article, we describe the homomorphism hammocks and autoequivalences on these categories. We classify silting objects and bounded t-structures.
\end{abstract}

\maketitle

\noindent
\hspace*{0.125\linewidth} \hspace*{-1.2em}  
\parbox{0.75\linewidth}{                    
  \small
  \setcounter{tocdepth}{1}
  \tableofcontents
}

\addtocontents{toc}{\protect{\setcounter{tocdepth}{-1}}}  

\section*{Introduction}
\addtocontents{toc}{\protect{\setcounter{tocdepth}{1}}}   

\noindent
In this article, we study the bounded derived categories of finite-dimensional algebras that are \emph{discrete} in the sense of Vossieck \cite{Vossieck}. Informally speaking, discrete derived categories can be thought of as having structure intermediate in complexity between the derived categories of hereditary algebras of finite representation type and those of tame type. Note, however, that the algebras with discrete derived categories are \emph{not} hereditary.  We defer the precise definition until the beginning of the next section.

Understanding homological properties of algebras means understanding the structure of their derived categories. We investigate several key aspects of the structure of discrete derived categories: the structure of homomorphism spaces, the autoequivalence groups of the categories, and the t-structures and co-t-structures inside discrete derived categories.

The study of the structure of algebras with discrete derived categories was begun by Vossieck, who showed that they are always gentle and classified them up to Morita equivalence. Bobi\'nski, Gei\ss\ and Skowro\'nski \cite{BGS} obtained a canonical form for the derived equivalence class of these algebras; see Figure~\ref{fig:canonical-form}.
This canonical form is parametrised by integers $n\geq r \geq 1$ and $m>0$, and the corresponding algebra denoted by $\LLambda$. We restrict to parameters $n>r$, which is precisely the case of finite global dimension.
In \cite{BGS}, the authors also determined the components of the Auslander--Reiten (AR) quiver of derived-discrete algebras and computed the suspension functor.

The structure exhibited in \cite{BGS} is remarkably simple, which brings us to our principal motivation for studying these categories. Discrete derived categories are sufficiently straightforward to make explicit computation highly accessible but also non-trivial enough to manifest interesting behaviour. For example, discrete derived categories contain natural examples of spherelike objects in the sense of \cite{HKP}. If one takes one of these spherelike objects and forms the smallest subcategory generated by it, this category is then equivalent to a triangulated category generated by a spherical object. Such categories have previously been studied in the context of (higher) cluster categories of type $A_\infty$ in \cite{HJ2,HJY}. Indeed, we shall see that every discrete derived category contains two such higher cluster categories as proper subcategories when the algebra has finite global dimension.

Furthermore, the structure of discrete derived categories is highly reminiscent of the categories of perfect complexes of cluster-tilted algebras of type $\tilde{A}_n$ studied in \cite{AG}. This suggests approaches developed here to understand discrete derived categories are likely to find applications more widely in the study of derived categories of gentle algebras.

The basis of our work is giving a combinatorial description via AR quivers of which indecomposable objects admit non-trivial homomorphism spaces between them, so called `Hom-hammocks'.
As a byproduct, we get the following interesting property of these categories:
the dimensions of the homomorphism spaces between indecomposable objects have a common bound. In fact, in Theorem~\ref{thm:hom-dimensions} we show there are unique homomorphisms, up to scalars, whenever $r>1$, and in the exceptional case $r=1$, the common dimension bound is $2$. 
We believe this property holds independent interest and warrants further investigation. See \cite{Han-Zhang} for a different approach to measuring the `smallness' of discrete derived categories. As another for categorical size, the Krull--Gabriel dimension of discrete derived categories has been computed in \cite{BK}; it is at most 2.

In Theorem~\ref{thm:auteq} we explicitly describe the group of autoequivalences.  For this, we introduce a generalisation of spherical twist functors arising from cycles of exceptional objects. The action of these twists on the AR components of $\LLambda$ is a useful tool, which is frequently employed here.

In Section~\ref{sec:classifications}, we address the classification of bounded t-structures and co-t-structures in $\Db(\LLambda)$, which are important in understanding the cohomology theories occurring in triangulated categories, and have recently become a focus of intense research as the principal ingredients in the study of Bridgeland stability conditions \cite{Bridgeland}, and their co-t-structure analogues \cite{JP}. Further investigation into the properties of (co-)t-structures and the stability manifolds is conducted in the sequel \cite{BPP2}; see also \cite{Woolf}.

We study the (co-)t-structures indirectly via certain generating sets: silting subcategories, which behave like the projective objects of hearts of bounded t-structures and generalise tilting objects. In general, one cannot get all bounded t-structures in this way, but in Proposition~\ref{prop:length}, we show that the heart of each bounded t-structure in $\Db(\LLambda)$ is equivalent to $\mod{\Gamma}$, where $\Gamma$ is a finite-dimensional algebra of finite representation type. The upshot is that using the bijections of K\"onig and Yang \cite{Koenig-Yang}, classifying silting objects is enough to classify all bounded (co-)t-structures.
We show that $\Db(\LLambda)$ admits a semi-orthogonal decomposition into $\Db(\kk A_{n+m-1})$ and the thick subcategory generated by an exceptional object. Using Aihara and Iyama's silting reduction \cite{AI}, we classify the silting objects in Theorem~\ref{thm:silting-classification}.
We finish with an explicit example of $\Lambda(2,3,1)$ in Section~\ref{sec:example}.

\smallskip

\noindent\textbf{Acknowledgments:}
We are grateful to Aslak Bakke Buan, Christof Gei\ss, Martin Kalck, Henning Krause, and Dong Yang for answering our questions and particularly to an anonymous referee for a careful reading and many valuable comments. Much of this paper was prepared while all three authors worked at Leibniz Universit\"at Hannover. The second author acknowledges the financial support of the EPSRC of the United Kingdom through the grant EP/K022490/1.


\section{Discrete derived categories and their AR-quiver}
\label{sec:discrete_categories}
\label{sec:AR-quiver}

We always work over a fixed algebraically closed field $\kk$. All modules will be finite-dimensional right modules. Throughout, all subcategories will be additive and closed under isomorphisms.

\subsection{Discrete derived categories}

We are interested in $\kk$-linear, Hom-finite triangulated categories which are small in a certain sense. One precise definition of such smallness is given by Vossieck \cite{Vossieck}; here we present a slight generalisation of his notion.

\begin{definition} \label{def:DDC}
A derived category (or, more generally and intrinsically, a Hom-finite triangulated category with a bounded t-structure) $\catD$ is \emph{discrete (with respect to this $t$-structure)}, if for every map $v\colon \IZ\to K_0(\catD)$ there are only finitely many isomorphism classes of objects $D\in\catD$ with $[H^i(D)]=v(i)\in K_0(\catD)$ for all $i\in\IZ$.
\end{definition}

Let us elaborate on the connection to \cite{Vossieck}: Vossieck speaks of \emph{finitely supported, positive} dimension vectors $v\in K_0(\catD)^{(\IZ)}$ which he can do since he has $\catD=\Db(\Lambda)$ for a finite-dimensional algebra $\Lambda$, so $K_0(\Lambda)\cong\IZ^r$. In our slight generalisation of his notion, we cannot do so, but for finite-dimensional algebras the new notion gives back the old one: if $v$ is negative somewhere, there will be no objects of that dimension vector whatsoever. For the same reason, we don't have to assume that $v$ has finite support: if it doesn't, the set of objects of that class is empty.

Note that our definition of discreteness appears to depend on the choice of bounded t-structure. Throughout this article, we shall be interested in the bounded derived category $\Db(\Lambda)$ of a finite-dimensional algebra $\Lambda$. We shall always use discreteness with respect to the standard t-structure, whose heart is $\mod{\Lambda}$, the category of finite-dimensional right $\Lambda$-modules. However, in \cite{DTC}, the results of this article will be used to show that the categories studied here are discrete with respect to any bounded t-structure.

Obviously, derived categories of path algebras of type ADE Dynkin quivers are examples of discrete categories. Moreover, \cite{Vossieck} shows that the bounded derived category of a finite-dimensional algebra $\Lambda$, which is not of finite representation type, is discrete if and only if $\Lambda$ is Morita equivalent to the bound quiver algebra of a gentle quiver with exactly one cycle having different numbers of clockwise and anticlockwise orientations.

Furthermore, in \cite{BGS}, Bobi\'nski, Gei\ss\ and Skowro\'nski give a derived Morita classification of such algebras. More precisely, for $\Lambda$ connected and not of Dynkin type, the derived category $\Db(\Lambda)$ is discrete if and only if $\Lambda$ is derived equivalent to the path algebra $\LLambda$ for the quiver with relations given in Figure~\ref{fig:gentle_algebra}, and some values of $r,n,m$.

\begin{figure}
\includegraphics[width=0.75\textwidth]{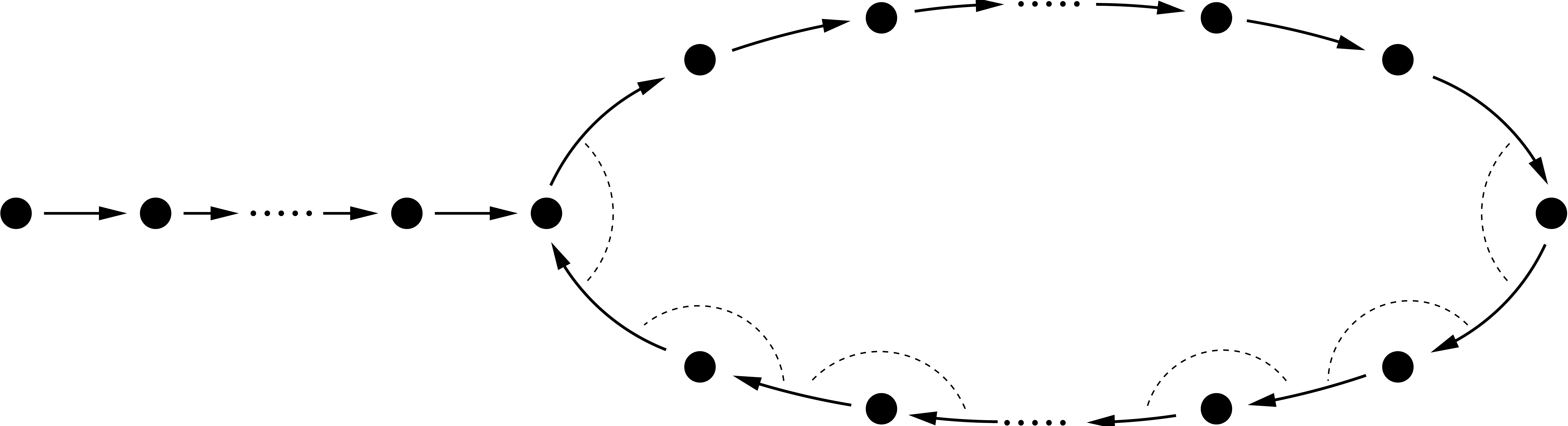}
\caption{ \label{fig:canonical-form} \label{fig:gentle_algebra}
The quiver $Q(r,n,m)$ consisting of an oriented cycle of length $n$ with
a tail of length $m$ and $r$ consecutive zero relations inside the cycle.
}
\end{figure}

\subsection{The AR quiver of $\boldsymbol{\Db(\LLambda)}$}

The algebra $\LLambda$ has finite global dimension if and only if $n>r$.
\emph{In the following, we always make this assumption.} Therefore the derived category $\Db(\LLambda)$ enjoys duality in the form
\[ \Hom(A,B) = \Hom(B,\SSS A)^* \]
functorially in $A,B\in\Db(\LLambda)$, where the Serre functor $\SSS$ is given by the Nakayama functor. In other words, $\Db(\LLambda)$ has Auslander--Reiten triangles and translation $\tau \coloneqq\Sigma^{-1}\SSS$. We will use both notations, depending on the context. Some general properties of $\Db(\LLambda)$ are: this triangulated category is algebraic, Hom-finite, Krull--Schmidt and indecomposable; see Appendix~\ref{app:category-properties} for details.

We collect together some more special properties of $\Db(\LLambda)$ which will be crucial throughout the paper; the reference is \cite{BGS}. 
By \cite[Theorem B]{BGS}, the AR quiver of $\Db(\LLambda)$ has precisely $3r$ components; these are denoted by
\[ \cX^0,\ldots,\cX^{r-1}, \qquad \cY^0,\ldots,\cY^{r-1}, \qquad \cZ^0,\ldots,\cZ^{r-1} . \]
The $\cX$ and $\cY$ components are of type $\IZ A_\infty$, whereas the $\cZ$ components are of type $\IZ A_\infty^\infty$.
It will be convenient to have notation for the subcategories generated by indecomposable objects of the same type:
\[ \cX \coloneqq \add \bigcup_i\cX^i, \qquad \cY \coloneqq \add \bigcup_i\cY^i, \qquad \cZ \coloneqq \add \bigcup_i\cZ^i .\]
For each $k= 0,\ldots, r-1$, we label the indecomposable objects in $\cX^k, \cY^k, \cZ^k$ as follows:
\[ X^k_{ij} \in \cX^k \text{ with } i,j\in\IZ, j\geq i; \quad
   Y^k_{ij} \in \cY^k \text{ with } i,j\in\IZ, i\geq j; \quad
   Z^k_{ij} \in \cZ^k \text{ with } i,j\in\IZ . \]

\begin{properties} \label{properties}
This labelling is chosen in such a way that the following properties hold:
\begin{enumerate}[leftmargin=2em,itemsep=0.8ex,label = (\arabic*)]
\item Irreducible morphisms go from an object with coordinate $(i,j)$ to objects $(i+1,j)$
      and $(i,j+1)$ in the same component (when they exist).

\bigskip
\noindent
\parbox{0.27\textwidth}{
$\cX$ coordinates:
}
\hfill
\parbox{0.27\textwidth}{
$\cY$ coordinates:
}
\hfill
\parbox{0.27\textwidth}{
$\cZ$ coordinates:
}

\noindent
\parbox{0.27\textwidth}{
\raisebox{-10ex}{\scalebox{0.75}{$\xymatrix@M=0.1em@H=0.0ex@W=0.4ex@!0{
\text{\raisebox{1.2ex}{\rotatebox{-10}{$\ddots$}}} \ar[dr]
               & \spvdots{4}    & \arft{-1,1} & \spvdots{4}   & \arft{0,2} & \spvdots{4}    &
\text{\raisebox{0.4ex}{\rotatebox{10}{$\Ddots$}}} \\
               & \arfc{-1,0} &             & \arfc{0,1} &            & \arfc{1,2}  &        \\
\cdots \ar[ur] &             & \arfb{0,0}  &            & \arfb{1,1} &             & \cdots
}$}}
}
\hfill
\parbox{0.27\textwidth}{
\raisebox{6.5ex}[0ex][0ex]{\scalebox{0.75}{$\xymatrix@M=0.1em@H=0.0ex@W=0.4ex@!0{
\cdots \ar[dr] &             & \arft{0,0}  &            & \arft{1,1} &             & \cdots \\
               & \arfc{0,-1} &             & \arfc{1,0} &            & \arfc{2,1}  &        \\
\text{\raisebox{0.4ex}{\rotatebox{10}{$\Ddots$}}} \ar[ur]
               & \spvdots{-4}    & \arfb{1,-1} & \spvdots{-4}   & \arfb{2,0} & \spvdots{-4}    &
\text{\raisebox{1.2ex}{\rotatebox{-10}{$\ddots$}}}
}$} }
}
\hfill
\parbox{0.27\textwidth}{
\scalebox{0.75}{$\xymatrix@M=0.1em@H=0.0ex@W=0.4ex@!0{
\text{\raisebox{1.2ex}{\rotatebox{-10}{$\ddots$}}} \ar[dr]
               & \spvdots{4}    & \arft{-1,1} & \spvdots{4}   & \arft{0,2} & \spvdots{4}    &
\text{\raisebox{0.4ex}{\rotatebox{10}{$\Ddots$}}} \\
               & \arfc{-1,0} &             & \arfc{0,1} &            & \arfc{1,2}  &       & \\
\cdots \ar[dr] \ar[ur] &      & \arfc{0,0} &             & \arfc{1,1}  &            & \cdots \\
               & \arfc{0,-1}  &             & \arfc{1,0} &            & \arfc{2,1}  &      &  \\
\text{\raisebox{0.4ex}{\rotatebox{10}{$\Ddots$}}} \ar[ur]
               & \spvdots{-4}    & \arfb{1,-1} & \spvdots{-4}   & \arfb{2,0} & \spvdots{-4}    &
\text{\raisebox{1.2ex}{\rotatebox{-10}{$\ddots$}}} &
}$}
}
\hfill
\item The AR translate of an object with coordinate $(i,j)$ is the object with coordinate
      $(i-1,j-1)$ in the same component, i.e.\ $\tau X^k_{i,j} = X^k_{i-1,j-1}$ etc.
\item The suspension of indecomposable objects is given below, with $k=0,\ldots,r-2$:
\[ \begin{array}{l@{\:}l l@{\:}l}
    \Sigma X^k_{ij} &= X^{k+1}_{ij},  &\quad  \Sigma X^{r-1}_{ij} &= X^0_{i+r+m,j+r+m}, \\[1ex]
    \Sigma Y^k_{ij} &= Y^{k+1}_{ij},  &\quad  \Sigma Y^{r-1}_{ij} &= Y^0_{i+r-n,j+r-n}, \\[1ex]
    \Sigma Z^k_{ij} &= Z^{k+1}_{ij},  &\quad  \Sigma Z^{r-1}_{ij} &= Z^0_{i+r+m,j+r-n}
\end{array} \]
In particular, $\Sigma^r|_\cX = \tau^{-m-r}$ and $\Sigma^r|_\cY = \tau^{n-r}$ on objects.
\item There are distinguished triangles, for any $i,j,d\in\IZ$ with $d\geq0$:
\[ \xymatrix@R=0ex{
    X^k_{i,i+d} \ar[r] & Z^k_{ij} \ar[r] & Z^k_{i+d+1,j} \ar[r] & \Sigma X^k_{i,i+d}, \\
    Y^k_{j+d,j} \ar[r] & Z^k_{ij} \ar[r] & Z^k_{i,j+d+1} \ar[r] & \Sigma Y^k_{j+d,j}.
} \]
\item There are chains of non-zero morphisms for any $i\in\IZ$ and $k=0,\ldots,r-1$:
\[ \xymatrix@C=0.8em@R=0ex{
   X^k_{ii}    \ar[r] & X^k_{i,i+1} \ar[r] & \ccdots \ar[r] &
   Z^k_{i,i-1} \ar[r] & Z^k_{ii} \ar[r] & Z^k_{i,i+1} \ar[r] & \ccdots \ar[r] &
   \Sigma X^k_{i+1,i-1} \ar[r] & \Sigma X^k_{i,i-1} \ar[r] & \Sigma X^k_{i-1,i-1},
\\
   Y^k_{ii} \ar[r] & Y^k_{i+1,i} \ar[r] & \ccdots    \ar[r] &
   Z^k_{i-1,i} \ar[r] & Z^k_{ii} \ar[r] & Z^k_{i+1,i} \ar[r] & \ccdots \ar[r] &
   \Sigma Y^k_{i-1,i-3} \ar[r] & \Sigma Y^k_{i-1,i-2} \ar[r] & \Sigma Y^k_{i-1,i-1}.
} \]
\end{enumerate}
\end{properties}

Later, we will often use the `height' of indecomposable objects in $\cX$ or $\cY$ components. \label{def:height}
For $X^k_{ij} \in \ind{\cX^k}$, we set $h(X^k_{ij}) = j-i$ and call it the \emph{height} of $X^k_{ij}$ in the component $\cX^k$.
Similarly, for $Y^k_{ij} \in \ind{\cY^k}$, we set $h(Y^k_{ij}) = i-j$ and call it the \emph{height} of $Y^k_{ij}$ in the component $\cY^k$. The \emph{mouth} of an  $\cX$ or $\cY$ component consists of all objects of height 0.


\section{Hom spaces: hammocks}
\label{sec:hammocks}

\noindent
For brevity, we will write $\Lambda \coloneqq \LLambda$.
In this section, for a fixed indecomposable object $A\in\Db(\Lambda)$ we compute the so-called `Hom-hammock' of $A$, i.e.\ the set of indecomposables $B\in\Db(\Lambda)$ with $\Hom(A,B)\neq0$. By duality, this also gives the contravariant Hom-hammocks: $\Hom(\blank,A) = \Hom(\SSS\inv A,\blank)^*$. Therefore we generally refrain from listing the $\Hom(\blank,A)$ hammocks explicitly.

The precise description of the hammocks is slightly technical. However, the result is quite simple, and the following schematic indicates the hammocks $\Hom(X,\blank)\neq0$ and $\Hom(Z,\blank)\neq0$ for indecomposables $X\in\cX$ and $Z\in\cZ$:

\bigskip

\begin{center}
\begin{tikzpicture}[scale=0.15,thick]

  \newcommand{\sst}[1]{\scalebox{0.75}{$#1$}}

  \draw[fill,gray!30] (16,-7) -- (19,-10) -- (20.5, -8.5) -- (17.5, -5.5) -- cycle;
  \draw[fill,gray!30] (18.5,-4.5) -- (21.5,-7.5) -- (31.5,2.5) -- (28.5, 5.5) -- cycle;
  \draw[fill,gray!30] (32.5,3.5) -- (36,7) -- (33,10) -- (29.5,6.5) -- cycle;
  \draw[fill] (16,-7) node [left =-0.1]  {$\scriptstyle X$}      circle (0.2);
  \draw[fill] (36, 7) node [right=-0.05] {$\scriptstyle \SSS X$} circle (0.2);

  \draw[fill,gray!30]  (50,2) -- (54,-2) -- (56,0) -- (52,4) -- cycle;
  \draw[fill,gray!30]  (53,5) -- (58,10) -- (62,6) -- (57,1) -- cycle;
  \draw[fill,gray!30]  (55,-3) -- (62,-10) -- (64,-8) -- (57,-1) -- cycle;
  \draw[fill,gray!30]  (58,0) -- (63,5) -- (70,-2) -- (65,-7) -- cycle;
  \draw[fill] (50, 2) node [left =-0.1]  {$\scriptstyle Z$}      circle (0.2);
  \draw[fill] (70,-2) node [right=-0.05] {$\scriptstyle \SSS Z$} circle (0.2);

  \node at ( 7.5, 8.75)  {\sst{\cY^0}};  \node at (29.5, 8.75)  {\sst{\cX^1}};  \node at (51.5, 8.75)  {\sst{\cY^2}};
  \node at ( 7.5,-8.75)  {\sst{\cX^0}};  \node at (29.5,-8.75)  {\sst{\cY^1}};  \node at (51.5,-8.75)  {\sst{\cX^2}};
  \node at (16.5,0.2) {\sst{\cZ^0}};  \node at (38.5,0.2) {\sst{\cZ^1}};  \node at (60.5,0.2) {\sst{\cZ^2}};

  \draw[dashed] (0,  8) -- ( 2, 10) -- (12, 0) -- (2,-10) -- (0,-8); 
  \draw[dashed] (4, 10) -- (13,  1) -- (22, 10);
  \draw[dashed] (4,-10) -- (13, -1) -- (22,-10);
  \draw         (4, 10) -- (22, 10);
  \draw         (4,-10) -- (22,-10);
  \draw[dashed] (14, 0) -- (24, 10) -- (34, 0) -- (24,-10) -- cycle;
  \draw[dashed] (26, 10) -- (35,  1) -- (44, 10);
  \draw[dashed] (26,-10) -- (35, -1) -- (44,-10);
  \draw         (26, 10) -- (44, 10);
  \draw         (26,-10) -- (44,-10);
  \draw[dashed] (36, 0) -- (46, 10) -- (56, 0) -- (46,-10) -- cycle;
  \draw[dashed] (48, 10) -- (57,  1) -- (66, 10);
  \draw[dashed] (48,-10) -- (57, -1) -- (66,-10);
  \draw         (48, 10) -- (66, 10);
  \draw         (48,-10) -- (66,-10);
  \draw[dashed] (58, 0) -- (68, 10) -- (78, 0) -- (68,-10) -- cycle;
  \draw[dashed] (70, 10) -- (79,  1) -- (81, 3);
  \draw[dashed] (70,-10) -- (79, -1) -- (81,-3);
  \draw         (70, 10) -- (81, 10);
  \draw         (70,-10) -- (81,-10);
\end{tikzpicture}
\end{center}

\smallskip

\subsection{Hammocks from the mouth}

We start with a description of the Hom-hammocks of objects at the mouths of all $\IZ A_\infty$ components. The proof relies on Happel's triangle equivalence of $\Db(\LLambda)$ with the stable module category of the repetitive algebra of $\LLambda$. As the repetitive algebras are special biserial algebras, the well-known theory of string (and band) modules provides a useful tool to understand the indecomposable objects and homomorphisms between them; we summarise this theory in Appendix~\ref{app:strings}.

To make our statements of Hom-hammocks more readable, we employ the language of rays and corays. Let $V=V_{i,j}$ be an indecomposable object of $\Db(\LLambda)$ with coordinates $(i,j)$.  Recall the conventions that $j\geq i$ if $V\in\cX$ whereas $i\geq j$ if $V\in\cY$. Denoting the AR component of $V$ by $\cC$ and its objects by $V_{a,b}$, the following six definitions give the
 \emph{rays/corays from/to/through $V$}, respectively
\begin{align*}
  \rayfrom{V_{i,j}}      & \coloneqq \{ V_{i,j+l} \in \cC \mid l\in\IN \}, &&&
  \corayfrom{V_{i,j}}    & \coloneqq \{ V_{i+l,j} \in \cC \mid l\in\IN \}, \\[-0.5ex]
  \rayto{V_{i,j}}        & \coloneqq \{ V_{i,j-l} \in \cC \mid l\in\IN \}, &&&
  \corayto{V_{i,j}}      & \coloneqq \{ V_{i-l,j} \in \cC \mid l\in\IN \}, \\[-0.5ex]
  \raythrough{V_{i,j}}   & \coloneqq \{ V_{i,j+l} \in \cC \mid l\in\IZ \}, &&&
  \coraythrough{V_{i,j}} & \coloneqq \{ V_{i+l,j} \in \cC \mid l\in\IZ \}.
\end{align*}
Note that, because of the orientation of the components, the (positive) ray of an indecomposable $X^k_{ii}\in\cX^k$ at the mouth consists of indecomposables in $\cX^k$ reached by arrows going out of $X^k_{ii}$, 
while in the $\cY$ components the (negative) ray of $Y^k_{ii}$ contains objects which have arrows going in to it.

For the next statement, whose proof is deferred to Lemma~\ref{lem:appendix-mouth}, recall that the Serre functor is given by suspension and AR translation: $\SSS = \Sigma\tau$. Also, rays and corays commute with these three functors.

\begin{lemma} \label{lem:mouth_hammocks}
Let $A \in \ind{\Db(\LLambda)}$ with $r>1$ and let $i,k\in \IZ$, $0\leq k<r$. Then
\[ \begin{array}{l@{\:}ll}
  \Hom(X^k_{ii}, A) & = \kk & \text{if } A \in \rayfrom{X^k_{ii}}   \cup \corayto{\SSS X^k_{ii}}     \cup \raythrough{Z^k_{ii}} ,\\[0.5ex]
  \Hom(Y^k_{ii}, A) & = \kk & \text{if } A \in \corayfrom{Y^k_{ii}} \cup \rayto{\SSS Y^k_{ii}}       \cup \coraythrough{Z^k_{ii}} , \\[0.5ex]
 \Hom(A, X^k_{ii})  & = \kk & \text{if } A \in \corayto{X^k_{ii}} \cup \rayfrom{\SSS\inv X^k_{ii}} \cup \raythrough{\SSS\inv Z^k_{ii}} ,\\[0.5ex]
 \Hom(A, Y^k_{ii})  & = \kk & \text{if } A \in \rayto{Y^k_{ii}}   \cup \corayfrom{\SSS\inv Y^k_{ii}} \cup \coraythrough{\SSS\inv Z^k_{ii}}
\end{array} \]
and in all other cases the Hom spaces are zero. For $r=1$ the Hom-spaces are as above, except $\Hom(X^0_{ii},X^0_{i,i+m}) = \kk^2$.
\end{lemma}

\subsection{Hom-hammocks for objects in $\cX$ components} \label{sub:hammocks_X}

Assume $A=X^k_{ij}\in\ind{\cX^k}$. In order to describe the various Hom-hammocks conveniently, we set
\begin{enumerate}[leftmargin=1em,itemsep=0.8ex,label = {}]
\item $A_0 \coloneqq X^k_{jj}$ to be the intersection of the coray through $A$ with the mouth of $\cX^k$, and
\item $_0 A \coloneqq X^k_{ii}$ to be the intersection of the ray through $A$ with the mouth of $\cX^k$.
\end{enumerate}
By definition, $A_0$ and $_0 A$ have height $0$. If $A$ sits at the mouth, then $A = A_0 = {}_0A$.

We now write down some standard triangles involving the objects $_0 A$, $A_0$ and $A$. The following lemma is completely general and holds in any $\IZ A_\infty$ component of the AR quiver of a Krull--Schmidt triangulated category --- we use the notation introduced above for the $\cX$ components of $\Db(\LLambda)$.

\begin{lemma} \label{lem:triangles}
Let $A$ be an indecomposable object of height $\hi{A} \geq 1$ in a $\IZ A_\infty$ component of a Krull--Schmidt triangulated category. Let
 $A' \xxrightarrow{u} A \oplus C \xxrightarrow{v} A'' \to \Sigma A'$
be the AR triangle with $A$ at its apex; assuming $C = 0$ if $\hi{A}=1$. Then there are triangles
\[
\trilabels{_0 A}{A}{A''}{}{v''}{} \quad\text{and}\quad \trilabels{A'}{A}{A_0}{u'}{}{} 
\]
where $u'$ and $v''$ are induced by $u$ and $v$, respectively.
\end{lemma}

\begin{proof}
By Lemma~\ref{lem:mouth_hammocks} the composition, ${}_0A \to A$, of irreducible maps along a ray is non-zero. Likewise the composition, $A \to A_0$, of irreducible maps along a coray is non-zero.

We proceed by induction on $\hi{A}$. If $\hi{A} = 1$, then both triangles coincide with the AR triangle  ${}_0A \to A \to A_0 \to \Sigma {}_0A$; in particular, $A' = {}_0 A$ and $A'' = A_0$. 

Assume $\hi{A} > 1$. We shall show the existence of one triangle, the other one is dual. Consider the AR triangle together with the split triangle $A \to A \oplus C \xxrightarrow{0} C \to \Sigma A$. These triangles fit into the following commutative diagram arising from the octahedral axiom.
\[ \xymatrix@C=+4em{
                                            & A \ar@{=}[r] \ar[d]_-{\colmat{1}{0}}                        & A \ar[d]^-{v'} \\
A' \ar[r]_-{u=\colmat{u'}{u''}}  \ar@{=}[d] & A \oplus C \ar[r]^-{v=(v'~v'')} \ar[d]^-{(0~1)} & A'' \ar[d]    \\
A' \ar[r]_-{u''}                            & C \ar[r]                                                    & D
} \]
Since $\hi{A'} = \hi{A} - 1$, by induction there is a triangle ${}_0 A' \to  A' \xxrightarrow{u''} C \to \Sigma({}_0A)$. Thus $D = \Sigma (_0 A')$. From $_0 A' = {}_0A$ we get the desired triangle.
\end{proof}

We introduce notation for line segments in the AR quiver: given two indecomposable objects $A,B\in\Db(\LLambda)$ which lie on a ray or coray (so in particular sit in the same component), then the finite set consisting of these two objects and all indecomposables lying between them on the (co)ray is denoted by $\overline{AB}$.
Finally, we recall our convention that $\cX^r=\cX^0$ and note that $_0(\SSS A) = \Sigma\tau({}_0A)$.

\begin{lemma} \label{lem:raysfromA}
Consider $\Db(\LLambda)$ with $r > 1$. If $A \in \ind{\cX} \cup \ind{\cY}$ then for each indecomposable object $B \in \rayfrom{\overline{AA_0}}$ we have $\Hom(A,B) \neq 0$.
\end{lemma}

\noindent
Note that we shall treat the case $r=1$ in Proposition~\ref{prop:r=1_intersection} below; we continue to use the notation for the $\cX$ components, however, the argument applies also to the $\cY$ components.

\begin{proof}
First observe that any indecomposable object $B$ lying in an $\cX$ or $\cY$ component admits morphisms to precisely \emph{two} objects on the mouth, and precisely \emph{one} object on the mouth of the \emph{same} component, since $B$ lies in precisely one ray and one coray.

Let $A$ be an indecomposable object in an $\cX$ or $\cY$ component. If $B \in \rayfrom{A}$ or $B \in \overline{AA_0}$ or $B \in \rayfrom{A_0}$, then $\Hom(A,B) \neq 0$, using Serre duality for the third statement.
Let $B \in \rayfrom{\overline{AA_0}}$; the rays and corays of $\rayfrom{A} \cup \overline{AA_0} \cup \rayfrom{A_0}$ are indicated in the left-hand sketch in Figure~\ref{fig:hammocks}. Consider the following part of the AR quiver of $\sD$: 
\[
\xymatrix@!=1pt{
                        & &                     & C \ar[dr]  &           &                  &    \\
                        & &                     &            & B \ar[dr] &                  &    \\
                        & &                     & B' \ar[ur] &           & B'' \ar@{.>}[dr] &    \\
{}_0 C \ar@{.>}[uuurrr] & & {}_0 B \ar@{.>}[ur] &            &           &                  & B_0
}
\]
where $C$ is one irreducible morphism closer to $\rayfrom{A}$, $B'$ one closer to $\overline{AA_0}$ and $B''$ one closer to $\rayfrom{A_0}$, if $B$ is in the interior of the region $\rayfrom{\overline{AA_0}}$.
Note that the triangles
\[
\tri{{}_0 C}{C}{C''}, \quad \tri{{}_0 B}{B}{B''} \quad\text{and}\quad \tri{B'}{B}{B_0}
\]
are those from Lemma~\ref{lem:triangles} and $C'' = B$. Furthermore, since $\Sigma$ is an autoequivalence, any (co)suspension of ${}_0 C$, ${}_0 B$ and $B_0$ must also lie on the mouth.

The idea is to proceed by induction up each ray of the hammock starting with the lowest ray, $\rayfrom{A_0}$. By the observation above if $B \in \rayfrom{A_0} \cup \overline{AA_0}$ then $\Hom(A,B) \neq 0$ and we are done. By induction, we may assume $B \notin \rayfrom{A_0} \cup \overline{AA_0}$ and that $\Hom(A,B') \neq 0$ and $\Hom(A,B'') \neq 0$.

Since ${}_0 B \neq A_0 \neq B_0$, we have $\Hom(A,{}_0 B) = \Hom(A,B_0) = 0$.
Applying $\Hom(A,-)$ to the triangles involving $B'$ and $B''$ above produces long exact sequences in which the vanishing of one of $\Hom(A,\Sigma({}_0 B))$ and $\Hom(A,\Sigma\inv B_0)$ is enough to give $\Hom(A,B) \neq 0$. 

However, in the case $r = 2$ it may happen that $\Sigma({}_0 B )= \Sigma\inv B_0 = \SSS({}_0 A)$ and by Serre duality $\Hom(A, \Sigma({}_0 B)) = \Hom(A, \Sigma\inv B_0) \neq 0$. In this case, starting with the induction from the topmost ray, $\rayfrom{A}$, instead will give us that $\Hom(A,C) \neq 0$. Now we only require the vanishing of $\Hom(A,\Sigma({}_0 C))$ to give us $\Hom(A,B) \neq 0$. However, we have $\Sigma({}_0 C) \neq \Sigma({}_0 B) = \SSS({}_0 A)$. Since $A$ admits morphisms only to the objects $A_0$ and $\SSS({}_0 A)$ on the mouth of a $\IZ A_\infty$ component, we get $\Hom(A,\Sigma({}_0 C)) = 0$. We can now resume the standard induction.
\end{proof}

\begin{proposition}[Hammocks $\Hom(\cX^k,\blank)$] \label{prop:X-hammocks}
Let $A=X^k_{ij}\in\ind{\cX^k}$ and assume $r>1$. \newline
For any indecomposable object $B\in\ind{\Db(\Lambda)}$ the following cases apply:

\begintabularhammock
$B\in\cX^k$:    & then $\Hom(A,B)\neq0 \iff B \in \rayfrom{\overline{AA_0}}$; \\[0.2ex]
$B\in\cX^{k+1}$: & then $\Hom(A,B)\neq0 \iff B \in \corayto{\overline{{}_0(\SSS A),\SSS A}}$; \\
$B\in\cZ^k$:    & then $\Hom(A,B)\neq0 \iff B \in \raythrough{\overline{Z^k_{ii}Z^k_{ji}}}$
\end{tabular}

\noindent
and $\Hom(A,B)=0$ for all other $B\in\ind{\Db(\Lambda)}$.
\newline
For $r=1$, these results still hold, except that the $\cX$-clauses are replaced by

\begintabularhammock
$B\in\cX^0$: & then $\Hom(A,B)\neq0 \iff B \in \rayfrom{\overline{AA_0}} \cup \corayto{\overline{_0(\tau^{-m}A),\tau^{-m}A}}$.
\end{tabular}
\end{proposition}

\begin{figure}
  \parbox{\textwidth}{
    \parbox{0.3\textwidth}{ \resizebox{0.3\textwidth}{!}{
      \begin{tikzpicture} 
        \Homhammocktemplate{-1}
        \fill (2,3) circle (0.16); \fill (3,2) circle (0.16); \fill (4,1) circle (0.16); \fill (5,0) circle (0.16);
        \draw (2,3) -- (3,2); \draw (3,2) -- (4,1); \draw (4,1) -- (5,0);
        \fill (6,1) circle (0.16); \fill (7,2) circle (0.16); \fill (8,3) circle (0.16);
        \draw (5,0) -- (6,1); \draw (6,1) -- (7,2); \draw (7,2) -- (8,3); \draw (8,3) -- (9,4);
        \fill (3,4) circle (0.16); \fill (4,5) circle (0.16); \fill (5,6) circle (0.16); \fill (6,7) circle (0.16); \fill (7,8) circle (0.16);
        \draw (2,3) -- (3,4); \draw (3,4) -- (4,5); \draw (4,5) -- (5,6); \draw (5,6) -- (6,7); \draw (6,7) -- (7,8); \draw (7,8) -- (8,9);
        \node at (1.4,3) {\Hht{A}}; \node at (5.7,0) {\Hht{A_0}};
        \begin{scope}[on background layer]
          \draw[fill, gray!20, line width = 0.5cm] (2,3) -- (5,0) -- (9,4) -- (9,9) -- (8,9) -- cycle; 
        \end{scope}
      \end{tikzpicture}
      }  
      \centerline{$\rayfrom{\overline{AA_0}} \subset \cX^0$}
    }
  \hfill
  \parbox{0.3\textwidth}{ \resizebox{0.3\textwidth}{!}{
    \begin{tikzpicture} 
      \Homhammocktemplate{-1}
      \fill (3,0) circle (0.16); \fill (4,1) circle (0.16); \fill (5,2) circle (0.16); \fill (6,3) circle (0.16);
      \fill[gray!70] (8,3) circle (0.16);
      \draw (3,0) -- (4,1); \draw (4,1) -- (5,2); \draw (5,2) -- (6,3);
      \node at (3.9,0) {\Hht{{}_0 \SSS A}}; \node at (6.8,3) {\Hht{\SSS A}}; \node at (8.7,3) {\Hht{\Sigma A}};
      \begin{scope}[on background layer]
        \draw[fill, gray!20, line width = 0.5cm] (3,0) -- (6,3) -- (0,9) -- (0,3) -- cycle; 
      \end{scope}
    \end{tikzpicture}
    }  
    \centerline{$\corayto{\overline{{}_0\SSS A,\SSS A}} \subset \cX^1$}
  }
  \hfill
  \parbox{0.3\textwidth}{ \resizebox{0.3\textwidth}{!}{ 
    \begin{tikzpicture} 
      \Homhammocktemplate{0}
      \fill (1,4) circle (0.16); \fill (2,3) circle (0.16); \fill (3,2) circle (0.16); \fill (4,1) circle (0.16);
      \fill[gray!70] (7,4) circle (0.16);
      \draw (1,4) -- (2,3); \draw (2,3) -- (3,2); \draw (3,2) -- (4,1);
      \node at (0.8,4.8) {\Hht{Z^0_{1,1}}}; \node at (4.2,0.2) {\Hht{Z^0_{4,1}}}; \node at (8.0,4) {\Hht{Z^0_{4,4}}};
      \begin{scope}[on background layer]
        \draw[fill, gray!20, line width = 0.5cm] (0,0) -- (0,3) -- (6,9) -- (9,9) -- (9,6) -- (3,0) -- cycle; 
      \end{scope}
    \end{tikzpicture}
    }  
    \centerline{$\raythrough{\overline{Z^0_{1,1}Z^0_{4,1}}} \subset \cZ^0$}
  }
  }
  \caption{Hom hammocks $\Hom(A,\blank) = \Hom(\blank,\SSS A)^*\neq 0$ for $A=X^0_{1,4}$. 
           \newline
           $A_0=X^0_{4,4}, \quad \Sigma A=X^1_{1,4} ~ (\text{if } r\geq2), \quad \SSS A = \Sigma\tau A = X^1_{0,3}, \quad
           {}_0\SSS A = X^1_{0,0}$
  } \label{fig:hammocks}
\end{figure}

\begin{proof}
The main tool in the proof of this, and the following propositions, will be induction on the height of $A$ --- the induction base step is proved in Lemma~\ref{lem:mouth_hammocks} which gives the hammocks for indecomposables of height $0$. We give a careful exposition for the first claim, and for $r>1$. The $r=1$ case will be treated in Proposition~\ref{prop:r=1_intersection}.

\Case{$B\in\cX^k$} For any indecomposable object $A\in\cX^k$, write $R(A)$ for the subset of $\cX^k$ specified in the statement, i.e.\ bounded by the rays out of $A$ and $A_0$, and the line segment $\overline{AA_0}$.
The existence of non-zero homomorphisms $A\to B$ for objects $B\in R(A)$ follows directly from Lemma~\ref{lem:raysfromA}.

For the vanishing statement, we proceed by induction on the height of $A$. If $A$ sits on the mouth of $\cX^k$, then Lemma~\ref{lem:mouth_hammocks} states indeed that the $\Hom(A,B)\neq0$ if and only if $B$ is in the ray of $A$. Note that $R(A)$ is precisely $\rayfrom{A}$ in this case.

Now let $A\in\cX^k$ be any object of height $h\coloneqq\hi{A}>0$. We consider the diamond in the AR mesh which has $A$ as the top vertex, and the corresponding AR triangle
  $\tri{A'}{A\oplus C}{A''}$,
where $\hi{A'} = \hi{A''} = h-1$ and $\hi{C} = h-2$. (If $h=1$, we are in the degenerate case with $C=0$.)
It is clear from the definitions that $A_0=A''_0$, $A'_0=C_0$ and there are inclusions $R(A'')\subset R(A)\subset R(A')\cup R(A'')$. We start with an object $B\in\cX^k$ such that $B\notin R(A')\cup R(A'')$. By the induction hypothesis, we know that $R(A')$, $R(C)$ and $R(A'')$ are the Hom-hammocks in $\cX^k$ for $A'$, $C$, $A''$, respectively. Since $B$ is contained in none of them, we see that $\Hom(A',B)=\Hom(C,B)=\Hom(A'',B)=0$. Applying $\Hom(\blank,B)$ to the given AR triangle shows $\Hom(A,B)=0$.

It remains to show that $\Hom(A,D)=0$ for objects $D \in (R(A')\cup R(A'')) \backslash R(A)$ which can be seen to be the line segment $\overline{A'A'_0}$. Again we work up from the mouth: $\Hom(A,A'_0)= 0$ and $\Hom(A,\tau A'_0)=0$ by Lemma~\ref{lem:mouth_hammocks}, as before. The extension $D_1$ given by $\tri{\tau A'_0}{D_1}{A'_0}{}{}{}$ is the indecomposable object of height 1 on $\overline{A'A'_0}$. Applying $\Hom(A,\blank)$ to this triangle, we find $\Hom(A,D_1)=0$, as required. The same reasoning works for the objects of heights $2,\ldots,h-1$ on the segment.

\Case{$B\in\cX^{k+1}$} We start by showing the existence of non-zero homomorphisms to indecomposable objects in the desired region. For any $B$ in this region, it follows directly from the dual of Lemma~\ref{lem:raysfromA} that there is a non-zero homomorphism from $B$ to $\SSS A$. However, by Serre duality we see that $\Hom(A,B) = \Hom(B, \SSS A)^* \neq0$, as required.
The statement that $\Hom(A,B) = 0$ for all other $B\in\cX^{k+1}$ can be proved by an induction argument which is analogous to the one given in the first case above.

\Case{$B\in\cZ^k$} For any indecomposable object $A=X^k_{ij}\in\cX^k$, write $V(A)$ for the region in $\cZ^k$ specified in the statement, i.e.\ the region bounded by the rays through $Z^k_{ii}$ and $Z^k_{ji}$.
We start by proving that $\Hom(A,B) \neq 0$ for $B \in V(A)$. The first chain of morphisms in Properties~\ref{properties}(5), implies that $\Hom(A,B)\neq0$ for any $B\in\raythrough{Z^k_{ii}}$.
For any other $B'=Z^{k}_{i+s,t} \in V(A)$, so $t \in \IZ$ and $s \in \{1,\ldots,\hi{A} = j-i\}$, we consider the special triangle $\tri{X^k_{i,i+s-1}}{B}{B'}{}{}{}$ from Properties~\ref{properties}(4), where $B=Z^k_{it}\in\raythrough{Z^k_{ii}}$.
Applying $\Hom(A,\blank)$ leaves us with the exact sequence
\[ \Hom(A,X^k_{i,i+s-1}) \to \Hom(A,B) \to \Hom(A,B') \to \Hom(A,\Sigma X^k_{i,i+s-1}) . \]
By looking at the Hom-hammocks in the $\cX$-components that we already know,  we see that the left-hand term vanishes as $X^k_{i,i+s-1}$ is on the same ray as $A$ but has strictly lower height.
Similarly, we observe that the right-hand term of the sequence vanishes:
$0 = \Hom(X^k_{i,i+s-1}, \tau A) = \Hom(A,\Sigma X^k_{i,i+s-1})$. Hence
$\Hom(A,B') = \Hom(A,B) \neq 0$.

For the Hom-vanishing part of the statement, we again use induction on the height $h\coloneqq\hi{A}\geq 0$. For $h=0$, Lemma~\ref{lem:mouth_hammocks} gives $V(A)=\raythrough{Z^k_{ii}}$. For $h>0$, as before we consider the AR mesh which has $A$ as its top vertex:
 $\tri{A'}{A\oplus C}{A''}{}{}{}$.
For any $Z\in\ind{\cZ^k}$, we apply $\Hom(\blank,Z)$ to this triangle and find that $\Hom(A,Z)\neq0$ implies $\Hom(A',Z)\neq0$ or $\Hom(A'',Z)\neq0$. Therefore $\Hom(A,B) = 0$ for all $B \notin V(A')\cup V(A'') = V(A)$, where the final equality is clear from the definitions.

\RemainingCases These comprise vanishing statements for entire AR components, namely $\Hom(\cX^k,\cX^j)=0$ for $j\neq k,k+1$, and $\Hom(\cX^k,\cY^j)=0$ for any $j$, and $\Hom(\cX^k,\cZ^j)=0$ for $j\neq k$. All of those follow at once from Lemma~\ref{lem:mouth_hammocks}: with no non-zero maps from $A$ to the mouths of the specified components of type $\cX$ and $\cY$, Hom vanishing can be seen using induction on height and considering a square in the AR mesh. The vanishing to the $\cZ^k$ components with $k\neq j$ follows similarly.
\end{proof}

\subsection{Hom-hammocks for objects in $\cY$ components} \label{sub:hammocks_Y}

Assume $A=Y^k_{ij}\in\ind{\cY^k}$. This case is similar to the one above. Put
\begin{enumerate}[leftmargin=1em,itemsep=0.8ex,label = {}]
\item $^0\! A \coloneqq Y^k_{ii}$ to be the intersection of the coray through $A$ with the mouth of $\cY^k$, and
\item $A^0    \coloneqq Y^k_{jj}$ to be the intersection of the ray through $A$ with the mouth of $\cY^k$.
\end{enumerate}

\begin{proposition}[Hammocks $\Hom(\cY^k,\blank)$] \label{prop:Y-hammocks}
Let $A=Y^k_{ij}\in\ind{\cY^k}$ and assume $r>1$. \newline
For any indecomposable object $B\in\ind{\Db(\Lambda)}$ the following cases apply:

\begintabularhammock
$B\in\cY^k$:     & then $\Hom(A,B)\neq0 \iff B \in \corayfrom{\overline{AA^0}}$; \\[0.2ex]
$B\in\cY^{k+1}$: & then $\Hom(A,B)\neq0 \iff B \in \rayto{\overline{{}^0(\SSS A),\SSS A}}$; \\
$B\in\cZ^k$:     & then $\Hom(A,B)\neq0 \iff B \in \coraythrough{\overline{Z^k_{ii}Z^k_{ij}}}$
\end{tabular}

\noindent
and $\Hom(A,B)=0$ for all other $B\in\ind{\Db(\Lambda)}$.
\newline
For $r=1$, these results still hold, except that the $\cY$-clauses are replaced by

\begintabularhammock
$B\in\cY^0$: & then $\Hom(A,B)\neq0 \iff B \in \corayfrom{\overline{AA^0}} \cup \rayto{\overline{^0(\tau^n A),\tau^n A}}$.
\end{tabular}
\end{proposition}

\begin{proof}
These statements are analogous to those of Proposition~\ref{prop:X-hammocks}.
%
\end{proof}

\subsection{Hom-hammocks for objects in $\cZ$ components}

Let $A = Z^k_{ij}\in\ind{\cZ^k}$. By Lemma~\ref{lem:mouth_hammocks} we know that the following objects are well defined:
\begin{enumerate}[leftmargin=1em,itemsep=0.8ex,label = {}]
  \item $A_0 \coloneqq $ the unique object at the mouth of an $\cX$ component for which $\Hom(A, A_0) \neq 0$,
  \item $A^0  \coloneqq$ the unique object at the mouth of a $\cY$ component for which $\Hom(A, A^0) \neq 0$.
\end{enumerate}
In fact, $ A_0 \in \cX^{k+1}$ and $A^0 \in \cY^{k+1}$.

\begin{proposition}[Hammocks $\Hom(\cZ^k,\blank)$] \label{prop:Z-hammocks}
Let $A=Z^k_{ij}\in\ind{\cZ^k}$ and assume $r>1$. \newline
For any indecomposable object $B\in\ind{\Db(\Lambda)}$ the following cases apply:

\begintabularhammock
$B\in\cX^{k+1}$: & then $\Hom(A,B)\neq0 \iff B \in \rayfrom{\corayto{A_0}}$; \\
$B\in\cY^{k+1}$: & then $\Hom(A,B)\neq0 \iff B \in \rayto{\corayfrom{A^0}}$; \\
$B\in\cZ^k$:    & then $\Hom(A,B)\neq0 \iff B \in \rayfrom{\corayfrom{A}}$;  \\
$B\in\cZ^{k+1}$: & then $\Hom(A,B)\neq0 \iff B \in \rayto{\corayto{\SSS A}}$
\end{tabular}

\noindent
and $\Hom(A,B)=0$ for all other $B\in\ind{\Db(\Lambda)}$.
\newline
For $r=1$, these results still hold, with the $\cZ$-clauses replaced by

\begintabularhammock
$B\in\cZ^0$: & then $\Hom(A,B)\neq0 \iff B \in \rayfrom{\corayfrom{A}} \cup \rayto{\corayto{\SSS A}}$.
\end{tabular}
\end{proposition}

\begin{figure}
  \parbox{\textwidth}{
    \parbox{0.3\textwidth}{ \resizebox{0.3\textwidth}{!}{
      \begin{tikzpicture} 
        \Homhammocktemplate{-1}
        \fill (1,4) circle (0.16); \fill (2,3) circle (0.16); \fill (3,2) circle (0.16); \fill (4,1) circle (0.16);
        \fill (5,0) circle (0.16); \fill (6,1) circle (0.16); \fill (7,2) circle (0.16); \fill (8,3) circle (0.16);
        \draw (0,5) -- (1,4); \draw (1,4) -- (2,3); \draw (2,3) -- (3,2); \draw (3,2) -- (4,1); \draw (4,1) -- (5,0);
        \draw (5,0) -- (6,1); \draw (6,1) -- (7,2); \draw (7,2) -- (8,3); \draw (8,3) -- (9,4);
        \node at (5,0.7) {\Hht{A_0}};
        \begin{scope}[on background layer]
          \draw[fill, gray!20, line width = 0.5cm] (0,9) -- (0,5) -- (5,0) -- (9,4) -- (9,9) -- cycle; 
        \end{scope}
      \end{tikzpicture}
      }  
      \centerline{$\rayfrom{\corayto{A_0}} \subset \cX^1$}
    }
  \hfill
  \parbox{0.3\textwidth}{ \resizebox{0.3\textwidth}{!}{
    \begin{tikzpicture} 
      \Homhammocktemplate{0}
      \fill (4,5) circle (0.16); \fill (5,6) circle (0.16); \fill (6,7) circle (0.16); \fill (7,8) circle (0.16);
                                 \fill (5,4) circle (0.16); \fill (6,3) circle (0.16); \fill (7,2) circle (0.16);  \fill (8,1) circle (0.16);
      \fill[gray!70] (2,5) circle (0.16);
      \draw (4,5) -- (5,6); \draw (5,6) -- (6,7); \draw (6,7) -- (7,8); \draw (7,8) -- (8,9);
      \draw (4,5) -- (5,4); \draw (5,4) -- (6,3); \draw (6,3) -- (7,2); \draw (7,2) -- (8,1); \draw (8,1) -- (9,0);
      \node at (3.4,5) {\Hht{A}}; \node at (1.2,5) {\Hht{\tau A}};
      \begin{scope}[on background layer]
        \draw[fill, gray!20, line width = 0.5cm] (4,5) -- (9,0) -- (9,9) -- (8,9) -- cycle; 
      \end{scope}
    \end{tikzpicture}
    }  
    \centerline{$\rayfrom{\corayfrom{A}} \subset \cZ^0$}
  }
  \hfill
  \parbox{0.3\textwidth}{ \resizebox{0.3\textwidth}{!}{ 
    \begin{tikzpicture} 
      \Homhammocktemplate{0}
      \fill (1,8) circle (0.16); \fill (2,7) circle (0.16); \fill (3,6) circle (0.16); \fill (4,5) circle (0.16);
      \fill (1,2) circle (0.16); \fill (2,3) circle (0.16); \fill (3,4) circle (0.16);
      \fill[gray!70] (4,7) circle (0.16);
      \draw (0,9) -- (1,8); \draw (1,8) -- (2,7); \draw (2,7) -- (3,6); \draw (3,6) -- (4,5);
      \draw (0,1) -- (1,2); \draw (1,2) -- (2,3); \draw (2,3) -- (3,4); \draw (3,4) -- (4,5);
      \node at (4.7,5) {\Hht{\SSS A}}; \node at (6.7,5) {\Hht{\Sigma A}};
      \begin{scope}[on background layer]
        \draw[fill, gray!20, line width = 0.5cm] (0,1) -- (4,5) -- (0,9) -- cycle; 
      \end{scope}
    \end{tikzpicture}
    }  
    \centerline{$\rayto{\corayto{\SSS A}} \subset \cZ^1$}
  }
  }
  \caption{Hammocks $\Hom(A,\blank) = \Hom(\blank,\SSS A)^*\neq 0$ for $A\in\ind{\cZ^0}$.\newline
           The remaining hammock $\rayfrom{\corayfrom{A^0}}\subset\cY^1$ is not shown.
  }
\end{figure}

\begin{proof}
The cases $B\in\cX^{k+1}$ and $B\in\cY^{k+1}$ follow by Serre duality from Proposition~\ref{prop:X-hammocks} and Proposition~\ref{prop:Y-hammocks}, respectively.

Thus let $B=Z^l_{ab}\in\cZ^l$ be an indecomposable object in a $\cZ$ component. There are two special distinguished triangles associated with $B$; see Properties~\ref{properties}(4):
\[
\parbox{0.4\textwidth}{
   \xymatrix@R=2ex{
   {}_0B \ar@{=}[d] \ar[r] & B \ar@{=}[d] \ar[r] & B' \ar@{=}[d] \ar[r] & \Sigma \, {}_0B \ar@{=}[d] \\
   X^l_{aa}          \ar[r] & Z^l_{ab}      \ar[r] & Z^l_{a+1,b}    \ar[r] & \Sigma X^l_{aa}
}} \quad\text{and }\quad
\parbox{0.4\textwidth}{
   \xymatrix@R=2ex{
   {}^{0\!}B \ar@{=}[d] \ar[r] & B \ar@{=}[d] \ar[r] & B'' \ar@{=}[d] \ar[r] & \Sigma \, {}^{0\!}B \ar@{=}[d] \\
   Y^l_{bb}          \ar[r] & Z^l_{ab}      \ar[r] & Z^l_{a,b+1}     \ar[r] & \Sigma Y^l_{bb}
}}
\]
where ${}_0B=X^l_{aa}$ is the unique object at the mouth of a $\cX$ component with $\Hom({}_0B,B)\neq0$ and similarly ${}^0B=Y^l_{bb}$ is unique at a $\cY$ mouth with $\Hom({}^0B,B)\neq0$. We get two exact sequences by applying $\Hom(A,\blank)$:
\begin{align*}
   \Hom(A,{}_0B) \too \Hom(A,B) \too \Hom(A,B')  \too \Hom(A,\Sigma \, {}_0B) ,  \\
   \Hom(A,{}^{0\!}B) \too \Hom(A,B) \too \Hom(A,B'') \too \Hom(A,\Sigma \, {}^{0\!}B) .
\end{align*}

\Case{$l\neq k, k+1$} In this case $\Hom(A,B)=\Hom(A,B')=\Hom(A,B'')$ follows from the above triangles via these exact sequences and Lemma~\ref{lem:mouth_hammocks}. But this implies $\Hom(A,B)=\Hom(A,Z)$ for all $Z \in\cZ^l$ and in particular $\Hom(A,B)=\Hom(A,\Sigma^{cr}B)$ for all $c\in\IZ$. It follows that $\Hom(A,B)=0$ as $\LLambda$ has finite global dimension.

\Case{$l=k$} Again, we first show that the dimension function $\hom(A,\blank)$ is constant on certain regions of $\cZ^k$. In particular, we have
\begin{align}
& \Hom(A,B) = \Hom(A,B')  \text{ for } B\notin \raythrough{\SSS A}\cup \raythrough{\tau A}; \label{first-triangle}\\
& \Hom(A,B) = \Hom(A,B'') \text{ for } B\notin \coraythrough{\SSS A}\cup \coraythrough{\tau A}. \label{second-triangle}
\end{align}
Half of the first equality follows through the chain of equivalences
\[ \Hom(A,{}_0B) \neq 0 \iff A \in \raythrough{\SSS\inv Z^k_{aa}}
                        \iff A \in \raythrough{\SSS\inv B}
                        \iff B \in \raythrough{\SSS A} . \]
Likewise one obtains $\Hom(A,\Sigma{}_0B) \neq 0 \iff B \in \raythrough{\tau A}$, giving the first equality. Using the other triangle, the second equality is analogous.

The component $\cZ^k$ is divided by $\raythrough{\tau A}$ and $\coraythrough{\tau A}$ into four regions:

\noindent
\parbox{0.7\textwidth}{
\begin{tabular}{@{} p{0.05\textwidth} @{} p{0.65\textwidth} @{}}
$\ku \colon$ & The upwards-open region including $\rayfrom{\tau A} \setminus \{\tau A\}$
               but excluding $\corayto{\tau A}$; \\
$\kl \colon$ & The left-open region including $\rayto{\tau A} \cup \corayto{\tau A}$; \\
$\kd \colon$ & The downwards-open region including $\corayfrom{\tau A}\setminus \{\tau A\}$
               but excluding $\rayto{\tau A}$; \\
$\kr \colon$ & The right-open region excluding $\raythrough{\tau A} \cup \coraythrough{\tau A}$.
\end{tabular}
}
\hfill
\parbox{0.20\textwidth}{ \resizebox{0.16\textwidth}{!}{\begin{tikzpicture}

\pgfmathsetmacro{\l}{1.5}    
\pgfmathsetmacro{\s}{1.8}    

\node    at (0,0)     {$\bullet$};
\node    at (1,0)     {$\bullet$};
\node    at (-0.6,0)  {$\tau A$};
\node    at ( 0.7,0)  {$A$};

\node    at ( 0, \l)    {\large$\ku$};
\node    at ( 0,-\l)    {\large$\kd$};
\node    at ( \l,-0.5)  {\large$\kr$};
\node    at (-\l, 0.5)  {\large$\kl$};

\draw (-\s,-\s) -- (\s, \s) ;
\draw (-\s, \s) -- (\s,-\s) ;

\end{tikzpicture}} }

Using \eqref{first-triangle} above coupled with the fact that $\ku$ contains infinitely many objects $\Sigma^{-rc} A$ with $c \in \IN$, shows by the finite global dimension of $\LLambda$ that no objects in $\ku$ admit non-trivial morphisms from $A$. Using \eqref{second-triangle} and analogous reasoning shows that no objects in $\kd$ admit non-trivial morphisms from $A$. Non-existence of non-trivial morphisms from $A$ to objects in $\kl$ follows as soon as $\Hom(A,\tau A) = \Hom^1(A,A)=0$ by using \eqref{second-triangle} above. The existence of the stalk complex of a projective module in the $\cZ$ component, Lemma~\ref{lem:proj-position}, coupled with the transitivity of the action of the automorphism group of $\Db(\LLambda)$ on the $\cZ$ component, which is proved in Section~\ref{sec:auteq} using only Lemma~\ref{lem:mouth_hammocks}, shows that $\Hom^1(A,A)=0$ for all $A\in\cZ$.

Finally, $\kr = \corayfrom{\rayfrom{A}}$ is the non-vanishing hammock simply by $\Hom(A,A)\neq0$ and using either \eqref{first-triangle} or \eqref{second-triangle}.

\Case{$l=k+1$} This is analogous to the previous case.
\end{proof}

\begin{remark}
In the case that $r>1$, Propositions~\ref{prop:X-hammocks}, \ref{prop:Y-hammocks} and \ref{prop:Z-hammocks} say that each component of the AR quiver of $\Db(\LLambda)$ is standard, i.e.\ that there are no morphisms in the infinite radical. Note that the components are not standard when $r=1$.
\end{remark}


\section{Twist functors from exceptional cycles} \label{sec:exceptional_cycles}

\noindent
In this purely categorical section, we consider an abstract source of auto\-equivalences coming from exceptional cycles. These generalise the tubular mutations from \cite{Meltzer} as well as spherical twists. In fact, a quite general and categorical construction has been given in \cite{Roosmalen}. However, for our purposes this is still a little bit too special, as the Serre functor will act with different degree shifts on the objects in our exceptional cycles. We also give a quick proof using spanning classes.

Let $\catD$ be a $\kk$-linear, Hom-finite algebraic triangulated category. Assume that $\catD$ has a Serre functor $\SSS$ and is indecomposable; see Appendix~\ref{app:category-properties} for these notions. Recall that an object $E\in\catD$ is called \emph{exceptional} if $\Hom^\bullet(E,E)=\kk\cdot\id_E$. For any object $A\in\catD$ we define the functor
\[ \FFF_A\colon \catD\to\catD, \quad X \mapsto \Hom^\bullet(A,X)\otimes A \]
and note that there is a canonical evaluation morphism $\FFF_A\to\id$ of functors. Also note that for two objects $A_1,A_2\in\catD$ there is a common evaluation morphism $\FFF_{A_1}\oplus\FFF_{A_2}\to\id$. In fact, for any sequence of objects $A_*=(A_1,\ldots,A_n)$, we define the associated \emph{twist functor} $\TTT_{A_*}$ as the cone of the evaluation morphism --- this gives a well-defined, exact functor by our assumption that $\sD$ is algebraic; see \cite[\S2.1]{HKP} for details:
\[ \FFF_{A_*} \to \id_\catD \to \TTT_{A_*} \to \Sigma\FFF_{A_*}
   \qquad\text{with}\qquad
   \FFF_{A_*} \coloneqq \FFF_{A_1} \oplus \cdots \oplus \FFF_{A_n} \]

These functors behave well under equivalences:

\begin{lemma} \label{lem:conjugated_twists}
Let $\varphi\colon \sD\isom\sD'$ be a triangle equivalence of algebraic $\kk$-linear triangulated categories induced from a dg functor, and let $A_*=(A_1,\ldots,A_n)$ be any sequence of objects. Then there are functor isomorphisms
 $\FFF_{\varphi(A_*)} = \varphi\FFF_{A_*}\varphi\inv$ and
 $\TTT_{\varphi(A_*)} = \varphi\TTT_{A_*}\varphi\inv$.
\end{lemma}

\begin{proof}
This follows the standard argument for spherical twists: For $\FFF_{A_*}$ we have
\[ \varphi\FFF_{A_*}\varphi\inv
 = \bigoplus_i\Hom^\bullet(A_i,\varphi\inv(\blank))\otimes\varphi(A_i)
 = \bigoplus_i\Hom^\bullet(\varphi(A_i),\blank)\otimes\varphi(A_i)
 = \FFF_{\varphi(A_*)} .\]
Conjugating the evaluation functor morphism $\FFF_{A_*}\to\id$ with $\varphi$, we find that $\varphi\TTT_{A_*}\varphi\inv$ is the cone of the conjugated evaluation functor morphism $\FFF_{\varphi(A_*)}\to\id$ which is the evaluation morphism for $\varphi(A_*)$. Hence that cone is $\TTT_{\varphi(A_*)}$.
\end{proof}

\begin{definition}
A sequence $(E_1,\ldots,E_n)$ of objects of $\catD$ is an \emph{exceptional $n$-cycle} if
\begin{enumerate}
\item every $E_i$ is an exceptional object,
\item there are integers $k_i$ such that $\SSS(E_i)\cong\Sigma^{k_i}(E_{i+1})$ for all $i$ (where $E_{n+1}\coloneqq E_1$),
\item $\Hom^\bullet(E_i,E_j)=0$ unless $j=i$ or $j=i+1$.
\end{enumerate}
\end{definition}

\noindent
This definition assumes $n\geq2$ but a single object $E$ should be considered an `exceptional 1-cycle' if $E$ is a spherical object, i.e.\ there is an integer $k$ with $\SSS(E)\cong\Sigma^k(E)$ and $\Hom^\bullet(E,E)=\kk\oplus\Sigma^{-k}\kk$. In this light, the above definition, and statement and proof of Theorem~\ref{thm:twist} are generalisations of the treatment of spherical objects and their twist functors as in \cite[\S8]{Huybrechts}.

In an exceptional cycle, the only non-trivial morphisms among the $E_i$ apart from the identities are given by $\alpha_i\colon E_i\to \Sigma^{k_i}E_{i+1}$. This explains the terminology: the subsequence $(E_1,\ldots,E_{n-1})$ is an honest exceptional sequence, but the full set $(E_1,\ldots,E_n)$ is not --- the morphism $\alpha_n\colon E_n\to\Sigma^{k_n}E_1$ prevents it from being one, and instead creates a cycle.

\begin{remark}
All objects in an exceptional $n$-cycle are fractional Calabi--Yau: since $\SSS(E_i)\cong \Sigma^{k_i}E_{i+1}$ for all $i$, applying the Serre functor $n$ times yields $\SSS^n(E_i)\cong\Sigma^k E_i$, where $k\coloneqq k_1+\cdots+k_n$. Thus the Calabi--Yau dimension of each object in the cycle is $k/n$.
\end{remark}

\begin{example}
We mention that this severely restricts the existence of exceptional $n$-cycles of geometric origin: Let $X$ be a smooth, projective variety over $\kk$ of dimension $d$ and let $\catD\coloneqq\Db(\mathsf{coh} X)$ be its bounded derived category. The Serre functor of $\catD$ is given by $\SSS(\blank) = \Sigma^d(\blank)\otimes\omega_X$ and in particular, is given by an autoequivalence of the standard heart followed by an iterated suspension. If $E_*$ is any exceptional $n$-cycle in $\catD$, we find $\SSS^n(E_i)=\Sigma^{dn} E_i\otimes\omega^n_X \cong \Sigma^k E_i$, hence $k=k_1+\cdots+k_n=dn$ and $E_i\otimes\omega_X^n\cong E_i$. If furthermore the exceptional $n$-cycle $E_*$ consists of sheaves, then this forces $k_i=d$ to be maximal for all $i$, as non-zero extensions among sheaves can only exist in degrees between 0 and $d$. However, $\SSS E_i=\Sigma^d E_i\otimes\omega_X\cong \Sigma^d E_{i+1}$ implies $E_{i+1}\cong E_i\otimes\omega_X$ for all $i$.

As an example, let $X$ be an Enriques surface. Its structure sheaf $\ko_X$ is exceptional, and the canonical bundle $\omega_X$ has minimal order 2. In particular, $(\ko_X,\omega_X)$ forms an exceptional 2-cycle and, by the next theorem, gives rise to an autoequivalence of $\Db(X)$.
\end{example}

\begin{theorem} \label{thm:twist}
Let $E_* = (E_1,\ldots,E_n)$ be an exceptional $n$-cycle in $\catD$. Then the twist functor $\TTT_{E_*}$ is an autoequivalence of $\catD$.
\end{theorem}

\begin{proof}
We define two classes of objects of $\catD$ by $\sE \coloneqq \{ \Sigma^l E_i \mid l\in\IZ, i=1,\ldots,n \}$ and $\Omega \coloneqq \sE \cup \sE\orth$.
Note that $\sE$ and hence $\Omega$ are closed under suspensions and cosuspensions.
It is a simple and standard fact that $\Omega$ is a spanning class for $\catD$, i.e.\ $\Omega\orth=0$ and ${}\orth\Omega=0$; the latter equality depends on the existence of a Serre functor for $\catD$. Note that spanning classes are often called `(weak) generating sets' in the literature.

\Step{1}
We start by computing $\TTT_{E_*}$ on the objects $E_i$ and the maps $\alpha_i$. For notational simplicity, we will treat $E_1$ and $\alpha_1\colon E_1\to\Sigma^{k_1}E_2$. It follows immediately from the definition of exceptional cycle that $\FF_{E_*}(E_1)=E_1\oplus\Sigma^{-k_n}E_n$. The cone of the evaluation morphism is easily seen to sit in the following triangle
\[ \xymatrix@C=5.5em{
         E_1\oplus\Sigma^{-k_n}E_n \ar[r]^-{\id\oplus\Sigma^{-k_n}\alpha_n} & 
         E_1 \ar[r]^-0 & \Sigma^{1-k_n}E_n \ar[r]^-{(-\Sigma^{1-k_n}\alpha_n,\id)^t} &
  \Sigma E_1\oplus\Sigma^{1-k_n}E_n ,
} \]
so that $\TTT_{E_*}(E_1) = \Sigma^{1-k_n}E_n$. The left-hand morphism has an obvious splitting, this implies the zero morphism in the middle. The third map is indeed the one specified above; this can be formally checked with the octahedral axiom, or one can use the vanishing of the composition of two adjacent maps in a triangle.

Likewise, we find $\FF_{E_*}(E_2)=\Sigma^{-k_1}E_1\oplus E_2$ and $\TTT_{E_*}(E_2) = \Sigma^{1-k_1}E_1$.
Now consider the following diagram of distinguished triangles, where the vertical maps are induced by $\alpha_1$:
\[ \xymatrix@C=5em@R=6ex{
   E_1\oplus\Sigma^{-k_n}E_n \ar[r]^-{\id\oplus\Sigma^{-k_n}\alpha_n}  \ar[d]^{\sqmat{\id}{0}{0}{0}}  &
   E_1                      \ar[r]^0                                 \ar[d]^{\alpha_1}            &
   \Sigma^{1-k_n}E_n         \ar[r]^-{(-\Sigma^{1-k_n}\alpha_n,\id)^t}  \ar[d]^{\TTT\:(\alpha_1)}     &
   \Sigma E_1 \oplus \Sigma^{1-k_n}E_n                                \ar[d]^{\sqmat{\id}{0}{0}{0}}
\\
   E_1\oplus\Sigma^{k_1}E_2  \ar[r]^-{\alpha_1\oplus\id}                            &
   \Sigma^{k_1}E_2           \ar[r]^0                                               &
   \Sigma E_1               \ar[r]^-{(\id,-\Sigma\alpha_1)^t}                      &
   \Sigma E_1 \oplus \Sigma^{1+k_1} E_2
} \]
Hence, the commutativity of the right-hand square forces $\TTT_{E_*}(\alpha_1) = -\Sigma^{1-k_n}\alpha_n$.

This also works if $n=2$ and $k_1=-k_2$ (with unchanged left-hand vertical arrow).

\Step{2}
The above computation shows that the functor $\TTT_{E_*}$ is fully faithful when restricted to $\sE$. It is also obvious from the construction of the twist that $\TTT_{E_*}$ is the identity when restricted to $\sE\orth$.

Let $E_i\in\sE$ and $X\in\sE\orth$. Then $\Hom^\bullet(E_i,X)=0$ and also
 $\Hom^\bullet(\TTT_{E_*}(E_i),\TTT_{E_*}(X))=\Hom^\bullet(\Sigma^{1-k_{i-1}}E_{i-1},X)=0$.
Finally, we use Serre duality and the defining property of $E_*$ to see that
\[ \Hom^\bullet(X,E_i) = \Hom^\bullet(X,\Sigma^{-k_{i-1}}\SSS(E_{i-1})) \cong
                        \Hom^\bullet(E_{i-1},\Sigma^{k_{i-1}}X)^* = 0 . \]
Combining all these statements, we deduce that $\TTT_{E_*}$ is fully faithful when restricted to the spanning class $\Omega$, hence bona fide fully faithful by general theory; see e.g.\ \cite[Proposition~1.49]{Huybrechts}. Note that $\TTT_{E_*}$ has left and rights adjoints as the identity and $\FF_{E_*}$ do.

\Step{3}
With $\TTT_{E_*}$ fully faithful, the defining property of Serre functors gives a canonical map of functors $\SSS\inv\TTT_{E_*}\SSS\to\TTT_{E_*}$ which can be spelled out in the following diagram:
\[ \xymatrix{
   \bigoplus_i \Hom^\bullet(E_i,\SSS(\blank))\otimes\SSS\inv(E_i) \ar[r] \ar[d] & \id \ar[r] \ar[d] & \SSS\inv\TTT_{E_*}\SSS \ar[d] \\
   \bigoplus_i \Hom^\bullet(E_i,\blank)      \otimes E_i          \ar[r]        & \id \ar[r]        & \TTT_{E_*}
} \]
It is easy to check that the left-hand vertical arrow is an isomorphism whenever we plug in objects from $\Omega$: both vector spaces are zero for objects from $\sE\orth$; for the top row, use
 $\Hom^\bullet(E_i,\SSS(\blank)) = \Hom^\bullet(\SSS\inv(E_i),\blank) = \Hom^\bullet(\Sigma^{-k_{i-1}}E_{i-1},\blank)$.
For objects $E_i$, again use $\SSS(E_i)\cong\Sigma^{k_i}E_{i+1}$.
Hence $\TTT_{E_*}$ commutes with the Serre functor on $\Omega$, and so by more general theory is essentially surjective; see \cite[Corollary~1.56]{Huybrechts}, this is the place where we need the assumption that $\catD$ is indecomposable.
\end{proof}

\begin{remark}
We point out that the twist $\TTT_{E_*}$ defined above is an instance of a spherical functor \cite{Anno-Logvinenko}, given by the following data:
\[ \begin{array}{rl @{\qquad} rl}
 S \colon & \Db(\kk^n) \to \catD, &
 (V^\bullet_1,\ldots,V^\bullet_n) & \mapsto V^\bullet_1\otimes_\kk E_1 \oplus\cdots \oplus V^\bullet_n\otimes_\kk E_n, \\[1ex]
 R \colon & \catD \to \Db(\kk^n), &
                 X                & \mapsto (\Hom^\bullet(E_1,X),\ldots,\Hom^\bullet(E_n,X))
\end{array} \]
where $\Db(\kk^n)=\bigoplus_n\Db(\kk)$ is a decomposable category. It is easy to see that $R$ is right adjoint to $S$ and that $\TTT_{E_*}$ coincides with the cone of the adjunction morphism $SR\to\id$.
\end{remark}

An object $X\in\catD$ is called \emph{$d$-spherelike} if $\Hom^\bullet(X,X) = \kk \oplus \Sigma^{-d}\kk$; see \cite{HKP} and also Section~\ref{sub:coarse_classification}. We will now show that reasonable exceptionable cycles come with a spherelike object. For this purpose, we call an exceptional cycle $E_* = (E_1,\ldots,E_n)$ \emph{irredundant} if $E_n\notin\thick{}{E_1,\ldots,E_{n-1}}$. Recall that an exceptional $n$-cycle $(E_1,\ldots,E_n)$ comes with a tuple of integers $(k_1,\ldots,k_n)$ and that we have set $k = k_1+\cdots+k_n$.

\begin{proposition} \label{prop:spherelike_top}
Let $E_* = (E_1,\ldots,E_n)$ be an irredundant exceptional $n$-cycle in $\catD$. Then there exists a $(k+1-n)$-spherelike object $X\in\catD$ with non-zero maps $X\to E_1$ and $\Sigma^{n-1-k+k_n}E_n\to X$.
\end{proposition}

\begin{proof}
Inductively, we construct a series of objects $X_1,\ldots,X_n$ with the following properties for $i<n$:
\begin{enumerate}[label = (\roman*)]
\item $X_i$ is exceptional,
\item $X_i \in \thick{}{E_1,\ldots,E_i}$,
\item $\Hom^\bullet(X_i,E_{i+1}) = \Sigma^{-l_i}\kk$ with $l_i \coloneqq k_1+\cdots+k_i + 1-i$.
\end{enumerate}
These conditions are satisfied for $X_1\coloneqq E_1$, because $\Hom^\bullet(E_1,E_2)$ is generated by $\alpha_1\colon E_1\to\Sigma^{k_1}E_2$. With $X_i$ already constructed, by (iii) there is a unique object $X_{i+1}$ with a non-split distinguished triangle
\[ X_{i+1} \to X_i \to \Sigma^{l_i} E_{i+1} \to \Sigma X_{i+1} . \]
Moreover, $(X_i,E_{i+1})$ is an exceptional pair with just one (graded) morphism by (ii) and (iii). Hence in the above triangle, the object $X_{i+1}$ is, up to suspension, the left mutation of that pair. In particular, $X_{i+1}$ is exceptional. By construction, $X_{i+1}$ satisfies (ii).

If $i+1<n$, then $\Hom^\bullet(X_i,E_{i+2})=0$ by (ii) and the definition of exceptional cycles, hence
 $\Hom^\bullet(X_{i+1},E_{i+2}) = \Hom^\bullet(\Sigma^{l_i-1}E_{i+1},E_{i+2})$. As $\alpha_{i+1}\colon E_{i+1}\to\Sigma^{k_{i+1}}E_{i+2}$ generates $\Hom^\bullet(E_{i+1},E_{i+2})$, we find that $\Hom^\bullet(X_{i+1},E_{i+2})$ is 1-dimensional, and situated in degree
 $l_i+k_{i+1} - 1 = l_{i+1}$.

Having constructed $X_{n-1}$ in this fashion, we can use (iii) to define
\[ X_n \to X_{n-1} \to \Sigma^{l_{n-1}} E_n \to \Sigma X_n . \]

This triangle induces a commutative diagram of complexes of $\kk$-vector spaces
\[ \xymatrix{
\Hom^\bullet(              X_n, X_n)  \ar[r]        & \Hom^\bullet(               X_n, X_{n-1}) \ar[r]           & \Hom^\bullet(X_n,              \Sigma^{l_{n-1}}E_n)       \\
\Hom^\bullet(           X_{n-1}, X_n)  \ar[r] \ar[u] & \Hom^\bullet(           X_{n-1}, X_{n-1})  \ar[r] \ar[u]   & \Hom^\bullet(X_{n-1},           \Sigma^{l_{n-1}}E_n) \ar[u] \\
\Hom^\bullet(\Sigma^{l_{n-1}}E_n, X_n)  \ar[r] \ar[u] & \Hom^\bullet(\Sigma^{l_{n-1}}E_n, X_{n-1})  \ar[r] \ar[u]^g & \Hom^\bullet(\Sigma^{l_{n-1}}E_n, \Sigma^{l_{n-1}}E_n) \ar[u]_f
} \]
and we know that $\Hom^\bullet(X_{n-1}, X_{n-1}) = \Hom^\bullet(\Sigma^{l_{n-1}}E_n,\Sigma^{l_{n-1}}E_n) = \kk$, since $X_{n-1}$ and $E_n$ are exceptional.
Moreover, we get $\Hom^\bullet(X_{n-1}, \Sigma^{l_{n-1}}E_n) = \kk$ from applying $\Hom^\bullet(\blank,E_n)$ to the triangle defining $X_{n-1}$ and using $X_{n-2}\in\clext{E_1,,\ldots,E_{n-2}}$, none of which map to $E_n$. In particular, the map $f$ sends the identity to the morphism $X_{n-1} \to \Sigma^{l_{n-1}}E_n$ defining $X_n$. Hence $f$ is an isomorphism, thus $\Hom^\bullet(X_n, \Sigma^{l_{n-1}}E_n) = 0$ and we arrive at the isomorphism $\Hom^\bullet(X_n, X_n) \isom \Hom^\bullet(X_n, X_{n-1})$.

We turn to $\Hom^\bullet(\Sigma^{l_{n-1}}E_n, X_{n-1})$. By (ii) and $\Hom^\bullet(E_n,E_i)=0$ for $1<i<n$,
\[ \Hom^\bullet(\Sigma^{l_{n-1}}E_n, X_{n-1}) =  \Hom^\bullet(\Sigma^{l_{n-1}}E_n, X_{n-2}) = \cdots =  \Hom^\bullet(\Sigma^{l_{n-1}}E_n, X_1) = \Sigma^{n-2-k}\kk , \]
where $k = k_1 + \cdots + k_n$ as before. Now $g$ is a map of two 1-dimensional complexes. This map cannot be an isomorphism, because this would force $\Hom^\bullet(X_n,X_n)=0$, hence $X_n=0$ but we have $X_{n-1}\in\thick{}{E_1,\ldots,E_{n-1}}$ by (ii) and also $E_n\notin\thick{}{E_1,\ldots,E_{n-1}}$ as $E_*$ is irredundant. Therefore we find
\[ \Hom^\bullet(X_n, X_n) \cong \Hom^\bullet(X_n, X_{n-1}) \cong \kk \oplus \Sigma^{n-1-k}\kk . \]
Hence $X\coloneqq X_n$ is indeed $(k+1-n)$-spherelike. The degrees of non-zero maps in $\Hom^\bullet(E_n,X)$ and $\Hom^\bullet(X,E_1)$ are computed with the same methods as above.
\end{proof}

\begin{example}
The additional hypothesis on $E_*$ is necessary: consider $\sD = \Db(\kk A_3)$ for the $A_3$-quiver $1\to2\to3$.
Denoting the injective-projective module by $M=P(1)=I(3)$, the sequence $E_*=(S(1),S(2),S(3),M)$ is an exceptional cycle with $k_*=(1,1,0,0)$. The cycle is redundant because of $M\in\thick{}{S(1),S(2),S(3)}$; note that $(S(1),S(2),S(3))$ is a full exceptional collection for $\sD$.

Following the iterative construction of the above proof, we get $X_1 = S(1)$, $X_2=I(2)$ and $X_3=M$. This forces $X=X_4=0$, and we do not get a spherelike object in this case. Note that $E_*$ still gives a twist autoequivalence, which for this example is just $\TTT_{E_*}=\tau\inv$.
\end{example}


\section{Autoequivalence groups of discrete derived categories}
\label{sec:auteq}

\noindent
We now use the general machinery of the previous section to show that categories $\Db(\LLambda)$ possess two very interesting and useful autoequivalences. We will denote these by $\TTT_\cX$ and $\TTT_\cY$ and prove some crucial properties: they commute with each other, act transitively on the indecomposables of each $\cZ^k$ component and provide a weak factorisation of the Auslander--Reiten translation: $\TTT_\cX\TTT_\cY = \tau\inv$ on objects. Moreover, $\TTT_\cX$ acts trivially on $\cY$ and $\TTT_\cY$ acts trivially on $\cX$; see Proposition~\ref{prop:twist-action} and Corollary~\ref{cor:twist-action} for the precise assertions. We then give an explicit description of the group of autoequivalences of $\Db(\LLambda)$ in Theorem~\ref{thm:auteq}.

The category $\catD = \Db(\LLambda)$ with $n>r$ is Hom-finite, indecomposable, algebraic and has Serre duality (see Appendix~\ref{app:category-properties}). Therefore we can apply the results of the previous section to $\catD$.

Our first observation is that every sequence of $m+r$ consecutive objects at the mouth of $\cX^0$ is an exceptional $(m+r)$-cycle; likewise, every sequence of $n-r$ consecutive objects at the mouth of $\cY^0$ is an exceptional $(n-r)$-cycle, by which we mean a $(r+1)$-spherical object in case $n-r=1$.
For the moment, we specify two concrete sequences:
\begin{align*}
E_* = (E_1,\ldots,E_{m+r}) &\coloneqq (X^0_{m+r,m+r}, \ldots, X^0_{11}) , && \text{i.e.\ } E_i = X^0_{m+r+1-i,m+r+1-i} , \\
F_* = (F_1,\ldots,F_{n-r}) &\coloneqq (Y^0_{n-r,n-r}, \ldots, Y^0_{11}) , && \text{i.e.\ } F_i = Y^0_{n-r+1-i,n-r+1-i} .
\end{align*}

\begin{lemma} \label{lem:exceptional-cycle}
$E_*$ forms an exceptional $(m+r)$-cycle in $\catD$ with $k_*=(1,\ldots,1,1-r)$, and $F_*$ forms an exceptional $(n-r)$-cycle in $\catD$ with $k_*=(1,\ldots,1,1+r)$.
\end{lemma}

\begin{proof}
The object $X^0_{11}$ is exceptional by Lemma~\ref{lem:mouth_hammocks}, hence any object at the mouth $X^0_{ii}=\tau^{1-i}(X^0_{11})$ is.
This point also gives the second condition of exceptional cycles: for $i=1,\ldots,m+r-1$, we have
 $\SSS E_i = \Sigma\tau X^0_{m+r+1-i,m+r+1-i} = \Sigma X^0_{m+r-i,m+r-i} = \Sigma E_{i+1}$ and at the boundary step we have
 $\SSS E_{m+r} = \Sigma\tau X^0_{11} = \Sigma X^0_{00} = \Sigma^{1-r} X^0_{m+r,m+r} = \Sigma^{1-r} E_1$,
where we freely make use of the results stated in Section~\ref{sec:AR-quiver}.
Hence the degree shifts of the sequence $E_*$ are $k_1=\ldots=k_{m+r-1}=1$ and $k_{m+r}=1-r$. Finally, the required vanishing $\Hom(E_i,E_j)=0$ unless $j=i+1$ or $i=j$ again follows from Lemma~\ref{lem:mouth_hammocks}.

The same reasoning works for $\cY$, now with the boundary step degree computation
 $\SSS F_{n-r} = \Sigma\tau Y^0_{11} = \Sigma Y^0_{00} = \Sigma^{1+r} Y^0_{n-r,n-r} = \Sigma^{1+r} F_1$.
\end{proof}

The next lemma shows that the functors $\FF_{E_*}$ and $\FF_{F_*}$ of the last section take on a particularly simple form, where we use the notation ${}_X,X_0,{}^0Y,Y^0$ from Sections~\ref{sub:hammocks_X},\ref{sub:hammocks_Y}:

\begin{lemma} \label{lem:F-functor}
For $X\in\ind{\cX}$ and $Y\in\ind{\cY}$,
\[ \FF_{E_*}(X) = {}_0X \oplus \SSS\inv X_0 , \qquad
   \FF_{F_*}(Y) = {}^0Y \oplus \SSS\inv X^0 . \]
\end{lemma}

\begin{proof}
This follows immediately from the definition of these functors in Section~\ref{sec:exceptional_cycles}, Proposition~\ref{prop:X-hammocks} and Properties~\ref{properties}(3), i.e.\ $\Sigma^r|_\cX = \tau^{-m-r}$ and $\Sigma^r|_\cY = \tau^{n-r}$ on objects.

Note that the right-hand sides extend to direct sums.
Another description of $\FF_{E_*}(X)$ is as the minimal approximation of $X$ with respect to the mouth of $\cX^0$, and analogously for $\FF_{F_*}$.
\end{proof}

The actual choice of exceptional cycle is not relevant as the following easy lemma shows. We only state it for $E_*$ but the similar statement holds for $F_*$, with the same proof. This allows us to write $\TTT_\cX$ instead of $\TTT_{E_*}$ and $\TTT_\cY$ instead of $\TTT_{F_*}$.

\begin{lemma} \label{lem:invariance_of_twists}
Any two exceptional cycles $E_*, E'_*$ at the mouths of $\cX$ components differ by suspensions and AR translations, and the associated twist functors coincide: $\TTT_{E_*} = \TTT_{E'_*}$.
\end{lemma}

\begin{proof}
A suitable iterated suspension will move $E'_*$ into the $\cX$ component that $E_*$ inhabits, and two exceptional cycles at the mouth of the same AR component obviously differ by some power of the AR translation. Thus we can write $E'_* = \Sigma^a\tau^b E_*$ for some $a,b\in\IZ$. We point out that the suspension and the AR translation commute with all autoequivalences (it is a general and easy fact that the Serre functor does, see \cite[Lemma~1.30]{Huybrechts}). Finally, we have $\TTT_{E'_*} = \TTT_{\Sigma^a\tau^bE_*} = \Sigma^a\tau^b\TTT_{E_*}\Sigma^{-a}\tau^{-b} = \TTT_{E_*}$, using Lemma~\ref{lem:conjugated_twists}.
\end{proof}

\begin{proposition} \label{prop:twist-action}
The twist functors $\TTT_\cX$ and $\TTT_\cY$ act as follows on objects of $\Db(\Lambda)$, where $X\in\cX, Y\in\cY$ and $k=0,\ldots,r-1$ and $i,j\in\IZ$:
\begin{align*}
\TTT_\cX(X) &= \tau^{-1}(X), & \TTT_\cX(Y) &= Y,      & \TTT_\cX(Z^k_{i,j}) &= Z^k_{i+1,j} , \\
\TTT_\cY(Y) &= \tau^{-1}(Y), & \TTT_\cY(X) &= X,      & \TTT_\cY(Z^k_{i,j}) &= Z^k_{i,j+1} .
\end{align*}
\end{proposition}

\begin{corollary} \label{cor:twist-action}
The twist functors $\TTT_\cX$ and $\TTT_\cY$ act simply transitively on each component $\cZ^k$ and factorise the inverse AR translation:
$\TTT_\cX\TTT_\cY = \TTT_\cY\TTT_\cX = \tau^{-1}$ on the objects of $\Db(\Lambda)$.
Moreover, $\TTT_\cX$, $\TTT_\cY$ and $\Sigma$ act transitively on $\ind{\cZ}$.
\end{corollary}

\begin{proof}[Proof of the proposition.]
By Lemma~\ref{lem:mouth_hammocks}, we have $\Hom^\bullet(X^k_{ii},Y)=0$ for all $Y\in\cY$. This immediately implies $\TTT_\cX|_\cY = \id$.

\emph{Action of $\TTT_\cX$ on objects of $\cX$:} we recall that the proof of Theorem~\ref{thm:twist} showed
$\TTT_\cX(E_i)=\Sigma^{1-k_{i-1}}(E_{i-1})$, and furthermore $k_1=\ldots=k_{m+r-1}=1$ and $k_{m+r}=1-r$ from Lemma~\ref{lem:exceptional-cycle}.
Hence $\TTT_\cX(E_i)=\tau^{-1}(E_i)$ for all $i$ --- as explained in Lemma~\ref{lem:invariance_of_twists}, this holds for any exceptional cycle at an $\cX$ mouth. Since $\TTT_\cX$ is an equivalence and each $\cX$ component is of type $\IZ A_\infty$, this forces $\TTT_{E_*}|_{\cX}=\tau^{-1}$ on objects.

\emph{Action of $\TTT_\cX$ on objects of $\cZ$:}
Pick $Z^0_{ij}\in\cZ^0$ with $1\leq i\leq m+r$. Using $\TTT_\cX = \TTT_{E_*}$ with the cycle originally specified, i.e.\ $E_{m+r}=X^0_{11}$, we invoke Lemma~\ref{lem:mouth_hammocks} once more to get
 $\kk = \Hom^\bullet(X^0_{ii},Z^0_{ij}) = \Hom^\bullet(E_{m+r+1-i},Z^0_{ij})$,
and $0 = \Hom^\bullet(E_l,Z^0_{ij})$ for all $l\neq m+r+1-i$. So
$\FF_{E_*}(Z^0_{ij}) = X^0_{ii}$ and the triangle defining $\TTT_{E_*}(Z^0_{ij})$ is one of the special triangles of Properties~\ref{properties}(4):
\[ \xymatrix@R=2ex{
 \FF_{E_*}(Z^0_{ij}) \ar[r] \ar@{=}[d] & Z^0_{ij} \ar[r] \ar@{=}[d] & \TTT_{E_*}(Z^0_{i,j}) \ar[r] \ar@{=}[d] & \Sigma\FF_{E_*}(Z^0_{i,j}) \ar@{=}[d]\\
 X^0_{ii}           \ar[r]            & Z^0_{ij} \ar[r]            & Z^0_{i+1,j}          \ar[r]            &  \Sigma X^0_{i,i}
} \]
Application of AR translations extends this computation to arbitary $Z\in\cZ^0$, and suspending extends it to all $\cZ$ components, thus $\TTT_\cX(Z^0_{i,j}) = Z^0_{i+1,j}$.

\emph{Remaining cases:}
Analogous reasoning shows $\TTT_{F_*}(F_i)=\tau^{-1}(F_i)$ for all $i=1,\ldots,n-r$, and the rest of the above proof works as well: $\TTT_\cY(Z^0_{i,j})=Z^0_{i,j+1}$, now using the other special triangle.
\end{proof}

The following technical lemma about the additive closures of the $\cX$ and $\cY$ components will be used later on, but is also interesting in its own right. Using the twist functors, the proof is easy.

\begin{lemma} \label{lem:thick_AR-components}
Each of $\cX$ and $\cY$ is a thick triangulated subcategory of $\catD$.
\end{lemma}

\begin{proof}
The proof of Proposition~\ref{prop:twist-action} contains the fact $\thick{}{E_*}\orth = \cY$. Perpendicular subcategories are always closed under extensions and direct summands; since $\thick{}{E_*}$ is by construction a triangulated subcategory, the orthogonal complement $\cY$ is triangulated as well.
\end{proof}

Our results enable us to compute the group of autoequivalences of $\Db(\LLambda)$. For $\Lambda(1,2,0)$, K\"onig and Yang \cite[Lemma~9.3]{Koenig-Yang} showed $\Aut(\Db(\Lambda(1,2,0))) \cong \IZ^2 \times \kk^*$.

\begin{theorem} \label{thm:auteq}
The group of autoequivalences of $\Db(\LLambda)$ is an abelian group generated by $\TTT_\cX$, $\TTT_\cY$, $\Sigma$ and $\Out(\LLambda)=\kk^*$, subject to one relation
\[  \Sigma^{r} = f_0\TTT_\cX^{\:m+r} \: \TTT_\cY^{\:r-n}  \qquad\text{for some } f_0\in\Out(\LLambda) . \]

As an abstract group, $\Aut(\Db(\LLambda)) \cong \IZ^2 \times \IZ/\ell \times \kk^*$, where $\ell\coloneqq\gcd(r,n,m)$.
\end{theorem}

\begin{proof}
In this proof, we will write $\sD = \Db(\LLambda)$ and $\Lambda = \LLambda$.

\Step{1} \emph{$\Out(\Lambda) = \kk^*$ from common scaling of arrows.}

It is a well-known fact that inner automorphisms induce autoequivalences of $\mod{\Lambda}$ and $\Db(\Lambda)$ which are isomorphic to the identity; see \cite[\S3]{Yekutieli}. The quotient group $\Out(\Lambda) = \Aut(\Lambda)/\Inn(\Lambda)$ acts faithfully on modules. The form of the quiver and the relations for $\LLambda$ imply that algebra automorphisms can only act by scaling arrows. 

Scaling of arrows leads to a subgroup $(\kk^*)^{m+n}$ of $\Aut(\Lambda)$. However, choosing an indecomposable idempotent $e$ (i.e.\ a vertex) together with a scalar $\lambda\in\kk^*$ produces a unit  $u = 1_\Lambda + (\lambda-1)e$, and hence an inner automorphism $c_u\in\Aut(\Lambda)$. It is easy to check that $c_u(\alpha)=\frac{1}{\lambda}\alpha$ if $\alpha$ ends at $e$, and $c_u(\alpha)=\lambda\alpha$ if $\alpha$ starts at $e$, and $c_u(\alpha)=\alpha$ otherwise. The form of the quiver of $\Lambda$ shows that an $(n+m-1)$-subtorus of the subgroup $(\kk^*)^{m+n}$ of arrow-scaling automorphisms consists of inner automorphisms.
Furthermore, the automorphism scaling all arrows simultaneously by the same number is easily seen not to be inner, hence, $\Out(\Lambda)=\kk^*$.

\Step{2} \emph{$\varphi\in\Aut(\sD)$ is the identity on objects $\iff$ $\varphi\in\Out(\Lambda)$.}

By Step 1, it is clear that algebra automorphisms act trivially on objects. Let now $\varphi\in\Aut(\sD)$ fixing all objects. In particular, $\varphi$ fixes the abelian category $\mod{\Lambda}$ and the object $\Lambda$, thus giving rise to $\varphi\colon \Lambda\to\Lambda$, i.e.\ an automorphism which by Step 1 can be taken to be outer.

\Step{3} \emph{The subgroup $\generate{\Sigma,\TTT_\cX,\TTT_\cY,\Out(\Lambda)}$ is abelian.}

The suspension commutes with all exact functors. Next, to see $[\TTT_\cX,\TTT_\cY]=\id$, we fix exceptional cycles $E_*$ for $\cX$ and $F_*$ for $\cY$; then $\TTT_{E*}\TTT_{F_*}(\TTT_{E_*})\inv = \TTT_{\TTT_{E_*}(F_*)} = \TTT_{F_*}$ by Lemma~\ref{lem:conjugated_twists} and Proposition~\ref{prop:twist-action}.
Let $f\in\Out(\Lambda)$. Then we have $[f,\TTT_\cX]=[f,\TTT_\cY]=\id$ by the same lemma, now using $f(E_*)=E_*$ and $f(F_*)=F_*$ from Step 2.

\Step{4} \emph{$\Aut(\sD)$ is generated by $\Sigma,\TTT_\cX,\TTT_\cY,\Out(\Lambda)$.}

Fix a $Z\in\ind{\cZ}$.  For any $\varphi\in\Aut(\sD)$, there are $a,b,c\in\IZ$ with $\Sigma^a\TTT_\cX^{\:b}\TTT_\cY^{\:c}(Z) = \varphi(Z)$, since the suspension and the twist functors act transitively on $\ind{\cZ}$ by Corollary~\ref{cor:twist-action}. Therefore, $\psi \coloneqq \Sigma^a\TTT_\cX^{\:b}\TTT_\cY^{\:c}\varphi\inv$ fixes $Z$. Moreover, since all autoequivalences commute with $\tau$ (because they commute with the Serre functor $\SSS = \Sigma\tau$ and with $\Sigma$) and $\cZ$ is a $\IZ A_\infty^\infty$-component, either $\psi$ is the identity on $\ind{\cZ}$ or else $\psi$ flips $\ind{\cZ}$ along the $Z\tau(Z)$ axis. However, the latter possibility is excluded by the action of $\Sigma^r|_\cZ$; see Properties~\ref{properties}(3).

By Properties~\ref{properties}(4), every indecomposable object of $\cX$ or $\cY$ is a cone of a morphism $Z_1\to Z_2$ for some $Z_1,Z_2\in\ind{\cZ}$. Moreover, the morphism $Z_1\to Z_2$ is unique up to scalars by Theorem~\ref{thm:hom-dimensions}.  (The proofs in that section make no use of the autoequivalence group. Note that by the proof of Theorem~\ref{thm:hom-dimensions}, morphism spaces between indecomposable objects in $\cZ$ are 1-dimensional, even for $r=1$.) Hence $\varphi$ actually fixes all indecomposable objects and thus all objects of $\Db(\Lambda)$.

Thus, by Step 2, $\psi\in\Out(\Lambda)$ and $\varphi\in\generate{\Sigma,\TTT_\cX,\TTT_\cY,\Out(\Lambda)}$.

\Step{5} \emph{$\Aut(\sD)$ is abelian with one relation $f_0 \Sigma^{-r} \TTT_\cX^{\:m+r} \TTT_\cY^{\:r-n}=\id$ for some $f_0\in\Out(\Lambda)$.}

By Steps 3 and 4, $\Aut(\sD)=\generate{\Sigma,\TTT_\cX,\TTT_\cY,\Out(\Lambda)}$ is abelian.
Properties~\ref{properties}(3) and Proposition~\ref{prop:twist-action} imply that the autoequivalence $\Sigma^{-r} \TTT_\cX^{\:m+r} \TTT_\cY^{\:r-n}$ fixes all objects of $\sD$, hence $f_0\Sigma^{-r} \TTT_\cX^{\:m+r} \TTT_\cY^{\:r-n} = \id$ for a unique automorphism $f_0\in\Out(\Lambda)$. 

Let now be $a,b,c\in\IZ$ and $g\in\Out(\Lambda)$ such that $g \Sigma^a\TTT_\cX^{\:b}\TTT_\cY^{\:c} = \id$. In particular, $\psi \coloneqq \Sigma^a\TTT_\cX^{\:b}\TTT_\cY^{\:c}$ fixes all objects. From $X=\psi(X)=\Sigma^a\TTT_\cX^{\:b}(X)=\Sigma^a\tau^{-b}(X)$ we deduce first $a=lr$ for some $l\in\IZ$ and then $b=-l(m+r)$; whereas $Y=\psi(Y)$ similarly implies $a=kr$ and $c=k(n-r)$ for some $k\in\IZ$. Hence $k=l$ and $\psi = \Sigma^{lr}\TTT_\cX^{\:-l(m+r)}\TTT_\cY^{\:l(n-r)}=f_0^{l}$. So $g=f_0^{-l}$ and altogether, 
  $ g\Sigma^a\TTT_\cX^{\:b}\TTT_\cY^{\:c}  = (f_0 \Sigma^{-r} \TTT_\cX^{\:m+r} \TTT_\cY^{\:n-r})^{-l}$ is a power of the stated relation.

\Step{6} \emph{$\Aut(\sD) \cong \IZ^2 \times \IZ/(r,n,m) \times \kk^*$.}

This is elementary algebra: let $A$ be a free abelian group of finite rank and $a_0\in A$, $f_0\in\kk^*$. Write $a_0=da_1$ with $d\in\IZ$ and $a_1$ indivisible. Choose $f_1\in\kk^*$ with $f_1^d=f_0$ --- this is possible because $\kk$ is algebraicaly closed. Now fix a group homomorphism $\nu\colon A\to\IZ$ with $\nu(a_1)=1$ --- this is possible because $a_1$ is indivisible. Consider the diagram with exact rows
\[ \xymatrix@R=3ex{
   0 \ar[r] & \{ (na_0,f_0^n) \mid n\in\IZ \} \ar[r]               & A \times \kk^* \ar[r]             & A \times \kk^* /\generate{(a_0,f_0)} \ar[r] & 0 \\
   0 \ar[r] & \{ (na_0,1)     \mid n\in\IZ \} \ar[r] \ar[u]^\alpha & A \times \kk^* \ar[r] \ar[u]^\beta & A \times \kk^* /\generate{(a_0,1)}   \ar[r] & 0 
} \]
where $\alpha(na_0,1) = (na_0,f_0^n)$ and $\beta(a,f) = (a,ff_1^{\nu(a)})$. Both maps are easily checked to be group homomorphisms and bijective. Moreover, the left-hand square commutes:
\[ \beta(na_0,1) = (na_0,f_1^{\nu(na_0)}) = (na_0,f_1^{nd\nu(a_1)}) = (na_0,f_0^{n\nu(a_1)}) = (na_0,f_0^n) = \alpha(na_0,1) . \]
Therefore we obtain an induced isomorphism between the right-hand quotients:
\[ A\times\kk^* / \generate{(a_0,f_0)} \cong A\times\kk^* / \generate{(a_0,1)} = A/\generate{a_0} \times \kk^* . \]
For the case at hand, $A=\IZ^3$ and $a_0=(r,n,m)\in\IZ^3$ and hence $A/a_0\cong\IZ^2\times\IZ/\ell$ with the greatest common divisor $\ell=(r,n,m)$, by the theory of elementary divisors.
\end{proof}

\begin{question}
It is natural to speculate about the action of the various functors on maps. More precisely, we ask whether
\begin{enumerate}[label=(\arabic*)~]
\item $\Sigma^r|_\cX = \tau^{-m-r}$ and $\Sigma^r|_\cY = \tau^{n-r}$ 
\item $\TTT_\cX|_\cX = \tau\inv$ and $\TTT_\cY|_\cY = \tau\inv$
\item $\Sigma^{r} = \TTT_\cX^{\:m+r} \: \TTT_\cY^{\:r-n}$
\end{enumerate}
hold as functors. In all cases, we know these relations hold on objects. Note that (1) and (2) together imply (3), and that (3) means $f_0=\id$ in Theorem~\ref{thm:auteq}.
\end{question}


\section{Hom spaces: dimension bounds and graded structure}
\label{sec:hom-dimensions}

\noindent
In this section, we prove a strong result about $\Db(\Lambda)\coloneqq \Db(\LLambda)$ which says that the dimensions of homomorphism spaces between indecomposable objects have a common bound. We also present the endomorphism complexes in Lemma~\ref{lem:graded-end}.

\subsection{Hom space dimension bounds}
The bounds are given in the the following theorem; for more precise information in case $r=1$ see Proposition~\ref{prop:r=1_intersection}.

\begin{theorem} \label{thm:hom-dimensions}
  Let $A,B$ be indecomposable objects of $\Db(\LLambda)$ where $n>r$. If $r\geq2$, then $\dim\Hom(A,B) \leq 1$ and if $r=1$, then $\dim\Hom(A,B) \leq 2$.
\end{theorem}

\begin{proof}
Our strategy for establishing the dimension bound follows that of the proofs of the Hom-hammocks.
Let $A,B\in\ind{\Db(\LLambda)}$ and assume $r>1$. In this proof, we use the abbreviation $\hom = \dim \Hom$. We want to show $\hom(A,B)\leq1$ by considering the various components separately.

\Case{$A\in\cX^k$ or $\cY^k$}
Consider first $A,B\in\cX^k$ and perform induction on the height of $A$.
If $A=A_0$ sits at the mouth, then $\hom(A,B)\leq1$ by Lemma~\ref{lem:mouth_hammocks}. For $A$ higher up, and assuming $\Hom(A,B)\neq0$, which means $B\in\rayfrom{\overline{AA_0}}$, we consider one of the triangles from Lemma~\ref{lem:triangles}
\[ \trilabels{{}_0 A}{A}{A''}{}{g}{}. \]

Using the Hom-hammock Proposition~\ref{prop:X-hammocks}, we see that $\Hom^\bullet(A'',B)=0$ if $B\in\rayfrom{A}$ and $\Hom^\bullet({}_0 A,B)=0$ otherwise. Thus the exact sequence
\[ \Hom(\Sigma ({}_0 A),B) \too \Hom(A'',B) \too \Hom(A,B) \too \Hom({}_0 A,B) \too \Hom(\Sigma\inv A'',B) \]
yields $\hom(A,B) \leq \hom(A'',B)$ if $B\in\rayfrom{A}$ and $\hom(A,B) \leq \hom({}_0 A,B)$ otherwise. The induction hypothesis then gives $\hom(A,B)\leq1$.

The subcase $B\in\cX^{k+1}$ follows from the above by Serre duality.

Furthermore, the above argument applies without change to $B\in\cZ^k$ --- with $\rayfrom{A}\subset\cZ^k$ understood to mean the subset of indecomposables of $\cZ^k$ admitting non-zero morphisms from $A$ (these form a ray in $\cZ^k$) and similarly $\rayto{B}\subset\cX^k$, and application of Proposition~\ref{prop:Z-hammocks}. An obvious modification, which we leave to the reader, extends the argument to $B\in\cZ^{k+1}$. The statements for $A\in\cY$ are completely analogous.

\Case{$A\in\cZ^k$}
In light of Serre duality, we don't need to deal with $B \in \cX$ or $B \in \cY$.
Therefore we turn to $B\in\cZ$. However, we already know from the proof of Proposition~\ref{prop:Z-hammocks} that the dimensions in the two non-vanishing regions $\rayfrom{\corayfrom{A}}$ and $\rayto{\corayto{\SSS A}}$ are constant. Since the $\cZ$ components contain the simple $S(0)$ and the twist functors together with the suspension act transitively on $\cZ$, it is clear that $\hom(A,A)=\hom(A,\SSS A)^*=1$. This completes the proof.
\end{proof}

\begin{proposition} \label{prop:r=1_intersection}
Let $r=1$ and $X,A\in\ind{\cX}$.
Then
\begin{align*}
 \hom(X,A)=2 &\iff A\in\rayfrom{\overline{XX_0}}\cap\corayto{\overline{{}_0(\SSS X),\SSS X}} . 
\end{align*}
\end{proposition}

The following diagram illustrates the proposition: all indecomposables $A$ in the heavily shaded square have $\dim\Hom(X,A)=2$:
\begin{center}
  \begin{tikzpicture}[scale=0.4]

\pgfmathsetmacro{\l}{1.5}    
\pgfmathsetmacro{\s}{1.8}    

\draw (-15,0) -- (15,0);

\draw (-11,13) -- ( 2,0) -- ( 6,4) -- (-3,13);
\draw ( 11,13) -- (-2,0) -- (-6,4) -- ( 3,13);

\fill[fill = gray!20]  (-6,4) -- ( 3,13) -- ( 11,13) -- (-2,0) -- cycle;
\fill[fill = gray!40]  ( 6,4) -- (-3,13) -- (-11,13) -- ( 2,0) -- cycle;
\fill[fill = gray!70]  ( 0,2) -- ( 4, 6) -- (  0,10) -- (-4,6) -- cycle;

\filldraw(-6,4) circle (4pt) node[left]  {$X$};
\filldraw( 6,4) circle (4pt) node[right] {$\SSS X$};
\filldraw(-2,0) circle (4pt) node[below] {$X_0$};
\filldraw( 2,0) circle (4pt) node[below] {${}_0(\SSS X)$};

\node at (-5,11)   {$\scriptstyle\corayto{\overline{{}_0(\SSS X),\SSS X}}$};
\node at ( 5,11)   {$\scriptstyle\rayfrom{\overline{XX_0}}$};

\end{tikzpicture}
\end{center}

\begin{proof}
The argument is similar to the computation of the Hom-hammocks in the $\cZ$ components from Section~\ref{sec:hammocks}. We proceed in several steps.

\Step{1} For any $A\in\ind{\cX}$ of height 0 the claim follows from Lemma~\ref{lem:mouth_hammocks}. Otherwise we consider the AR mesh which has $A$ at the top, and let $A'$ and $A''$ be the two indecomposibles of height $h(A)-1$. There are two triangles (see Lemma~\ref{lem:triangles}):
\begin{align}
{}_0A &\too A \too A'' \too \Sigma({}_0A)  = {}_0\Sigma A , \tag{ray} \label{ray} \\
A' &\too A \too A_0 \too \Sigma A'   , \tag{coray} \label{coray}
\end{align}
where, as before, ${}_0A$ and $A_0$ are the unique indecomposable objects on the mouth which are contained in respectively $\rayto{A}$ and $\corayfrom{A}$. Applying the functor $\Hom(X,\blank)$ to both triangles we obtain two exact sequences:
\begin{align}
 \Hom(X,{}_0A) \too \Hom(X,A) \stackrel{\phi}{\too} \Hom(X,A'')  \stackrel{\psi}{\too} \Hom(X,\Sigma \, {}_0A) , \label{first-seq} \\
 \Hom(X,\Sigma^{-1}A_0) \stackrel{\mu}{\too} \Hom(X,A') \too  \Hom(X,A) \stackrel{\delta}{\too} \Hom(X,A_0)     . \label{second-seq}
\end{align}
Since ${}_0A$ and $A_0$ lie on the mouth of the component, Lemma~\ref{lem:mouth_hammocks} implies that the outer terms have dimension at most 2.
Using the fact that $X_0$ and ${}_0 \SSS X$ are the only objects of the Hom-hammock from $X$ lying on the mouth, Lemma~\ref{lem:mouth_hammocks} actually yields:
\begin{align*}
   \hom(X,{}_0 A) > 0            &~\iff~  A \in \rayfrom{X_0}\cup \rayfrom{{}_0 \SSS X}, \\
   \hom(X,\Sigma \, {}_0A) > 0   &~\iff~  A \in \rayfrom{\Sigma^{-1} \, X_0} \cup \rayfrom{\Sigma^{-1} \,{}_0 \SSS X}, \\
   \hom(X,\Sigma^{-1} \, A_0) > 0 &~\iff~  A \in \corayto{\Sigma \, X_0} \cup \corayto{\Sigma \,{}_0 \SSS X}, \\
   \hom(X,A_0) > 0               &~\iff~  A \in \corayto{X_0} \cup \corayto{{}_0 \SSS X}.
\end{align*}
The spaces are 2-dimensional precisely when $A$ belongs to the intersections of the (co)rays on the right-hand side, which can only happen when ${}_0 \SSS X = X_0$.
The set of rays and corays listed above divide the component into regions. In this proof, each region is considered to be closed below and open above.

\Step{2} \textit{The function $\hom(X,\blank)$ is constant on each region, and changes by at most 1 when crossing a (co)ray if ${}_0 \SSS X \neq X_0$, and by at most 2 otherwise.}

The first claim is clear from exact sequences \eqref{first-seq} and \eqref{second-seq}. We show the second claim for rays; for corays the argument is similar.
We get
 $\hom(X,A)  \leq \hom(X,A'')+ \hom(X, {}_0 A)$
from sequence \eqref{first-seq}. This yields the stated upper bound for $\hom(X,A) $, as $\hom(X, {}_0 A) \leq 1$ when ${}_0 \SSS X \neq X_0$ and $\hom(X, {}_0 A) \leq 2$ otherwise. For the lower bound,
instead observe that
  $\hom(X,A'')\leq \hom(X,\Sigma {}_0 A) +\hom(X,A)$,
again from sequence \eqref{first-seq}.

\Step{3} \textit{ $\psi =0$ unless $A \in \rayfrom{\Sigma^{-1} \,{}_0 \SSS X}$ and $\mu =0$ unless $A \in \corayto{\Sigma X_0}$}

If $A \notin \rayfrom{\Sigma^{-1} \, X_0} \cup \rayfrom{\Sigma^{-1} \,{}_0 \SSS X}$ then $\hom(X,\Sigma {}_0 A)=0$ and so $\psi =0$ trivially. Therefore, we just need to consider $A \in \rayfrom{\Sigma^{-1} \, X_0}$ but $A\notin\rayfrom{\Sigma^{-1} \,{}_0 \SSS X}$, and in this case $\hom(X,\Sigma \, {}_0A)=1$. It is clear that the maps going down the coray from $X$ to $X_0$ span a 1-dimensional subspace of $\Hom(X,\Sigma \, {}_0A)$, which therefore is the whole space. Using properties of the $\IZ A_\infty$ mesh, the composition of such maps with a map along $\rayfrom{X_0}$ from $X_0$ to $\Sigma A$ defines a non-zero element in $\Hom(X,\Sigma \, A)$. Thus the map $\Hom(X,\Sigma {}_0 A) \to \Hom(X, \Sigma A)$ in the sequence \eqref{first-seq} is injective and it follows that $\psi =0$.
The proof of the second statement is similar: here we use the chain of morphisms in Properties~\ref{properties}(5) to show that the map $\Hom(X,\Sigma^{-1} A) \to \Hom(X, \Sigma^{-1} A_0)$ in the sequence \eqref{second-seq} is surjective.

\Step{4} \textit{If $\rayfrom{\Sigma^{-1} X_0}$ (or $\corayto{\Sigma {}_0 \SSS X}$, respectively) does not coincide with one of the other three (co)rays, then crossing it does not affect the value of $\hom(X,\blank)$.
}

Suppose $\rayfrom{\Sigma\inv \, X_0} \ni A$ doesn't coincide with $\rayfrom{X_0}$, $\rayfrom{{}_0\SSS X}$ or $\rayfrom{\Sigma\inv{}_0\SSS X}$. Thus $\hom(X,{}_0A)=0$, and from Step 3 the map $\psi =0$, hence $\Hom(X,A) = \Hom(X,A'')$. Similarly, suppose $A \in \corayto{\Sigma \,{}_0 \SSS X}$ and this doesn't coincide with any of the other corays. Then $\hom(X,A_0)=0$ and $\mu=0$ and again the claim follows.

\Step{5} \textit{ There are three possible configurations of rays and corays determining the regions where $\hom(X, \blank)$ is constant.}

It follows from Step 4 that it suffices to consider the remaining rays and corays,
\[ \rayfrom{\Sigma^{-1} \,{}_0 \SSS X}, \rayfrom{{\Sigma \,{}_0 \SSS X}}, \rayfrom{ X_0}
   \text{  and }
   \corayto{\Sigma X_0}, \corayto{{\Sigma \,{}_0 \SSS X}}, \corayto{ X_0}, \]
for determining the regional constants $\hom(X, \blank)$. Note that these are precisely the rays and corays required to bound the regions $\rayfrom{\overline{XX_0}}$ and $\corayto{\overline{{}_0(\SSS X),\SSS X}}$ of the statement of the proposition. Considering their relative positions on the mouth, $\Sigma^{-1} \,{}_0 \SSS X$ is always furthest to the left and $\Sigma X_0$ is furthest to the right, while ${}_0 \SSS X$ can lie to the left, or to the right, or coincide with $X_0$, depending on the height of $X$. We consider now the case where ${}_0 \SSS X$ is to the left of $X_0$. We label the regions in the following diagram by letters A--M (this is the order in which we treat them, and the subscripts indicate the claimed $\hom(X,\blank)$ for the region):
\begin{center}
\includegraphics[width=0.5\textwidth]{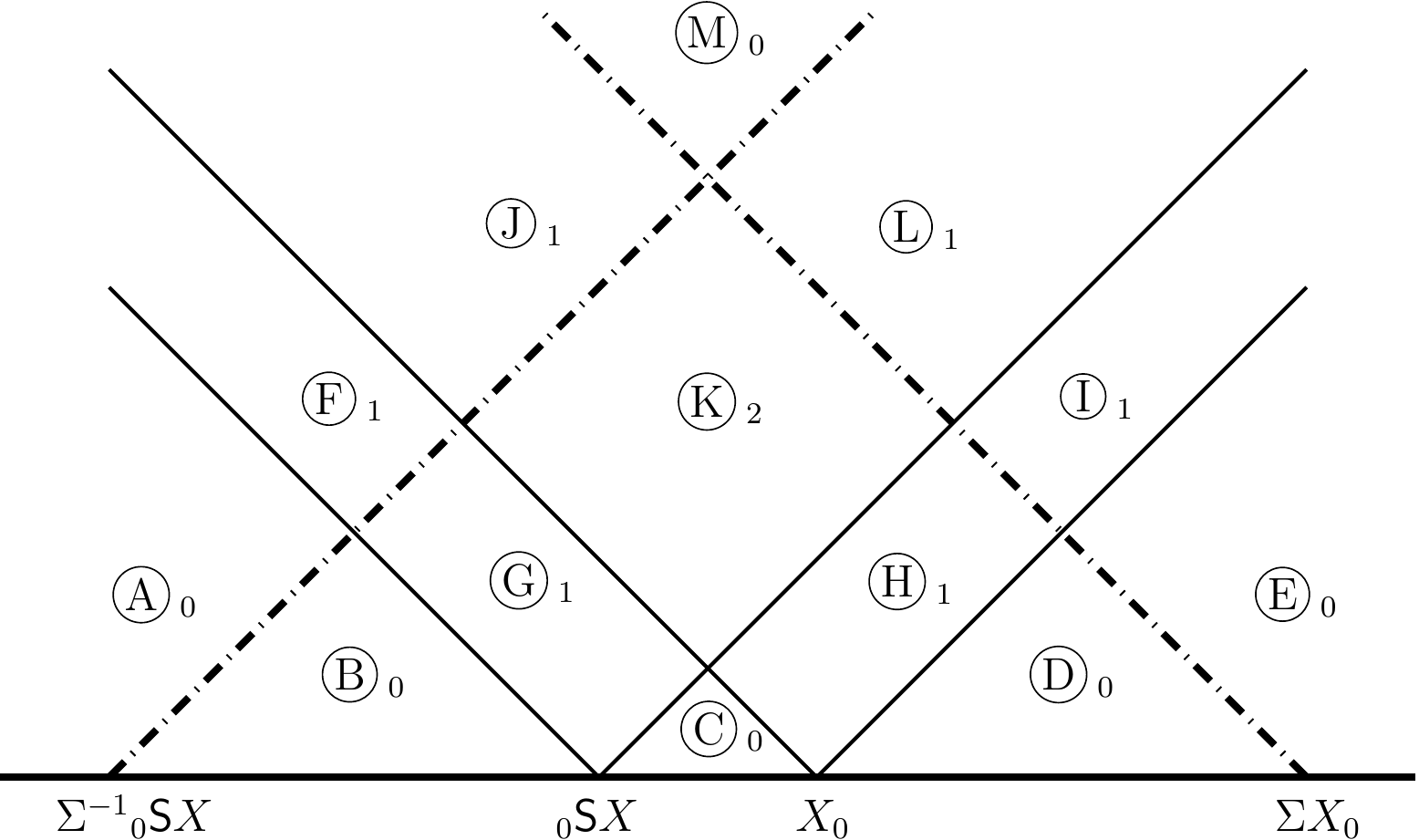}
\end{center}
First we note that regions A--E all contain part of the mouth and so $\hom(X,\blank)=0$ here. Looking at the maps from $X$ that exist in the AR component we see that $\hom(X,\blank) \geq 1$ on regions H, I, K and L; and on F, G, J and K using Serre duality. However regions F--I are reached by crossing a single ray or coray from one of the regions A--E. By Step 2 we thus get $\hom(X,\blank) = 1$ on regions F--I.

Now look at the element $A \in \rayfrom{\SSS{}_0X} \cap \corayto{X_0}$; this is the object of minimal height in region K. We can see that $A \in \corayfrom X$ and the map down the coray from $X$ to $A_0$, factors through the map from $A$ to $A_0$. Therefore the map $\delta$ in the second exact sequence \eqref{second-seq} is non-zero.
It is clear that $ A \notin \corayto{\Sigma X_0}$ so $\mu =0$ by Step 3 above.
We deduce from sequence \eqref{second-seq} that $\hom(X,A)>\hom(X,A')$, so $\hom(X,A)>1$ since $A'$ is in region G. Since $A$ is an object in region K, which can be reached from region D by crossing just two rays, Step 2 now gives $\hom(X,\blank)=2$ on region K.

In the same vein, consider $A \in \rayfrom{\SSS{}_0X} \cap \corayto{\Sigma X_0}$, the object of minimal height in region L. Observe that $A'' \in \rayfrom{\tau^{-1}\SSS{}_0X} \cap \corayto{\Sigma X_0} = \add \Sigma X$ from which we can see that the map to $\Hom(X,{}_0 A)$ in \eqref{first-seq} is surjective. Now $A \notin \rayfrom{\Sigma\inv{}_0 \SSS X}$, so $\psi =0$ by Step 3 and hence $\hom(X,A)=\hom(X,A'')$. With $A''$ in region I where we already know $\hom(X,A'')=1$, we get $\hom(X,\blank)=1$ on region L.

Finally we now take up $A \in \rayfrom{\Sigma^{-1}\SSS{}_0X} \cap \corayto{\Sigma X_0}$, the object of minimal height in region M. It is clear that $A \notin \rayfrom{\SSS{}_0X} \cup \rayfrom{X_0}$, so $\hom(X,{}_0 A) = 0$. A short calculation shows $A'' \in \rayfrom{X}$, and again using the chain of morphisms in Properties~\ref{properties}(5), we see that there is a map $X \to \Sigma {}_0 A = \SSS {}_0 X$ factoring through $A''$. Looking at the sequence \eqref{first-seq} it follows that $\hom(X,A)<\hom(X,A'')= 1$ since $A''$ is in region L. Therefore, $\hom(X,\blank)=0$ on region M. For region J, we see that since it is sandwiched between regions K and M, $\hom(X,\blank)=1$ here.

This deals with the case that ${}_0 \SSS X$ lies to the left of $X_0$. If instead it lies to the right, analogous reasoning applies. Finally, if ${}_0 \SSS X = X_0$, matters are simpler: in that case, the regions C and F--I all vanish.
\end{proof}

\subsection{Graded endomorphism algebras}
In this section we use the Hom-hammocks and universal hom space dimension bounds to recover some results of Bobi\'nski on the graded endomorphism algebras of algebras with discrete derived categories; see \cite[Section~4]{Bobinski}.
Our approach is somewhat different, so we provide proofs for the convenience of the reader.
Using these descriptions, we give a coarse classification of indecomposable objects of discrete derived categories in terms of their homological properties.

In order to conveniently write down the endomorphism complexes, we define four functions
$\lend^+_\cX, \lend^-_\cX, \lend^+_\cY, \lend^-_\cY \colon \IN\to\IN$ by
\[ \lend^+_\cX(h) \coloneqq \left\lfloor{\frac{ h }{m+r}}\right\rfloor, \quad
   \lend^-_\cX(h) \coloneqq \left\lfloor{\frac{h+1}{m+r}}\right\rfloor, \qquad
   \lend^+_\cY(h) \coloneqq \left\lfloor{\frac{h+1}{n-r}}\right\rfloor, \quad
   \lend^-_\cY(h) \coloneqq \left\lfloor{\frac{ h }{n-r}}\right\rfloor. \]
We write $\lend^\pm(A)$ to mean $\lend^\pm_\cX(h(A))$ or $\lend^\pm_\cY(h(A))$ for $A\in\ind{\cX}$ or $A\in\ind{\cY}$, respectively.

\begin{lemma} \label{lem:graded-end}
The endomorphism complexes of $A\in\ind{\cX}$ and $B\in\ind{\cY}$ are
\[ \Hom^\bullet(A,A) = \bigoplus_{l=0}^{\lend^+(A)} \Sigma^{-lr}\kk \:\oplus \bigoplus_{l=1}^{\lend^-(A)} \Sigma^{lr-1}\kk
   \text{~~and~~}
   \Hom^\bullet(B,B) = \bigoplus_{l=0}^{\lend^-(B)} \Sigma^{lr}\kk \:\oplus \bigoplus_{l=1}^{\lend^+(B)} \Sigma^{-lr-1}\kk . \]
\end{lemma}

\noindent
In words, the functions $\lend^+$ and $\lend^-$ determine the ranges of self-extensions of positive and negative degree, respectively. We point out that the result holds for all $r\geq1$.

\begin{proof}
Let $A\in\ind{\cX}$, assuming $r>1$. Suspending if necessary, we may suppose that $A = X^0_{ij}$. We are looking for all $d\in\IZ$ with $\Hom^d(A,A)=\Hom(A,\Sigma^dA)\neq0$. By Proposition~\ref{prop:X-hammocks}, this is only possible for either $d \equiv 0$ or $d \equiv 1$ modulo $r$.

We start with the first possibility: $d=lr$ for some $l\in\IZ$. By Properties~\ref{properties}(3) and (2),
\[ \Sigma^{lr}A = \tau^{-l(m+r)}A = X^0_{i+l(m+r),j+l(m+r)} \]
which is an indecomposable object in $\cX^0$ sharing its height $h=j-i$ with $A$. Again using Proposition~\ref{prop:X-hammocks},
we can reformulate the claim as follows:
\begin{align*}
           & \Hom^{lr}(A,A)\neq0
   \iff     \Sigma^{lr}A\in\rayfrom{\overline{AA_0}}
\\ \iff~~ & \Sigma^{lr}A = X^0_{i+l(m+r),j+l(m+r)} \in \{ A,\tau^{-1}A,\ldots,\tau^{-h}A \}
                                                 = \{ X^0_{ij}, X^0_{i+1,j+1},\ldots, X^0_{i+h,j+h} \}
\\ \iff~~ & i \leq i+l(m+r) \leq i+h
   \iff     0 \leq l(m+r) \leq h
\\ \iff~~ & 0 \leq l \leq \lend_\cX^+(h) = \lend^+(A) , \\
\intertext{where the set of $h+1$ objects in the second line are precisely the objects in $\rayfrom{\overline{AA_0}}$ of height $h$.
We now turn to the other possibility, $d=1+lr$ for some $l\in\IZ$. Here we get}
          & \Hom^{1+lr}(A,A)\neq0
   \iff     \Sigma^{1+lr}A\in\rayfrom{\overline{{}_0\SSS A,\SSS A}}
\\ \iff~~ & \Sigma^{1+lr}A = X^1_{i+l(m+r),j+l(m+r)} \in \{ \tau^h\SSS A,\ldots,\SSS A \}
                                                  = \{ X^1_{i-h-1,j-h-1},\ldots, X^1_{i-1,j-1} \}
\\ \iff~~ & i-h-1 \leq i+l(m+r) \leq i-1
   \iff     -h-1 \leq l(m+r) \leq -1
\\ \iff~~ & 1 \leq -l \leq \lend_\cX^-(h) = \lend^-(A) .
\end{align*}
As we know from Theorem~\ref{thm:hom-dimensions}, all Hom spaces have dimension 1 when $r>1$, these two computations give
\vspace{-2ex}
\[ \Hom^\bullet(A,A)
 = \bigoplus_{l\in\IZ} \Sigma^{-l} \Hom(A,\Sigma^l A)
 = \bigoplus_{l=0}^{\lend^+(A)} \Sigma^{-lr}\kk \:\oplus \bigoplus_{l=1}^{\lend^-(A)} \Sigma^{lr-1}\kk . \]

For $r=1$ and $A=X^0_{ij}\in\ind{\cX}$, by Proposition~\ref{prop:X-hammocks} the hammock $\Hom(A,\blank)\neq0$ is $\rayfrom{\overline{AA_0}}\cup\corayto{\overline{{}_0(\SSS A),\SSS A}}$. We treat each part separately:
\begin{align*}
           & \Sigma^lA = \tau^{-l(m+1)}X^0_{ij} = X^0_{i+l(m+1),j+l(m+1)} \in \rayfrom{\overline{AA_0}}
\\ \iff ~~ & 0 \leq l(m+1) \leq h \iff 0 \leq l \leq \lend^+(h)
\intertext{and, noting $\SSS A = X^0_{i+m,j+m}$,}
           & \Sigma^lA \in \corayto{\overline{{}_0(\SSS A),\SSS A}}
\\[-1.0ex]
   \iff ~~ & m-h \leq l(m+1) \leq m \iff 0 \leq -l \leq \Big\lfloor\frac{h-m}{m+1}\Big\rfloor = 1 + \lend^-(h) .
\end{align*}
The last inequality translates to the same degree range as in the statement of the lemma --- note the index shift by 1.
The claim for $\Hom^\bullet(B,B)$ for $B\in\ind{\cY}$ is proved in the same way, now using $h=i-j$, $\Sigma^r=\tau^{n-r}$ and the hammocks specified by Proposition~\ref{prop:Y-hammocks}.
\end{proof}

\subsection{Coarse classification of objects} \label{sub:coarse_classification}
Our previous results allow us to give a crude grouping of the indecomposable objects of $\Db(\LLambda)$. In the $\cX$ and $\cY$ components, the distinction depends on the height of an object, i.e.\ the distance from the mouth; see page~\pageref{def:height}. Recall that an object $D$ of a $\kk$-linear Hom-finite triangulated category $\sD$ is \emph{exceptional} if $\hom^*(D,D)=1$, then $\Hom^\bullet(D,D)=\kk$; see Appendix~\ref{sec:exceptionals}, and $D$ is called \emph{spherelike} if $\hom^*(D,D)=2$, then $\Hom^\bullet(D,D)=\kk\oplus\Sigma^{-d}\kk$ as graded vector spaces for some $d\in\IZ$ and $D$ is called $d$-spherelike; see \cite{HKP} for details. Assuming $\sD$ has a Serre functor $\SSS$, a $d$-spherelike object $D$ is called \emph{$d$-spherical} if $\SSS(D)=\Sigma^dD$; see \cite[\S8]{Huybrechts}.

\begin{proposition} \label{prop:object_types}
Each object $A\in\ind{\Db(\LLambda)}$ is of exactly one type below:
\begin{itemize}[leftmargin=1em]
\item Exceptional if $A\in\cZ$, or $A\in\cX$ with $h(A)<m+r-1$, or $A\in\cY$ with $h(A)<n-r-1$.
\item $(1-r)$-spherelike if $A\in\cX$ with $h(A)=m+r-1$.
\item $(1+r)$-spherelike if $A\in\cY$ with $h(A)=n-r-1$.
\item $\dim\Hom^*(A,A)\geq3$ with $\Hom^{<0}(A,A)\neq0$ otherwise.
\end{itemize}
\end{proposition}

\begin{remark}
In fact, the direct sum $E_1\oplus E_2$ of two exceptional objects $E_1$ and $E_2$ with
 $\Hom^\bullet(E_1,E_2) = \Hom^\bullet(E_2,E_1) = 0$ is a 0-spherelike object.
Examples for $r>1$ are given by taking $E_1\in\cX$ and $E_2\in\cY$ at the mouths.
The theory of spherelike objects also applies in this degenerate case, but is less interesting \cite[Appendix]{HKP}.
\end{remark}

\begin{remark}
We can infer the existence of $(1-r)$-spherelike indecomposable objects in $\cX$ and $(1+r)$-spherelike objects in $\cY$ also from Proposition~\ref{prop:spherelike_top} and Lemma~\ref{lem:exceptional-cycle}. To any reasonable $\kk$-linear triangulated category, \cite{HKP2} associates a poset derived from indecomposable spherelike objects. In \cite[\S6]{HKP2}, these posets are computed for discrete derived algebras.
\end{remark}

\begin{proof}
We know from Lemma~\ref{lem:proj-position} that the projective module $P(n-r)\in\cZ$. This is an exceptional object by Proposition~\ref{prop:Z-hammocks}. As the autoequivalence group acts transitively on $\ind{\cZ}$ by Corollary~\ref{cor:twist-action}, every indecomposable object of $\cZ$ is exceptional.
The remaining parts of the proposition all follow from Lemma~\ref{lem:graded-end}. We only give the argument for $A\in\ind{\cX}$, as the one for indecomposable objects of $\cY$ runs entirely parallel.

Observing the trivial inequalities $0\leq\lend^+(A)\leq\lend^-(A)$, we see that $A$ is exceptional if and only if $1=\dim\Hom^*(A,A)=1+\lend^+(A)+\lend^-(A)$. In turn, this happens precisely if $\lend^-(A)=0$, which means $h<m+r-1$.

Similarly, $A$ is spherelike if and only if $2=\dim\Hom^*(A,A)=1+\lend^+(A)+\lend^-(A)$ which is equivalent to $\lend^+(A)=0$ and $\lend^-(A)=1$. The only solution of these equations is $h=m+r-1$. Furthermore, in this case the endomorphism complex is
$\Hom^\bullet(A,A) = \kk \oplus \Sigma^{r-1}\kk$, so that $A$ is indeed $(1-r)$-spherelike.
\end{proof}

\begin{corollary}
Spherical objects exist in $\Db(\LLambda)$ only if $m=0$, $r=1$ or $n-r=1$. More precisely, $A\in\ind{\Db(\LLambda)}$ is
\begin{itemize}[leftmargin=1em]
\item 0-spherical if and only if $m=0$, $r=1$ and $A$ sits at an $\cX$-mouth;
\item $n$-spherical if and only if $n=r+1$  and $A$ sits at a $\cY$-mouth.
\end{itemize}
\end{corollary}

\begin{proof}
The only candidates for spherical objects are the spherelike objects listed in Proposition~\ref{prop:object_types}.
Start with $A\in\cX$ with $h(A)=m+r-1$. Then $A$ is spherical if and only if $\SSS A = \Sigma^{1-r} A$. By $\SSS = \Sigma\tau$ and $\Sigma^{1-r} = \Sigma\tau^{m+r}$ (Properties~\ref{properties}(3)), this is equivalent to $\tau^{m+r-1}A=A$ which happens precisely if $m+r=1$. The only solution for this equation is $m=0$, $r=1$.

Next, $B\in\cY$ with $h(B)=n-r-1$ is spherical if and only if $\Sigma\tau B = \SSS B = \Sigma^{1+r}B = \Sigma\tau^{n-r} B$, so that here we get $\tau^{n-r-1}B=B$ which is possibly only for $n=r+1$.
\end{proof}


\section{Reduction to Dynkin type $A$ and classification results}
\label{sec:classifications}

\noindent
Two keys for understanding the homological properties of algebras are t-structures and co-t-structures, especially bounded ones. The main theorem of \cite{Koenig-Yang}, cited in the appendix as Theorem~\ref{thm:koenig-yang}, states that for finite-dimensional algebras, bounded co-t-structures are in bijection with silting objects, which are in turn in bijection with bounded t-structures whose heart is a length category; see Appendices~\ref{app:torsion} and \ref{app:correspondences} for a more detailed overview.

It turns out, however, that any bounded t-structure in $\Db(\LLambda)$ has length heart, and hence to classify both bounded t-structures and bounded co-t-structures it is sufficient to classify silting objects in $\Db(\LLambda)$. This is the main goal of this section. In the first part, we prove that any bounded t-structure in $\Db(\LLambda)$ is length, then we obtain a semi-orthogonal decompositon $\Db(\LLambda) = \sod{\Db(\kk A_{n+m-1}),\sZ}$, for some trivial thick subcategory $\sZ$, and use this to bootstrap Keller--Vossieck's classification of silting objects in the bounded derived categories of path algebras of Dynkin type $A$ to get a classification of silting objects in discrete derived categories.

\subsection{All hearts in $\boldsymbol{\Db(\LLambda)}$ are length}
\label{sec:length}

The main result of this section is:

\begin{proposition} \label{prop:length}
 Any heart of a t-structure of a discrete derived category has only a finite number of indecomposable objects up to isomorphism, and is a length category.
\end{proposition}

We prove these statements separately in the following lemmas. The first lemma is a general statement regarding t-structures, which is well known to experts, and included for the convenience of the reader. The second is a generalisation of the corresponding statement for the algebra $\Lambda(1,2,0)$ proved in \cite{Koenig-Yang}; the third is a general statement about Hom-finite abelian categories.

\begin{lemma}[cf.\ {\cite[Lemma 4.1]{HJY}}] \label{lem:one-suspension}
Let $\sD$ be a triangulated category equipped with a t-structure $(\sX,\sY)$ with heart $\sH = \sX \cap \Sigma \sY$. Then at most one suspension of any object of $\sD$ may lie in the heart $\sH$.
\end{lemma}

\begin{proof}
Let $0 \neq H \in \sH$.  We show that $\Sigma^n H \notin \sH$ for any $n \neq 0$. First suppose that $\Sigma^n H \in \sH$ for some $n > 0$. Then $H \in \Sigma^{-n} \sH$. We have $\Sigma^{-n} \sH \subseteq \Sigma^{-n} \Sigma \sY \subseteq \sY$. The condition $\Hom(\sX,\sY)=0$ then implies that $\Hom(H,H)=0$, a contradiction.

Now suppose $\Sigma^{-n} H \in \sH$ for some $n>0$. In this case we have $\Sigma^{n} \sH \subseteq \Sigma^{n} \sX \subseteq \Sigma \sX$, whence the condition $\Hom(\Sigma \sX,\Sigma \sY)=0$ gives the required contradiction.
\end{proof}

\begin{lemma} \label{lem:finiteheart}
   Any heart of a t-structure of a discrete derived category has a finite number of indecomposable objects up to isomorphism.
\end{lemma}

\begin{proof}
We use the fact that there can be no negative extensions between objects in the heart $\sH$ of a t-structure $(\sX,\sY)$. Suppose $\sH$ contains an indecomposable $Z\in\ind{\cZ}$. Then any other indecomposable object in $\sH$ must lie outside the hammocks $\Hom^{<0}(Z,\blank) \neq 0$ and $\Hom^{<0}(\blank,Z) \neq 0$.  Looking at the complement of these Hom-hammocks, it is clear that all objects of $\ind{\sH}\cap\cZ$ must be (co)suspensions of a finite set of objects, see Figure~\ref{fig:type-A} for an illustration. Lemma~\ref{lem:one-suspension} implies that at most one suspension can sit in the heart $\sH$; hence $\ind{\sH}\cap\cZ$ is finite.

Now consider the $\cX$ component. By Proposition~\ref{prop:object_types}, any object $X_{i,j}^l$ which is sufficiently high up in an $\cX$ component --- here $j-i \geq r+m-1$ will do --- has a negative self-extension. Such objects cannot lie in the heart (\cite[Lemma 4.1(a)]{HJY}, for instance) and so again, up to (co)suspension, $\ind{\sH}\cap\cX$ is finite. The argument for the $\cY$ component is similar.
\end{proof}

\begin{lemma}
Let $\sH$ be a Hom-finite abelian category with finitely many indecomposable objects. Then $\sH$ is a finite length category.
\end{lemma}

\begin{proof}
Since $\sH$ is a Hom-finite, $\kk$-linear abelian category, it is Krull--Schmidt; see \cite{Atiyah}. Now let $L$ be the direct sum of all indecomposable objects (up to isomorphism) of $\sH$. By assumption, this sum is finite and hence $L\in\sH$. We define the function $d\colon\text{Ob}(\sH)\to\IN$, $A\mapsto\dim\Hom(L,A)$.

If $A\subset B$ is a subobject, we obtain exact sequences $0 \to A \to B \to C \to 0$ and
 $0 \to \Hom(L,A) \to \Hom(L,B) \to \Hom(L,C)$.
This shows $d(A)\leq d(B)$. Moreover, if $d(A)=d(B)$, then the induced map $\Hom(L,B) \to \Hom(L,C)$ is zero.
For some $s\in\IN$, there is a surjection $p\colon L^{\oplus s}\onto B$, inducing a further surjection $q\colon L^{\oplus s}\onto C$. However, we also get $0 \to \Hom(L^{\oplus s},A) \to \Hom(L^{\oplus s},B) \xrightarrow{v} \Hom(L^{\oplus s},C)$. The dimensions of the first two Hom spaces are $sd(A)=sd(B)$, so that $v=0$. Since $v(p)=q$ by construction, this forces $C=0$.

Hence for $B\in\sH$, the function $d$ can only take the values $1,\ldots,d(B)-1$ on non-trivial subobjects. Thus ascending or descending chains of subobjects of $B$ must stabilise.
\end{proof}

\begin{remark}
Proposition~\ref{prop:length} means that the heart of each bounded t-structure in $\Db(\LLambda)$ is equivalent to $\mod{\Gamma}$, for a finite-dimensional algebra $\Gamma$ of finite representation type. Note that, by work of Schr\"oer and Zimmermann \cite{Schroer-Zimmermann}, $\Gamma$ is again gentle.
\end{remark}

Knowing this, we can now turn our attention solely to classifying the silting objects. The first step in our approach is to decompose $\Db(\LLambda)$ into a semi-orthogonal decomposition, one of whose orthogonal subcategories is the bounded derived category of a path algebra of Dynkin type $A$.

\subsection{A semi-orthogonal decomposition: reduction to Dynkin type $A$}
\label{sec:dynkin-A}

We start by showing that the derived categories of derived-discrete algebras always arise as extensions of derived categories of path algebras of type $A$ by a single exceptional object.

\begin{proposition} \label{prop:embedding-A}
Let $Z\in\ind{\cZ}$ and $\sZ=\thick{\Db(\Lambda)}{Z}$. Then $\sZ\orth \simeq \Db(\kk A_{n+m-1})$ and there is a semi-orthogonal decomposition $\Db(\LLambda) = \sod{\Db(\kk A_{n+m-1}), \sZ}$. In particular, $\sZ$ is functorially finite in $\Db(\LLambda)$. Moreover, $\Db(\LLambda)$ has a full exceptional sequence.
\end{proposition}

\begin{proof}
By Proposition~\ref{prop:object_types}, the object $Z$ is exceptional. This implies, on general grounds, that the thick hull of $Z$ just consists of sums, summands and (co)suspensions: $\sZ = \add(\Sigma^i Z \mid i \in \IZ)$ and that $\sZ$ is an admissible subcategory of $\Db(\Lambda)$; for this last claim see \cite[Theorem~3.2]{Bondal}. Furthermore $\Db(\Lambda) = \sod{\sZ\orth,\sZ}$ is the standard semi-orthogonal decomposition for an exceptional object; see Appendix~\ref{sec:exceptionals} for details.
%

Lemma~\ref{lem:proj-position} places the indecomposable projective $P(n-r)$ in the $\cZ$ component of the AR quiver of $\Db(\Lambda)$. Using the transitive action of the autoequivalence group on $\ind{\cZ}$, see Corollary~\ref{cor:twist-action}, we thus can assume, without loss of generality, that $Z=P(n-r)=e_{n-r}\Lambda$. There is a full embedding
 $\iota\colon \Db(\Lambda / \Lambda e_{n-r} \Lambda) \to \Db(\Lambda)$
with essential image
 $\thick{\Db(\Lambda)}{e_{n-r} \Lambda}\orth = \sZ\orth$; see, for example, \cite[Lemma~3.4]{AKL}. Inspecting the Gabriel quiver of $\Lambda / \Lambda e_{n-r} \Lambda$, we see that this quiver satisfies the criteria of \cite[Theorem,~p.~2122]{Assem}. For the convenience of the reader, we list those criteria which are relevant for our case, where we have specialised the conditions of \cite{Assem} to bound quivers:

\bigskip
\noindent
\parbox{0.48\textwidth}{
\includegraphics[width=0.48\textwidth]{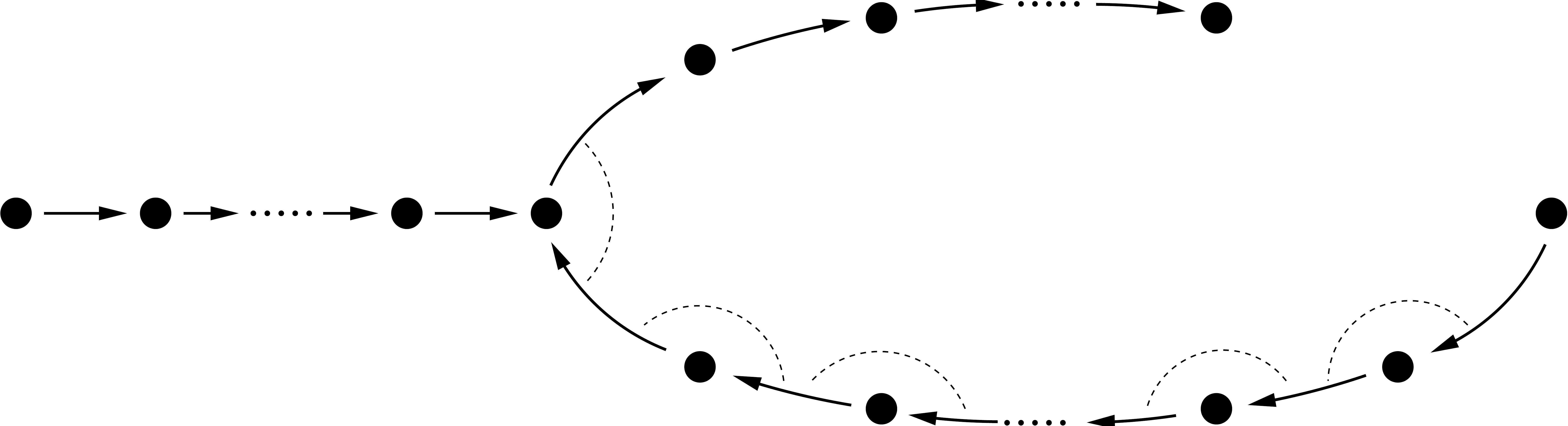}
}
\hfill
\parbox{0.47\textwidth}{\scriptsize
\noindent
\begin{tabular}{@{} p{0.05\textwidth} @{} p{0.42\textwidth} @{} p{0em} @{}}
($\alpha_1$) & The underlying graph is a tree. \\
($\alpha_3$) & All relations are zero-relations of length two. \\
($\alpha_4$) & Each vertex has at most four neighbours. \\
($\alpha_6$) &  \multicolumn{1}{@{}p{0.30\textwidth}}{A vertex with three neighbours sits in a full subgraph of the form:}
             &  \multicolumn{1}{@{}p{0.09\textwidth}}{ \hspace*{-0.10\textwidth}
                                                       \raisebox{-3ex}{\includegraphics[height=4.0ex]{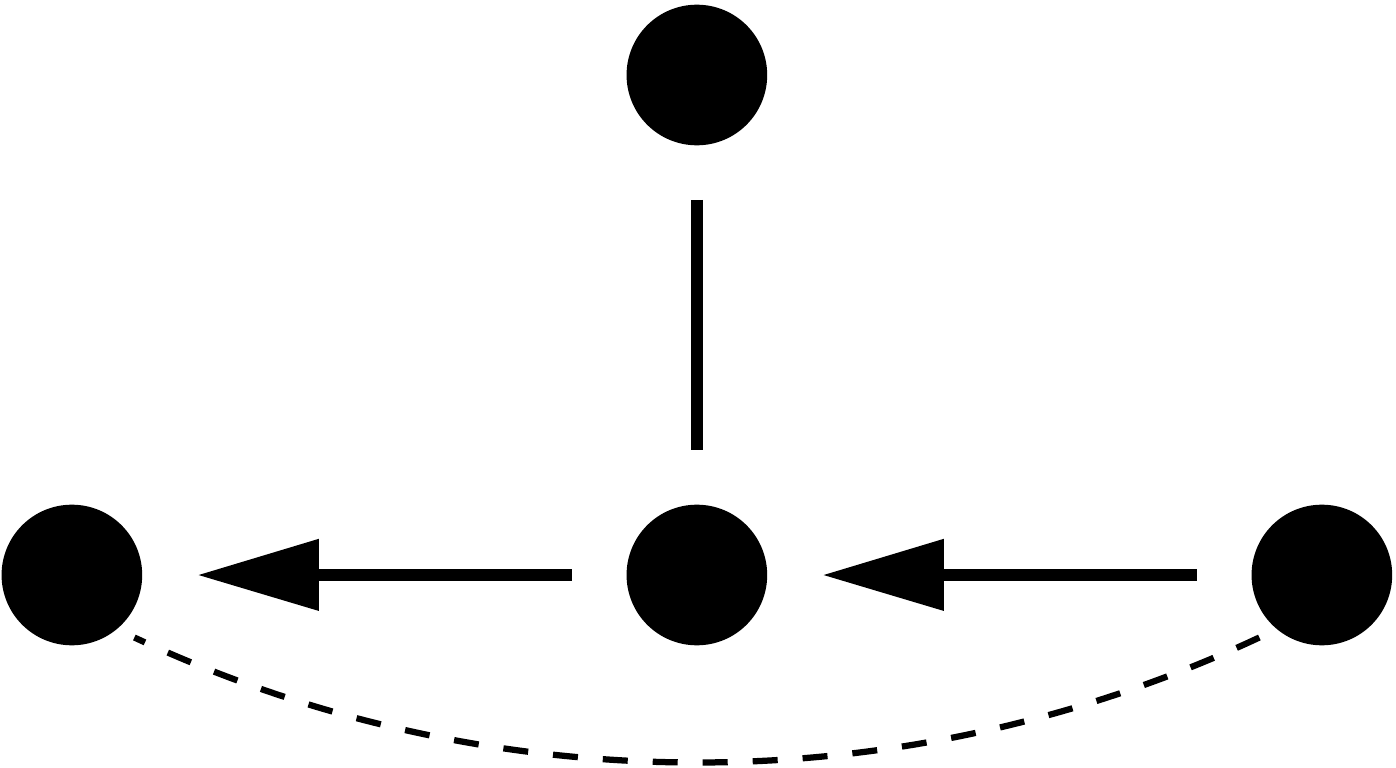}}} \\
\end{tabular}
}
\bigskip

Therefore $\Lambda / \Lambda e_{n-r} \Lambda$ is an iterated tilted algebra of type $A_{n+m-1}$. It is well known that
this implies
 $\Db(\Lambda / \Lambda e_{n-r} \Lambda) \simeq \Db(\kk A_{n+m-1})$; see \cite{Happel}.
Combining these pieces, we get $\sZ\orth \simeq \Db(\kk A_{n+m-1})$.
The final claim about $\Db(\Lambda)$ having a full exceptional sequence follows at once from the fact that $\Db(\kk A_{n+m-1})$ has one.
\end{proof}

\begin{remark} \label{rem:explicit}
  The subcategory of type $\Db(\kk A_{n+m-1})$ can be explicitly identified in the AR quiver of $\Db(\LLambda)$; see Figure \ref{fig:type-A}. The choice of right orthogonal to $Z$ was arbitrary, since Serre duality provides an equivalence $\orth\sZ\to\sZ\orth$, $X\mapsto\SSS(X)$. We mention in passing that the thick subcategory $\sZ$ is equivalent to $\Db(\kk A_1)$.
\end{remark}

\begin{figure}
  \begin{tikzpicture}[scale=0.2]
  \begin{scope}[gray!50]
    \fill ( 1,0) circle (8pt); \fill ( 3,0) circle (8pt); \fill ( 5,0) circle (8pt); 
    \fill ( 2,1) circle (8pt); \fill ( 4,1) circle (8pt); \fill ( 3,2) circle (8pt);
    \fill (17,0) circle (8pt); \fill (19,0) circle (8pt); \fill (18,1) circle (8pt); \fill (19,2) circle (8pt);
  \end{scope}
  \begin{scope}[black]
    \fill ( 9,0) circle (8pt); \fill (11,0) circle (8pt); \fill (13,0) circle (8pt); 
    \fill (10,1) circle (8pt); \fill (12,1) circle (8pt); \fill (11,2) circle (8pt);
  \end{scope}  
  \draw[fill, gray!50] (3,4) -- (7,0) -- (11,4) -- (15,0) -- (19,4) -- (19.5,3.5) -- (19.5,18) -- (-1.5,18) -- (-1.5,0.5) -- (-1,0) -- cycle;
  \node at (0.5,9) {$\cX^0$};
  \begin{scope}[shift = {(26,0)}]
    \begin{scope}[gray!50]
      \fill ( 1,18) circle (8pt); \fill ( 3,18) circle (8pt); \fill (13,18) circle (8pt); \fill (15,18) circle (8pt); 
      \fill ( 2,17) circle (8pt); \fill (14,17) circle (8pt); \fill (19,18) circle (8pt);
    \end{scope}
    \begin{scope}[black]
      \fill ( 7,18) circle (8pt); \fill ( 9,18) circle (8pt); \fill ( 8,17) circle (8pt);
    \end{scope}
    \draw[fill, gray!50] (-1,18) -- (2,15) -- (5,18) -- (8,15) -- (11,18) -- (14,15) -- (17,18) -- (19.5,15.5) -- (19.5,0) -- (-1,0) -- cycle;
    \node at (1,9) {$\cY^0$};
  \end{scope}
  \begin{scope}[shift = {(54,0)}]
    \begin{scope}[gray!50]
      \fill (7,0) circle (8pt); \fill ( 9,0) circle (8pt); \fill (11,0) circle (8pt);  
      \fill (8,1) circle (8pt); \fill (10,1) circle (8pt); \fill ( 9,2) circle (8pt); 
      \fill (8,11) circle (8pt);       
      \fill (7,12) circle (8pt); \fill (9,12) circle (8pt); 
      \fill (6,13) circle (8pt); \fill (8,13) circle (8pt); \fill (10,13) circle (8pt);
      \fill (5,14) circle (8pt); \fill (7,14) circle (8pt); \fill ( 9,14) circle (8pt);
      \fill (6,15) circle (8pt); \fill (8,15) circle (8pt); \fill (7,16) circle (8pt); \fill (7,18) circle (8pt); 
    \end{scope}
    \begin{scope}[black]
      \fill (9,4) circle (8pt); \fill (8,5) circle (8pt); \fill (10,5) circle (8pt);
      \fill (7,6) circle (8pt); \fill (9,6) circle (8pt); \fill (11,6) circle (8pt);
      \fill (6,7) circle (8pt); \fill (8,7) circle (8pt); \fill (10,7) circle (8pt);      
      \fill (7,8) circle (8pt); \fill (9,8) circle (8pt); \fill ( 8,9) circle (8pt);      
    \end{scope} 
    \draw[fill, gray!50] (-3,18) -- (5,18) -- (6,17) -- (3,14) -- (7,10) -- (4,7) -- (8,3) -- (5,0) -- (-3,0) -- cycle;
    \draw[fill, gray!50] (9,18) -- (8,17) -- (12,13) -- (9,10) -- (13,6) -- (10,3) -- (13,0) -- (19,0) -- (19,18) -- cycle;
    \node at (-1,9) {$\cZ^0$};
    \draw[fill=white, thick] (9,10) circle (8pt);
    \node at (10.2,10) {$\scriptstyle Z$};
  \end{scope} 
\end{tikzpicture}

\bigskip

\begin{tikzpicture}[scale=0.3]
  \newcommand{\Sst}[1]{\scalebox{0.75}{$#1$}}

  \foreach \x in {-4,...,40} \foreach \y in {0,...,5}
     { \ifnumodd{\x+\y}{}{ \fill[gray!30] (\x,\y) circle (5pt); }}
 
  \fill (12,0) circle (5pt); \fill (14,0) circle (5pt); \fill (16,0) circle (5pt); \fill (18,0) circle (5pt);
  \fill (13,1) circle (5pt); \fill (15,1) circle (5pt); \fill (17,1) circle (5pt); \fill (19,1) circle (5pt);
  \fill (14,2) circle (5pt); \fill (16,2) circle (5pt); \fill (18,2) circle (5pt); \fill (20,2) circle (5pt);
  \fill (15,3) circle (5pt); \fill (17,3) circle (5pt); \fill (19,3) circle (5pt);
  \fill (14,4) circle (5pt); \fill (16,4) circle (5pt); \fill (18,4) circle (5pt); 
  \fill (13,5) circle (5pt); \fill (15,5) circle (5pt); \fill (17,5) circle (5pt); 

  \draw[gray!70] (-4,-1) -- (3,6) -- (10,-1) -- (17,6) -- (24,-1) -- (31,6) -- (38,-1) -- (40,1);
  \draw[gray!70] (-4,5) -- (-3,6) -- (4,-1) -- (11,6) -- (18,-1) -- (25,6) -- (32,-1) -- (39,6) -- (40,5);
   \begin{scope}[semithick]
    \draw (-5,-1) -- (41,-1); \draw (-5, 6) -- (41, 6);
    \draw (-3,6) -- ( 0,3) -- (-4,-1);
    \draw ( 3,6) -- ( 7,2) -- ( 4,-1);
    \draw (11,6) -- (14,3) -- (10,-1);
    \draw (17,6) -- (21,2) -- (18,-1);
    \draw (25,6) -- (28,3) -- (24,-1);
    \draw (31,6) -- (35,2) -- (32,-1);
    \draw (39,6) -- (40,5); \draw (38,-1) -- (40,1);
  \end{scope}

  \node at (17,2) {\Sst{\cZ^0}};  \node at (14,5) {\Sst{\cY^0}};  \node at (15,0) {\Sst{\cX^0}};
  \begin{scope}[gray!60]
     \node at (11,2) {\Sst{\cZ^1}};  \node at ( 8,5) {\Sst{\cY^1}};  \node at ( 7,0) {\Sst{\cX^1}};
     \node at (25,2) {\Sst{\cZ^1}};  \node at (22,5) {\Sst{\cY^1}};  \node at (21,0) {\Sst{\cX^1}};
     \node at (31,2) {\Sst{\cZ^0}};  \node at (28,5) {\Sst{\cY^0}};  \node at (29,0) {\Sst{\cX^0}};
     \node at ( 3,2) {\Sst{\cZ^0}};  \node at ( 0,5) {\Sst{\cY^0}};  \node at ( 1,0) {\Sst{\cX^0}};
     \node at (-3,2) {\Sst{\cZ^1}};
     \node at (39,2) {\Sst{\cZ^1}};  \node at (36,5) {\Sst{\cY^1}};  \node at (35,0) {\Sst{\cX^1}};
  \end{scope}
\end{tikzpicture}
  \caption{ \label{fig:type-A}
     Above: $\Db(\kk A_{6}) \cong \thick{}{Z}\orth \embed \Db(\Lambda(2,5,2))$. \newline
     Remaining components $\cX^1=\Sigma\cX^0, \cY^1=\Sigma\sY^0, \cZ^1=\Sigma\cZ^0$ not shown. \newline
     Below: AR quiver of $\Db(\kk A_6)$ with its $\Db(\Lambda(2,5,2))$ pieces.
  }
\end{figure}

The silting objects of $\Db(\kk A_{n+m-1})$ are well understood from work of Keller and Vossieck in \cite{Keller-Vossieck}. We shall now bootstrap their classification to discrete derived categories using the technique of silting reduction Aihara and Iyama in \cite{AI}.

\subsection{Silting reduction}
\label{sec:silting-reduction}

The main technical tool in the classification is the following result of Aihara and Iyama in \cite{AI}:

\begin{theorem}[Silting reduction {\cite[Theorem~2.37]{AI}}] \label{thm:silt-red}
Let $\sD$ be a Krull--Schmidt triangulated category, $\sU\subset\sD$ a thick, contravariantly finite subcategory and $F\colon\sD \to \sD/\sU$ the canonical functor. Then for any silting subcategory $\sN$ of $\sU$, there is an injective map
\[ \{ \text{silting subcategories $\sM$ of $\sD$} \mid \sN \subseteq \sM \} \embed
   \{ \text{silting subcategories of $\sD/\sU$} \}, \quad \sM \mapsto F(\sM) .
\]
If $\sU$ is functorially finite in $\sD$, then the map is bijective.
\end{theorem}

We are working towards an explicit description of the inverse map $G$ in Proposition~\ref{prop:cocones}.
The subcategory  $\sB\coloneqq \susp{\Sigma \sN}$ is the `co-aisle' of the co-t-structure associated to $\sN$ (see Theorem~\ref{thm:koenig-yang}) and thus covariantly finite in $\sU$.
Putting this together with $\sU$ being functorially finite in $\sD$, it gives rise to a co-t-structure $(\sA,\sB)$ in $\sD$, where $\sA \coloneqq {}\orth\sB$. Now let $\sK$ be a silting subcategory of $\sU\orth$ and consider the approximation triangle of $K \in \sK$ with respect to the co-t-structure $(\sA,\sB)$,
\[
\tri{A_K}{K}{B_K}
\]
with $A_K \in \sA$ and $B_K \in \sB$. In their proof of Theorem~\ref{thm:silt-red} in \cite{AI}, Aihara and Iyama show that $G(\sK) \coloneqq \add(\sN \cup \{A_K \mid K \in \sK\})$ is a silting subcategory of $\sD$.

\begin{definition} \label{map-G}
Assume the notation and hypotheses of Theorem~\ref{thm:silt-red} above. Given a silting subcategory $\sN$ of $\sU$, by abuse of notation we write $G_{\sN}$ for the map $G_{\sN}\colon \sU\orth \to \sD$, which for $V \in \sU\orth$, is defined by
\[
\trilabels{G_{\sN}(V)}{V}{B_V}{}{f_V}{},
\]
where $f_V\colon V \to B_V$ is a \emph{minimal} left $\sB$-approximation of $V$. Note that here, in contrast to elsewhere in this paper, we require that the approximation is minimal to ensure well-definedness of the map $G_{\sN}$. Furthermore, we stress here that $G_{\sN}$ is a map not a functor.
\end{definition}

In light of Proposition~\ref{prop:embedding-A}, the natural choice for a functorially finite thick subcategory to which we can apply Theorem~\ref{thm:silt-red} is $\sZ$ for some $Z$ in the $\cZ$ components. For silting reduction to work, we first need to establish that any silting subcategory of $\Db(\LLambda)$ contains an indecomposable object from the $\cZ$ components. The following lemma is a small generalisation of the statement we need, which we specialise in the subsequent corollary. Simple-minded collections (see \cite{Koenig-Yang} for the definition) are also an important focus of current research. Therefore, while we do not use them in this paper, it is useful to highlight in the corollary below that the following lemma also applies to them.

\begin{lemma}
If $\sM$ is a subcategory of $\Db(\Lambda)$ such that $\thick{}{\sM}=\Db(\Lambda)$, then $\sM$ contains an indecomposable object from the $\cZ$ components.
\end{lemma}

\begin{corollary} \label{cor:silting-in-Z}
Any silting subcategory of $\Db(\Lambda)$ and any simple-minded collection in $\Db(\Lambda)$ contain objects from some $\cZ$ component.
\end{corollary}

\begin{proof}[Proof of lemma]
By Lemma~\ref{lem:thick_AR-components}, the additive closure of the $\cX$ components of $\Db(\Lambda)$ is a thick subcategory of $\Db(\Lambda)$, and likewise for the additive closure of the $\cY$ components. Furthermore, these two subcategories are fully orthogonal by Propositions~\ref{prop:X-hammocks} and \ref{prop:Y-hammocks}, so that their sum is a thick subcategory of $\Db(\Lambda)$ as well.
Therefore we cannot have $\sM\subset\cX\oplus\cY$ as that would force $\Db(\Lambda) = \thick{}{\sM} = \cX\oplus\cY$, a contradiction.
\end{proof}

Theorem~\ref{thm:silt-red} coupled with Proposition~\ref{prop:embedding-A} tells us that all silting objects in $\Db(\Lambda)$ containing $Z$ can be obtained by lifting silting objects in $\sZ\orth \simeq \Db(\kk A_{n+m-1})$ back up to $\Db(\Lambda)$. In other words, any silting object in $\Db(\Lambda)$ can be described by a pair $(Z,M')$ consisting of an indecomposable object $Z\in\cZ$ and a silting object $M'\in\sZ\orth \simeq \Db(\kk A_{n+m-1})$.

We now make a brief expository digression explaining Keller and Vossieck's classification of silting subcategories of $\Db(\kk A_t)$, from which the silting subcategories of $\Db(\LLambda)$ can be `glued'.

\subsection{Classification of silting objects in Dynkin type $A$}
\label{sec:classification-A}

Consider the following diagram of the  AR quiver of $\Db(\kk A_t)$ with coordinates $(g,h)$ with $g \in \IZ$ and $h\in\{1,\ldots,t\}$.
\[ \xymatrix@M=0.1em@H=0.0ex@W=0.4ex@!0{
\cdots \ar[dr] &             & \arft{-1,3} &            & \arft{0,3} &            & \arft{1,3} &            & \arft{2,3} &            & \cdots \\
               & \arfc{-1,2} &             & \arfc{0,2} &            & \arfc{1,2} &            & \arfc{2,2} &            & \arfc{3,2} &        \\
\cdots \ar[ur] &             & \arfb{0,1}  &            & \arfb{1,1} &            & \arfb{2,1} &            & \arfb{3,1} &            & \cdots
} \]
Given an indecomposable object $U \in \Db(\kk A_t)$ we write its coordinates as $(g(U),h(U))$.

Following \cite{Keller-Vossieck}, a quiver $Q=(Q_0,Q_1)$ is called an \emph{$\AAA_t$-quiver} if $|Q_0|=t$, its underlying graph is a tree, and $Q_1$ decomposes into a disjoint union $Q_1 = Q_{\alpha} \cup Q_{\beta}$ such that at any vertex at most one arrow from $Q_{\alpha}$ ends, at most one arrow from $Q_{\alpha}$ starts, at most one arrow from $Q_{\beta}$ ends and at most one arrow from $Q_{\beta}$ starts.
One should think of an $\AAA_t$-quiver as a `gentle tree quiver', where gentle is used in the sense of gentle algebras.

We define maps $s_\alpha,e_\alpha,s_\beta,e_\beta\colon Q_0\to\IN$ by
\begin{align*}
s_\alpha(x) & \coloneqq \#\{y \in Q_0 \mid \textrm{the shortest walk from } x \textrm{ to } y \textrm{ starts with an arrow in } Q_{\alpha}\}; \\
e_\alpha(x) & \coloneqq \#\{y \in Q_0 \mid \textrm{the shortest walk from } y \textrm{ to } x \textrm{ ends with an arrow in } Q_{\alpha}\}.
\end{align*}
The functions $s_\beta$ and $e_\beta$ are defined analogously.
With these maps, there is precisely one map
 $\phi_Q \coloneqq (g_Q(x),h_Q(x)) \colon Q_0 \to (\IZ A_t)_0$,
where $g_Q$ and $h_Q$ correspond to the coordinates in the AR quiver of $\Db(\kk A_t)$, such that $h_Q(x) = 1 + e_\alpha(x) + s_\beta(x)$ and $g_Q(y) = g_Q(x)$ for each arrow $x \too y$ in $Q_\alpha$, and $g_Q(y) = g_Q(x) + e_\alpha(x) +s_\alpha(x) +1$ for each arrow $x \too y$ in $Q_\beta$, and finally normalised by  $\min_{x \in Q_0}\{g_Q(x)\} = 0$.

By abuse of notation we identify the object $T_Q \coloneqq \phi_Q(Q_0)$ with the direct sum of the indecomposables lying at the corresponding coordinates. This map gives rise to the following classification result.

\begin{theorem}[\cite{Keller-Vossieck}, Section 4]
The assignment $Q \mapsto T_Q$ induces a bijection between isomorphism classes of $\AAA_t$-quivers and tilting objects $T$ in $\Db(\kk A_t)$ satisfying the condition $\min\{g(U) \mid U \textrm{ is an indecomposable summand of } T\} =0$.
\end{theorem}

Note that in Dynkin type $A_t$, the summands of any tilting object $T=\bigoplus_{i=1}^t T_i$ can be re-ordered to give a strong, full exceptional sequence $\{T_1,\cdots,T_t\}$, see \cite[Section~5.2]{Keller-Vossieck}.
We now have the following classification of silting objects in $\Db(\kk A_t)$.

\begin{theorem}[\cite{Keller-Vossieck}, Theorem 5.3] \label{thm:A-silting}
Let $T=T_1\oplus\cdots\oplus T_t$ be a tilting object in $\Db(\kk A_t)$ whose summands form an exceptional collection. Let $p\colon \{1,\ldots,t\} \to \IN$ be a weakly increasing function. Then $\Sigma^{p(1)} T_1\oplus\cdots\oplus\Sigma^{p(t)} T_t$ is a silting object in $\Db(\kk A_t)$. Moreover, all silting objects of $\Db(\kk A_t)$ occur in this way.
\end{theorem}

The machinery above is slightly technical, so we give a quick example of the classification of tilting (and hence silting) objects in $\Db(\kk A_3)$.

\begin{example}[Classification of tilting objects in $\Db(\kk A_3)$] \label{ex:A_3}
When $t=3$, up to isomorphism there are the following possible $A_3$-quivers:
\begin{align*}
& 1 \rightlabel{\alpha} 2 \rightlabel{\alpha} 3, \quad  1 \rightlabel{\alpha} 2 \rightlabel{\beta} 3, \quad 1 \rightlabel{\alpha} 2 \leftlabel{\beta}   3, \\
& 1 \leftlabel{\alpha}  2 \rightlabel{\beta}  3, \quad  1 \rightlabel{\beta}  2 \rightlabel{\beta} 3, \quad 1 \rightlabel{\beta}  2 \rightlabel{\alpha} 3.
\end{align*}
Computing the $\phi_Q$ for each of the above quivers gives the following, where each $3$-tuple denotes $(\phi_Q(1),\phi_Q(2),\phi_Q(3))$:
\begin{align*}
& ((0,1),(0,2),(0,3)), \quad ((0,1),(0,3),(2,1)), \quad ((1,1),(1,2),(0,3)), \\
& ((0,3),(0,2),(1,1)), \quad ((0,3),(1,2),(2,1)), \quad ((0,3),(2,1),(2,3)).
\end{align*}
We indicate the corresponding tilting objects in the following sketch:

\begin{center}
\begin{tikzpicture}[scale=0.35]
  \pgfmathsetmacro{\X}{14}    
  \pgfmathsetmacro{\Y}{4.5}   
  
  \foreach \column in {0,1,2} \foreach \row in {0,1}
     {
     \pgfmathsetmacro{\XX}{\column*\X} 
     \pgfmathsetmacro{\YY}{\row*\Y}
     \foreach \x in {0,1,...,8} \foreach \y in {0,1,2}
        { \ifnumodd{\x + \y}{}{ \fill[gray!50] (\x+\XX,\y+\YY) circle (6pt); } }
     \begin{scope}[gray!70, shift={(\XX,\YY)}]
       \draw (1,-0.5) -- (7,-0.5) -- (4,2.75) -- cycle;
     \end{scope} 
     }   
  \begin{scope}[black!70] 
    \begin{scope}[shift={(   0,\Y)}]  \fill (2,0) circle (6pt); \fill (3,1) circle (6pt); \fill (4,2) circle (6pt); \dnr{1}  \end{scope}
    \begin{scope}[shift={(  \X,\Y)}]  \fill (2,0) circle (6pt); \fill (4,2) circle (6pt); \fill (6,0) circle (6pt); \dnr{2}  \end{scope}
    \begin{scope}[shift={(2*\X,\Y)}]  \fill (4,0) circle (6pt); \fill (5,1) circle (6pt); \fill (4,2) circle (6pt); \dnr{3}  \end{scope}
                                      \fill (4,0) circle (6pt); \fill (3,1) circle (6pt); \fill (4,2) circle (6pt); \dnr{4}
    \begin{scope}[shift={(  \X, 0)}]  \fill (6,0) circle (6pt); \fill (5,1) circle (6pt); \fill (4,2) circle (6pt); \dnr{5}  \end{scope}
    \begin{scope}[shift={(2*\X, 0)}]  \fill (4,2) circle (6pt); \fill (6,0) circle (6pt); \fill (8,2) circle (6pt); \dnr{6}  \end{scope}
  \end{scope}
\end{tikzpicture}
\end{center}

\noindent
In each sketch the triangle depicts the standard heart for the quiver $1 \longleftarrow 2 \longleftarrow 3$ whose indecomposable projectives have coordinates $(0,1),(0,2),(0,3)$. These are precisely the tilting objects having an indecomposable summand $U$ with minimal $g(U)=0$. In particular, these are precisely the exceptional sequences in $\Db(\kk A_3)$ containing one of $P(i)$ for $1 \leq i \leq 3$ as a least element.

To obtain all tilting objects (up to suspension), we next consider those for which there exists an indecomposable summand $U$ with minimal $g(U) = 1$. These correspond precisely to $\tau\inv$ applied to each of the diagrams $\tns{1}$ to $\tns{6}$. Observe that
 $\tau\inv \tns{5} = \Sigma\tns{1}$ and $\tau\inv \tns{6} = \Sigma\tns{2}$.
Therefore, up to suspension, we pick up only four more tilting objects. Next we consider those for which there exists an indecomposable summand $U$ with minimal $g(U) =2$, which correspond precisely to $\tau^{-2}$ applied to each of the diagrams $\tns{1}$ to $\tns{6}$. We have $\tau^{-2} \tns{3} = \Sigma \tns{3}$, $\tau^{-2} \tns{4} = \Sigma \tns{4}$, $\tau^{-2} \tns{5} = \Sigma\tau\inv \tns{1}$, and $\tau^{-2} \tns{6} = \Sigma\tau\inv \tns{2}$, which leaves, up to suspension, only $\tau^{-2} \tns{1}$ and $\tau^{-2} \tns{2}$ as new tilting objects.  Continuing in this way, one sees that, up to suspension, these are all tilting objects. Hence, there are twelve tilting objects in $\Db(\kk A_3)$ up to suspension:

\smallskip
\noindent
\resizebox{\textwidth}{!}{
$ \begin{array}{@{} r@{\:=\:}l r@{\:=\:}l r@{\:=\:}l @{}}
 \tns{1}          & P(1)\oplus P(2)\oplus P(3),        & \tns{2}          & P(1)\oplus P(3) \oplus S(3),              & \tns{3}          & P(3)\oplus I(2)\oplus S(2), \\
 \tns{4}          & P(2)\oplus P(3)\oplus S(2),        & \tns{5}          & P(3)\oplus I(2)\oplus S(3),               & \tns{6}          & P(3)\oplus S(3)\oplus \Sigma S(2), \\
 \tau\inv\tns{1}  & S(2)\oplus I(2)\oplus \Sigma P(1), & \tau\inv\tns{2}  & S(2)\oplus \Sigma P(1)\oplus \Sigma P(3), & \tau^{-1}\tns{3} & \Sigma P(1)\oplus \Sigma P(2)\oplus S(3), \\
 \tau^{-1}\tns{4} & I(2)\oplus \Sigma P(1)\oplus S(3), & \tau^{-2}\tns{1} & S(3)\oplus \Sigma P(2)\oplus \Sigma S(2), & \tau^{-2}\tns{2} & S(3)\oplus \Sigma S(2)\oplus \Sigma^2 P(1).
\end{array}
$ }
\end{example}

\subsection{Classification of silting objects for derived-discrete algebras}
\label{sec:classification-dd}

As this section is rather technical, the reader may find it helpful to refer to the detailed example, $\Lambda(2,3,1)$ studied in Section~\ref{sec:example} whilst reading this section.

We first start with some preliminary results regarding the indecomposability of the images of indecomposable objects under the map $G_Z \colon \sZ\orth \to \Db(\LLambda)$ from Definition~\ref{map-G}, where $\sZ = \thick{}{Z}$ for some fixed, arbitrary, indecomposable object $Z \in \ind{\cZ}$.

We first explicitly compute the map $G_Z\colon \ind{\sZ\orth} \to \Db(\LLambda)$ on objects in the case $Z = Z^0_{0,0}$.

\begin{proposition} \label{prop:cocones}
If $r>1$ and $Z = Z^0_{0,0}$, and $G \coloneqq G_{Z_{0,0}}$, then $G(U) = U$ for all but finitely many (up to positive suspension) $U \in \ind{\sZ\orth}$. The exceptions are:
\begin{enumerate}[label=(\arabic*)~]
\item $G(\Sigma^i X^1_{0,j})     = \Sigma^i Z^0_{j+1,0}$ for $0 \leq j < r +m-1$ and $i \geq 0$.
\item $G(\Sigma^i Y^1_{j,0})     = \Sigma^i Z^0_{0,j+1}$ for $0 \leq j < n-r-1$ and $i \geq 0$.
\item $G(\Sigma^i Z^1_{j,0})     = \Sigma^i X^1_{j,-1}$ for $1 \leq j \leq r+m-1$ and $i \geq 0$.
\item $G(\Sigma^i Z^1_{0,j})     = \Sigma^i Y^1_{-1,j}$ for $r-n+1 \leq j \leq -1$ and $i \geq 0$.
\item $G(\Sigma^i Z^1_{-r-m,0})  = \left\{ \begin{array}{ll}
                                       \Sigma^i X^1_{-r-m,-1} & \text{for } 0 \leq i \leq r, \\
                                       \Sigma^i Z^1_{0,n-r} & \text{for } i > r.
                                 \end{array} \right.$
\end{enumerate}
\end{proposition}

\begin{proposition} \label{prop:cocones_r=1}
If $r=1$ and $Z = Z_{0,0}$, and $G \coloneqq G_{Z_{0,0}}$, then $G(U) = U$ for all but finitely many (up to positive suspension) $U \in \ind{\sZ\orth}$. The exceptions are:
\begin{enumerate}[label=(\arabic*)~]
\item $G(\Sigma^i X_{1+m,1+m+j}) = \Sigma^i Z_{j+1,0}$ for $0 \leq j < m$ and $i \geq 0$.
\item $G(\Sigma^i Y_{1-n+j,1-n}) = \Sigma^i Z_{0,j+1}$ for $0 \leq j < n-2$ and $i \geq 0$.
\item $G(\Sigma^i Z_{j,1-n})    = \Sigma^i X_{j,m}$ for $0 < j < m+1$ and $i \geq 0$.
\item $G(\Sigma^i Z_{m+1,j})    = \Sigma^i Y_{-n,j}$ for $2-2n < j < 1-n $ and $i \geq 0$.
\item $G(\Sigma^i Z_{0,1-n})    = \left\{ \begin{array}{ll}
                                       X_{0,m} & \text{for } i=0, \\
                                       \Sigma^i Z_{0,n-1} & \text{for } i > 0.
                                 \end{array} \right.$
\end{enumerate}
\end{proposition}

\begin{proof}[Proof of Propositions~\ref{prop:cocones} and \ref{prop:cocones_r=1}]
We do the calculations for the generic case with $r>1$ in Proposition~\ref{prop:cocones}; those for Proposition~\ref{prop:cocones_r=1} are similar.
The function $G$ is defined via the `co-aisle' of the co-t-structure $(\sA,\sB)$ with $\sB = \susp{\Sigma Z_0}=\add\{\Sigma^i Z_0 \mid i\geq 1\}$. Using Proposition~\ref{prop:Z-hammocks}, one can easily compute $\sA = {}\orth\sB$. If $U \in \sA$, then $G(U) = U$, so examining $\sA \cap \sZ\orth$ gives the list of exceptions above.

 We now compute the cocones $G(U)$ directly using the triangles from Properties~\ref{properties}$(4)$:
\begin{enumerate}[label=(\arabic*)]
\item The relevant triangles here are $\tri{Z^0_{j+1,0}}{X^1_{0,j}}{Z^1_{0,0}}$ for  $0 \leq j < r+m-1$, where we note that $\Sigma Z^0_{j+1,0} = Z^1_{j+1,0}$.
\item Here we have $\tri{Y^1_{j,0}}{Z^1_{0,0}}{Z^1_{0,j+1}}$ for $0 \leq j < n-r-1$, again noting that $Z^0_{0,j+1} = \Sigma^{-1} Z^1_{0,j+1}$.
\item \label{ZX-triangles} The triangles are $\tri{X^1_{j,-1}}{Z^1_{j,0}}{Z^1_{0,0}}$ for $1 \leq j \leq r+m-1$.
\item The triangles are $\tri{Y^1_{-1,j}}{Z^1_{0,j}}{Z^1_{0,0}}$ for $r-n+1 \leq j \leq -1$.
\item When $0 \leq i \leq r$, the relevant triangle belongs with the family in \ref{ZX-triangles} above, and can be computed analogously. However, when $i > r$, we need to take the cocone of the morphism $\Sigma^i Z_{-r-m,0} \to \Sigma^i \big( Z^1_{-r-m,r-n} \oplus Z^1_{0,0} \big)$. We claim that the cone of $Z_{-r-m,0} \to Z^1_{-r-m,r-n} \oplus Z^1_{0,0}$ is $Z^1_{0,r-n}$. To show this, we compute the cocone of $ Z^1_{-r-m,r-n} \oplus Z^1_{0,0} \to Z^1_{0,r-n}$ via the following octahedron:
\[
\xymatrix{
                    & Z^1_{-r-m,r-n} \ar@{=}[r] \ar[d]                 & Z^1_{-r-m,r-n} \ar[d]  \\
C \ar@{=}[d] \ar[r] & Z^1_{-r-m,r-n} \oplus Z^1_{0,0} \ar[r] \ar[d] & Z^1_{0,r-n} \ar[d] \\
C \ar[r]            & Z^1_{0,0} \ar[r]                        & \Sigma X^1_{-r-m,-1},
}
\]
where the second column is the split triangle, and the third column is a standard triangle from Properties~\ref{properties}$(4)$.
The triangle forming the bottom row is none other than $\tri{X^1_{-r-m,-1}}{Z^1_{-r-m,0}}{Z^1_{0,0}}$, which computes the cocone $C = Z^1_{-r-m,0}$ as claimed. \qedhere
\end{enumerate}
\end{proof}

\begin{corollary} \label{cor:indecomp}
Let $Z \in \ind{\cZ}$ be arbitrary.
If $U \in \sZ\orth$ is indecomposable then $G_Z(U)$ is also indecomposable.
\end{corollary}

\begin{proof}
Since the autoequivalences $\TTT_{\cX}$, $\TTT_{\cY}$ and $\Sigma$ act transitively on the $\cZ$ components, it is sufficient to see this for $Z = Z^0_{0,0}$. This is clear from the computations in (the proof of) Proposition~\ref{prop:cocones} above.
\end{proof}

Silting objects in $\Db(\Lambda)$ correspond to pairs $(Z,M')$, where $Z\in \ind{\cZ}$ and $M'$ is a silting object of $\sZ\orth \simeq \Db(\kk A_{n+m-1})$. However, a silting object in $\Db(\Lambda)$ may have more than one indecomposable summand in the $\cZ$ components. Thus, using silting reduction, we will obtain multiple descriptions of the same object. To rectify this problem, we classify silting objects for which $Z\in \ind{\cZ}$ is minimal with respect to a total order on $\ind{\cZ}$ defined as follows.
Let $Z\in \ind{\cZ^i}$ and $Z'\in \ind{\cZ^j}$ and define
\begin{align*}
Z \preceq Z' \iff &
\left\{
\begin{array}{r@{\:}c@{\:}l l}
\ray{\Sigma^{j-i} Z}         &\leq  & \ray{Z'}    & \text{if } i < j; \\
\ray{\tau^{-1}\Sigma^{j-i} Z} & \leq & \ray{Z'}    & \text{if } i > j; \\
\coray{Z}                   &\leq  & \coray{Z'}  & \text{if } i=j \text{ and } \ray{Z} = \ray{Z'}; \\
\ray{Z}                     & <    & \ray{Z'}    & \text{if } i=j \text{ and } \ray{Z} \neq \ray{Z'},
\end{array}
\right.
\end{align*}
where $\ray{Z^a_{ij}} \leq \ray{Z^a_{kl}}$ if and only if $i \leq k$ and $\coray{Z^a_{ij}} \leq \coray{Z^a_{kl}}$ if and only if $j \leq l$.
Equivalently, for $Z\in \ind{\cZ^i}$, the total order is defined by successor sets,
\begin{align*}
 \{ \tilde Z \in \ind{\cZ} \mid Z \preceq Z' \} =
    \rayfrom{Z} & \cup \raythrough{\corayfrom{\tau\inv Z}}
                  \cup \raythrough{\corayfrom{\Sigma^{\{i+1,\ldots,r-1\}} Z}} \\
                & \cup \raythrough{\corayfrom{\tau^{-1}\Sigma^{\{0,\ldots,i-1\}} Z}} .
\end{align*}

\noindent
The following diagrams indicate the indecomposables $Z\in\cZ$ with $Z\preceq Z'$:

\begin{center}
  \begin{tikzpicture}[scale=0.3]

    \foreach \x in {-5,...,5} \foreach \y in {-5,-4,...,-1,1,2,...,5}
       { \fill[gray!60] (\x,\y) circle (2pt); }
    \fill[gray!60] (5,0) circle (2pt); 
    \draw[fill,gray!50,line width=1pt] (-5,-5) -- (5,5) -- (-5,5) -- cycle;
    \filldraw(0,0) circle (4pt) node[right] {$\Sigma^{i-j}Z'$};	
    \node at (0,-6) {$\cZ^i, i<j$};

  \begin{scope}[shift={(15,0)}]
    \foreach \x in {-5,...,5} \foreach \y in {-5,-4,...,-1,1,2,...,5}
       { \fill[gray!60] (\x,\y) circle (2pt); }
    \fill[gray!60] (5,0) circle (2pt); \fill[gray!60] (4,0) circle (2pt); \fill[gray!60] (3,0) circle (2pt); 
    \draw[fill,gray!50,line width=1pt] (-5,-5) -- (0,0) -- (-1,1) -- (3,5) -- (-5,5) -- cycle;
    \filldraw(0,0) circle (4pt) node[right] {$Z'$};	
    \node at (0,-6) {$\cZ^j$};
  \end{scope}

  \begin{scope}[shift={(30,0)}]
    \foreach \x in {-5,...,5} \foreach \y in {-5,-4,...,-1,1,2,...,5}
       { \fill[gray!60] (\x,\y) circle (2pt); }
    \fill[gray!60] (5,0) circle (2pt); 
    \draw[fill,gray!50,line width=1pt] (-5,-4) -- (4,5) -- (-5,5) -- cycle;
    \filldraw(-1,0) circle (4pt) node[right] {$\tau\Sigma^{i-j}Z'$};
    \node at (0,-6) {$\cZ^i, i>j$};
  \end{scope}

\end{tikzpicture}
\end{center}

\begin{lemma} \label{lem:total-order}
The relation $\preceq$ defines a total order on $\ind{\cZ}$.
\end{lemma}

\begin{proof}
\textit{Anti-symmetry:} Suppose $Z \preceq Z'$ and $Z' \preceq Z$ with $Z\in \ind{\cZ^i}$ and $Z'\in \ind{\cZ^j}$. If $i=j$, then anti-symmetry is clear. For a contradiction, suppose $i<j$. Then $\ray{\Sigma^{j-i}Z} \leq \ray{Z'}$ and $\ray{\tau^{-1}\Sigma^{i-j}Z'}\leq \ray{Z}$. In particular, it follows that $\ray{\tau^{-1} Z'} \leq \ray{\Sigma^{j-i} Z} \leq \ray{Z'}$, which is a contradition, since $\ray{\tau^{-1}Z'} > \ray{Z'}$. The same argument works when $i>j$.

\textit{Transitivity:} Suppose $Z \preceq Z'$ and $Z' \preceq Z''$ with $Z\in\ind{\cZ^i}$, $Z'\in \ind{\cZ^j}$ and $Z''\in \ind{\cZ^k}$. One simply analyses the different possibilities for $i$, $j$ and $k$. We do the case $i>j$ and $j<k$; the rest are similar. The first inequality means that $\ray{\tau^{-1}\Sigma^{j-i}Z} \leq \ray{Z'}$ and the second inequality means that $\ray{\Sigma^{k-j}Z'} \leq \ray{Z''}$. There are two subcases: first assume $i\leq k$. In this case, apply $\tau \Sigma^{k-j}$ to the condition arising from the first inequality and combine this with the second inequality to get
$\ray{\Sigma^{k-i}Z}\leq \ray{\tau \Sigma^{k-j} Z'} < \ray{\Sigma^{k-j} Z'} \leq \ray{Z''}$.
Now assume that $i > k$ and apply $\Sigma^{k-j}$ to the condition arising from the first inequality and combine with the second inequality to get
$\ray{\tau^{-1}\Sigma^{k-i}Z} \leq \ray{\Sigma^{k-j}Z'} \leq \ray{Z'}$.

\textit{Totality:} Suppose $Z\in \ind{\cZ^i}$ and $Z'\in \ind{\cZ^j}$. If $i=j$ then it is clear that either $Z \preceq Z'$ or $Z' \preceq Z$. Now suppose $i<j$. If $\ray{\Sigma^{j-i}Z} \leq \ray{Z'}$ then $Z \preceq Z'$ and we are done, so suppose that $\ray{\Sigma^{j-i}Z} > \ray{Z'}$. Then it follows that $\ray{\Sigma^{i-j}Z'} < \ray{Z}$, in which case, because $\tau^{-1}$ increases the index of the ray by $1$, one gets $\ray{\tau^{-1}\Sigma^{i-j}Z'}\leq \ray{Z}$ and hence $Z' \preceq Z$. A similar argument holds in the case $i>j$. Thus, $\preceq$ is indeed a total order.
\end{proof}

Using Corollary~\ref{cor:indecomp}, we now ensure we identify each silting subcategory of $\sM$ of $\Db(\Lambda)$ as precisely one pair $(Z,M')$, with $\sM'$ a silting object of $\sZ\orth \simeq \Db(\kk A_{n-m+1})$  by insisting that $Z \preceq Z'$ for each $Z' \in \ind{\cZ} \cap \add{M'}$.

\begin{definition}
We define the following additive subcategory of $\sD$:
\[
\sZ\orth_{\prec} \coloneqq \add \{U \in \ind{\sZ\orth} \mid G_Z(U) \in \cZ \text{ and } G_Z(U) \prec Z \}.
\]
\end{definition}

With the identification of $\Db(\kk A_{n+m-1})$ in $\Db(\LLambda)$ of Remark~\ref{rem:explicit}, using Proposition~\ref{prop:cocones}, we now give an explicit description of the additive subcategory $\sZ\orth_{\prec}$.

Recall from Proposition~\ref{prop:embedding-A} that $\sZ\orth \simeq \Db(\kk A_{n+m-1})$. Let $\Gamma \coloneqq \kk A_{n+m-1}$ be the path algebra of the $A_{n+m-1}$ quiver with the linear orientation:
\[
\xymatrix{1 & \ar[l] 2 & \ar[l]  3 & \ar@{.}[l] n+m-2 & \ar[l] n+m-1.}
\]
Consider the unique $\Sigma^i\mod{\Gamma}\subset\Db(\Gamma)$ that contains the indecomposable objects in $\sZ\orth \cap \cZ$ admitting non-zero morphisms to $Z$. In Lemma~\ref{lem:forbidden} below, when we specify $\mod{\Gamma}$, we shall mean precisely this copy sitting inside $\Db(\kk A_{n+m-1})$.

\begin{lemma} \label{lem:forbidden}
With the conventions described above, the additive subcategory $\sZ\orth_{\prec}$ is
\[
\sZ\orth_{\prec} = \add\{\Sigma^i \sA \mid i \leq -r\} \cup \add\{\Sigma^i \sB \mid 1-r \leq i < 0 \} \cup \add(\sC),
\]
where the sets of indecomposables $\sA$, $\sB$ and $\sC$ are defined as follows:
\begin{align*}
\sA & \coloneqq \{X \in \mod{\Gamma} \mid \Hom_{\Gamma}(P(r+m),X) \neq 0\}; \\
\sB & \coloneqq \sA \cap \{X \in \mod{\Gamma} \mid \Hom_{\Gamma}(P(r+m+1),X) \neq 0\} && \text{(empty when $n-r=1$);} \\
\sC & \coloneqq \{P(r+m-1),\ldots,P(n+m-1)\}                                        && \text{(empty when $n-r=1$),}
\end{align*}
where $P(i)$ is the indecomposable projective at vertex $i$ for the path algebra $\Gamma = \kk A_{n+m-1}$.
\end{lemma}

\begin{proof}
This is a direct computation using Proposition~\ref{prop:cocones}, the total order on the indecomposable objects of the $\cZ$ components of Lemma~\ref{lem:total-order}, and the identification of the subcategory from Remark~\ref{rem:explicit}.
\end{proof}

To illustrate Lemma~\ref{lem:forbidden}, we sketch the additive subcategory $\sZ\orth_{\prec}$ in the case of $\Lambda(2,5,2)$ and $Z = Z^0_{0,0}$ below.

\vspace{-2ex}
\begin{center}
\begin{tikzpicture}[scale=0.3] 

  \foreach \x in {-4,...,40} \foreach \y in {0,...,5}
     { \ifnumodd{\x+\y}{}{ \fill[gray!30] (\x,\y) circle (5pt); }}
 

  \fill (18,0) circle (5pt); \fill (19,1) circle (5pt); \fill (20,2) circle (5pt);
  \fill (17,1) circle (5pt); \fill (18,2) circle (5pt); \fill (19,3) circle (5pt);
  \fill (16,2) circle (5pt); \fill (17,3) circle (5pt); \fill (18,4) circle (5pt); 
  \fill (15,3) circle (5pt); \fill (16,4) circle (5pt); \fill (17,5) circle (5pt);  
  \fill (4,0) circle (5pt); \fill (5,1) circle (5pt); \fill (6,2) circle (5pt);
  \fill (3,1) circle (5pt); \fill (4,2) circle (5pt); \fill (5,3) circle (5pt);
  \fill (2,2) circle (5pt); \fill (3,3) circle (5pt); \fill (4,4) circle (5pt);
  \fill (1,3) circle (5pt); \fill (2,4) circle (5pt); \fill (3,5) circle (5pt);
  \fill ( 8,2) circle (5pt); \fill ( 9,1) circle (5pt); \fill (10,0) circle (5pt);
  \fill ( 9,3) circle (5pt); \fill (10,2) circle (5pt); \fill (11,1) circle (5pt);
  \fill (10,4) circle (5pt); \fill (11,3) circle (5pt); \fill (12,2) circle (5pt);
  \fill (11,5) circle (5pt); \fill (12,4) circle (5pt); \fill (13,3) circle (5pt);
  \fill (-3,5) circle (5pt); \fill (-2,4) circle (5pt); \fill (-1,3) circle (5pt);
  \fill (-4,4) circle (5pt); \fill (-3,3) circle (5pt); \fill (-2,2) circle (5pt);
                             \fill (-4,2) circle (5pt); \fill (-3,1) circle (5pt);
                                                        \fill (-4,0) circle (5pt);
  \node at (8,7) {$\overbrace{\phantom{\rule{18.5em}{0.5em}}}^{\Sigma^{\leq -r}\sA}$};

  \fill (23,1) circle (5pt); \fill (24,2) circle (5pt); \fill (25,3) circle (5pt); \fill (26,4) circle (5pt);
  \fill (24,0) circle (5pt); \fill (25,1) circle (5pt); \fill (26,2) circle (5pt); \fill (27,3) circle (5pt);
  \node at (25,7) {$\overbrace{\phantom{\rule{3em}{0.5em}}}^{\Sigma^{-1}\sB}$};

  \fill (30,4) circle (5pt); \fill (31,5) circle (5pt);
  \node at (30.5,7) {$\overbrace{\phantom{\rule{0.5em}{0.5em}}}^{\sC}$};

  \begin{scope}[semithick]
    \draw (-5,-1) -- (41,-1); \draw (-5, 6) -- (41, 6);
    \draw (-5,0) -- (-4,-1) -- ( 3,6) -- (10,-1) -- (17,6) -- (24,-1) -- (31,6) -- (38,-1) -- (41,2);
  \end{scope}
  \begin{scope}[very thin]
    \draw (-5,4) -- (-3,6) -- (4,-1) -- (11,6) -- (18,-1) -- (25,6) -- (32,-1) -- (39,6) -- (41,4);
  \end{scope}

\end{tikzpicture}
\end{center}

\bigskip\noindent
We summarise this discussion in the following proposition, and obtain the main theorem of the section as a corollary.

\begin{proposition}
Suppose $Z \in \ind{\cZ}$ and write $\sZ = \thick{\Db(\Lambda)}{Z}$. Then there is a bijection between
\begin{enumerate}[label=(\arabic*)~]
\item Silting subcategories $\sM$ of $\Db(\Lambda)$ with $Z \in \sM$ and $Z \preceq \ind{\cZ}\cap\sM$.
\item Silting subcategories $\sN$ of $\sZ\orth$ with $\sN \cap \sZ\orth_{\prec} = \varnothing$.
\end{enumerate}
\end{proposition}

\begin{theorem} \label{thm:silting-classification}
In $\Db(\LLambda)$ there are bijections between
\begin{enumerate}[label=(\arabic*)~]
\item Pairs $(Z,\sN)$ where $Z \in \ind{\cZ}$ and $\sN$ is a silting subcategory of $\Db(\kk A_{m+n-1})$ containing no objects in the additive subcategory $\sZ\orth_{\prec}$.
\item Silting subcategories of $\Db(\LLambda)$.
\item Bounded t-structures in $\Db(\LLambda)$.
\item Bounded co-t-structures in $\Db(\LLambda)$.
\end{enumerate}
\end{theorem}


\section{A detailed example: $\Lambda(2,3,1)$}
\label{sec:example}

\noindent
In this section we examine the algebra $\Lambda(2,3,1)$ in detail. Let $Z = Z^0_{0,0}$ and write $\sZ=\thick{}{Z}$. Take the convention for homological degree as in Lemma~\ref{lem:forbidden}. With this convention, we identify the indecomposable objects in $\sZ\orth$ and of $\Db(\kk A_3)$ as follows:
\[ \xymatrix@=3pt{
         &           & Z^0_{0,-1} &            &            &   &        &   &      &      & P(3) &      &      \\
         & X^0_{0,1}  &           & Z^0_{1,-1}  &            &   &\mapsto &   &      & P(2) &      & I(2) &      \\
X^0_{0,0} &           & X^0_{1,1}  &            & Z^0_{2,-1}  &   &        &   & P(1) &      & S(2) &      & S(3)
} \]
Using Lemma~\ref{lem:forbidden}, Theorem~\ref{thm:A-silting} and the explicit calulation of the tilting objects, up to suspension, in Example~\ref{ex:A_3}, we compute the twelve families of silting objects in $\Db(\kk A_3)$ that lift to silting objects in $\Db(\Lambda(2,3,1))$ containing $Z^0_{0,0}$ as the minimal indecomposable summand in the $\cZ$ components. The results of this computation are presented in Table~\ref{table:2-3-1}.

\begin{table}[t]
\begin{center}
\begin{tabular}{r@{$\:\oplus\:$}c@{$\:\oplus\:$}l @{\qquad} l} \toprule
\multicolumn{3}{c}{tilting object in $\kk A_3$ \phantom{xxx}} & silting family in $\Lambda(2,3,1)$ \\
$T_1$  & $T_2$ & $T_3$           & $\Sigma^iT_1 \oplus \Sigma^jT_2 \oplus \Sigma^kT_3$ \\ \midrule
$P(1)$ & $P(2)$ & $P(3)$        & $j\geq i$ and $k\geq \max\{j,-1\}$ \\
$P(1)$ & $P(3)$ & $S(3)$        & $k\geq j \geq \max\{i,-1\}$ \\
$P(2)$ & $S(2)$ & $P(3)$        & $j\geq i$ and $k\geq \max\{i,-1\}$ \\
$S(2)$ & $P(3)$ & $I(2)$        & $j\geq -1$ and $k\geq \max\{i,j,-1\}$ \\
$P(3)$ & $I(2)$ & $S(3)$        & $k\geq j\geq i\geq -1$ \\
$P(3)$ & $S(3)$ & $\Sigma S(2)$ & $k\geq j\geq i\geq -1$ \\
$S(2)$ & $I(2)$ & $\Sigma P(1)$ & $k\geq j\geq \max\{i,-1\}$ \\
$I(2)$ & $S(3)$ & $\Sigma P(1)$ & $k\geq i$ and $j\geq i\geq -1$ \\
$S(2)$ & $\Sigma P(1)$ & $\Sigma P(3)$   & $j\geq i$ and $k\geq \max\{j,-2\}$ \\
$\Sigma P(1)$ & $ S(3)$ & $\Sigma P(2)$  & $j\geq -1$ and $k\geq \max\{i,j\}$ \\
$S(3)$ & $\Sigma P(2)$ & $\Sigma S(2)$   & $k\geq j\geq i\geq -1$ \\
$S(3)$ & $\Sigma S(2)$ & $\Sigma^2 P(1)$ & $k\geq j\geq i\geq -1$ \\ \bottomrule
\end{tabular}
\end{center}
\caption{The twelve tilting objects in $\kk A_3$ giving rise to the silting objects containing $Z^0_{0,0}$ as the $\preceq$-minimal summand in $\cZ$ for $\Lambda(2,3,1)$.} \label{table:2-3-1}
\end{table}

We make the following observation regarding tilting objects in $\Db(\Lambda(2,3,1)$.

\begin{proposition}
Let $\Lambda = \Lambda(2,3,1)$. Fir any $Z\in \ind{\cZ}$, put $\sZ=\thick{}{Z}$ and $F_Z \colon \Db(\Lambda) \to \sZ\orth \simeq \Db(\kk A_3)$. Then:
\begin{enumerate}
\item There are precisely six tilting objects in $\Db(\Lambda)$ containing $Z$ as a summand.
\item If $T \in \Db(\Lambda)$ is a tilting object containing $Z$ as a summand then $F_Z(T)$ is a tilting object in $\sZ\orth$.
\end{enumerate}
\end{proposition}

\begin{proof}
The proof is a direct computation. Without loss of generality, we may set $Z=Z^0_{0,0}$. Consider the additive subcategory
 $\sT\coloneqq \big(\bigcap_{n\neq 0} {}\orth(\Sigma^n Z)\big) \cap \big(\bigcap_{n\neq 0} (\Sigma^n Z)\orth\big) \cap \sZ\orth$.
The subcategory $\sT$ consists of the thick subcategory $\sZ\orth \cap {}\orth\sZ \simeq \Db(\kk)$, which has just one indecomposable object in each homological degree, together with finitely many indecomposables in homological degrees 0,1 and 2.

Examining the Hom-hammocks from each of the indecomposables in $\sZ\orth \cap {}\orth\sZ$ shows that unless the object lies in homological degree 0, 1 or 2, there is not sufficient intersection with $\sT$ to give rise to a tilting object. Thus we must form tilting objects from only finitely many indecomposables. A detailed analysis of the Hom-hammocks of these finitely many indecomposables gives rise to the six tilting objects obtained from $Z^0_{0,0}$ and the following objects:
\[ \begin{array}{l @{\quad} l @{\quad} l}
 Z^0_{-1,0}\oplus X^1_{-2,-2}\oplus X^0_{0,0}, & X^1_{-2,-2}\oplus X^1_{-1,-1}\oplus X^0_{0,0}, & X^1_{-1,-1}\oplus X^1_{-2,-1}\oplus X^0_{0,0} \\[1ex]
 X^1_{-1,-1}\oplus X^0_{0,0}\oplus X^0_{0,1},  & X^1_{-1,-1}\oplus X^0_{0,1}\oplus X^0_{1,1},  & X^1_{-1,-1}\oplus X^0_{1,1}\oplus Z^0_{1,0}.
\end{array} \]
The second claim can be directly computed.
\end{proof}

Our computations lead us to state the following conjecture:

\begin{conjecture}
For an arbitrary $Z\in \ind{\cZ}$, writing $\Lambda = \LLambda$ and $\sZ=\thick{}{Z}$ and $F_Z \colon \Db(\Lambda) \to \sZ\orth \simeq \Db(\kk A_{n+m-1})$, we have:
\begin{enumerate}
\item There are finitely many tilting objects in $\Db(\Lambda)$ containing $Z$ as a summand.
\item If $T \in \Db(\Lambda)$ is a tilting object containing $Z$ as a summand then $F_Z(T)$ is a tilting object in $\sZ\orth$.
\end{enumerate}
\end{conjecture}

%

\subsection*{An explicit example for a bounded t-structure in $\boldsymbol{\Db(\Lambda(2,3,1))}$}
We finish by choosing a silting object $N\in\Db(\kk A_3)$, assembling this with $Z=Z^0_{0,0}$ to the associated silting object $M\in\Db(\Lambda(2,3,1))$ and computing the bounded t-structure on $\Db(\Lambda(2,3,1))$ induced by $M$.

Let us start with the silting object
 \[ N = \Sigma^{-2} S(2) \oplus P(1) \oplus \Sigma^3 P(3) \in \Db(\kk A_3) \]
and set $Z=Z^0_{00}$ and $\sZ = \thick{}{Z}$.
As explained above, $N$ corresponds to the object
 \[ M' = \Sigma^{-2} X^0_{1,1}\oplus X^0_{0,0} \oplus \Sigma^3 Z^0_{0,-1} = X^0_{-2,-2} \oplus X^0_{0,0} \oplus Z^1_{3,-2} \in \sZ\orth . \]
By Proposition~\ref{prop:cocones}, $M'$ lifts under $G_Z$ to the silting object
 \[ M=Z^0_{0,0} \oplus X^0_{-2,-2} \oplus X^0_{0,0} \oplus Z^0_{6,-1} \in \Db(\Lambda(2,3,1)) . \]


\begin{center}

\begin{tikzpicture}[scale=0.24]
 
  
    \foreach \x in {0,...,30} \foreach \y in {0,...,10}
      { \ifnumodd{\x+\y}{}{ \node [dot] at (\x,\y) {}; } }
  
    \draw[Yaisle] (0,0) -- (8,0) -- (10,2) -- (2,10) -- (0,10) -- cycle;
    \begin{scope} { \clip (0,0) rectangle (31,10.5); \draw[Yaisle] (7,11) -- (14,4) -- cycle; } \Ydot{14,4} \end{scope}
    \Ydot{14,2} \Xdot{10,0} \Ydot{12,0}
    \draw[Xaisle] (14,0) -- (16,0) -- (15,1) -- cycle;
    \draw[Xaisle] (20,0) -- (22,0) -- (21,1) -- cycle;
    \draw[Xaisle] (26,0) -- (30,0) -- (30,4) -- cycle;
  
    \boxm{2,0}  \boxm{4,0}  \boxm{8,0} 
    \boxo{10,0} \boxo{14,0} \boxp{16,0} 
    \boxp{20,0} \boxp{22,0} \boxp{26,0} \boxp{28,0}
  
    \node at (14,-1.2) {\scalebox{0.5}{(0,0)}};
    \node at (1,11.5) {$\cX^0$}; 


  \begin{scope}[shift={(35,0)}]
  
    \foreach \x in {0,...,30} \foreach \y in {0,...,10}
      { \ifnumodd{\x+\y}{}{ \node [dot] at (\x,\y) {}; } }
  
    \draw[Yaisle] (0,0) -- (8,0) -- (0,8) -- cycle;
    \begin{scope} { \clip (0,0) rectangle (31,10.5); \draw[Yaisle] (11,1) -- (1,11) -- cycle; } \Ydot{11,1} \end{scope}
    \Xdot{12,0}
    \draw[Xaisle] (16,0) -- (18,0) -- (17,1) -- cycle;
    \draw[Xaisle] (22,0) -- (30,0) -- (30,8) -- cycle;
  
    \boxm{0,0}  \boxm{4,0}  \boxm{6,0}  \boxm{10,0} 
    \boxp{12,0} \boxp{16,0} \boxp{18,0} \boxp{22,0} \boxp{26,0} \boxp{28,0}
  
    \node at (16,-1.2) {\scalebox{0.5}{(0,0)}};
    \node at (29.5,11.5) {$\cX^1$}; 

  \end{scope}

\end{tikzpicture}

\begin{tikzpicture}[scale=0.24]


  \begin{scope}[shift={(0,-16)}]
  
    \foreach \x in {0,...,30} \foreach \y in {0,...,10}
      { \ifnumodd{\x+\y}{}{ \node [dot] at (\x,\y) {}; } }
  
    \draw[Yaisle] (14,10) -- (30,10) -- (30,0) -- (24,0) -- cycle;
    \begin{scope} {\clip (-0.5,-0.5) rectangle (11,11); \draw[Xaisle] (0,0) -- (0,10) -- (10,10) -- cycle; } \end{scope}
  
    \node at (16,11.2) {\scalebox{0.5}{(0,0)}}; \node [dot,gray!60] at (16,10) {};
    \node at (1,11.5) {$\cY^0$}; 

  \end{scope}


  \begin{scope}[shift={(35,-16)}]
  
    \foreach \x in {0,...,30} \foreach \y in {0,...,10}
      { \ifnumodd{\x+\y}{}{ \node [dot] at (\x,\y) {}; } }
  
    \draw[Yaisle] (16,10) -- (30,10) -- (30,0) -- (26,0) -- cycle;
    \draw[Xaisle] (0,0) -- (0,10) -- (12,10) -- (2,0) -- cycle;
  
    \node at (16,11.2) {\scalebox{0.5}{(0,0)}}; \node [dot,gray!60] at (16,10) {};
    \node at (29.5,11.5) {$\cY^1$}; 

  \end{scope}

\end{tikzpicture}

\medskip

\begin{tikzpicture}[scale=0.24]


  \begin{scope}[shift={(0,-52)}]
  
    \foreach \x in {0,...,30} \foreach \y in {0,...,30}
      { \ifnumodd{\x+\y}{}{ \node [dot] at (\x,\y) {}; } }
  
    \draw[Yaisle] (0,30) -- (24,30) -- (11,17) -- (0,28) -- cycle;
    \begin{scope} { \clip (0,0) rectangle (31,30.5); \draw[Yaisle] (29,31) -- (13,15) -- cycle; } \Ydot{13,15} \end{scope}
    \Ydot{16,12} \Ydot{18,8}
    \draw[Xaisle] (12,0) -- (20,8) -- (28,0) -- cycle;
    \draw[Xaisle] (20,8) -- (17,11) -- cycle; \Xdot{17,11}
    \draw[Xaisle] (17,13) -- (15,15) -- cycle; \Xdot{17,13} \Xdot{15,15}

    \boxo{15,15} \boxo{20,8}
    \boxm{16,16} \boxm{18,12} \boxm{10,28} \boxm{9,27} \boxm{12,24} \boxm{11,23} \boxm{14,20} \boxm{13,19} 
    \boxp{17,11} \boxp{19,7}  \boxp{24,0}  \boxp{22,4} \boxp{21,3}
  
    \node at (1,31.5) {$\cZ^0$}; 
  
  \end{scope}
  

  \begin{scope}[shift={(35,-52)}]
  
    \foreach \x in {0,...,30} \foreach \y in {0,...,30}
      { \ifnumodd{\x+\y}{}{ \node [dot] at (\x,\y) {}; } }
  
    \begin{scope} { \clip (-0.5,-0.5) rectangle (30,30.5); \draw[Yaisle] (0,30) -- (24,30) -- (12,18) -- cycle; 
                                                     \draw[Yaisle] (29,31) -- (14,16) -- cycle; \Ydot{14,16} } \end{scope}
    \Ydot{16,12}
    \draw[Xaisle] (6,0) -- (17,11) -- (28,0) -- cycle;
    \draw[Xaisle] (17,13) -- (15,15) -- cycle; \Xdot{17,13} \Xdot{15,15}

    \boxm{16,16} \boxm{18,12} \boxm{10,28} \boxm{9,27} \boxm{12,24} \boxm{11,23} \boxm{14,20} \boxm{13,19} 
    \boxp{17,11} \boxp{20,8} \boxp{19,7}  \boxp{24,0}  \boxp{22,4} \boxp{21,3}  \boxp{15,15}
  
    \node at (29.5,31.5) {$\cZ^1$}; 
  
  \end{scope}
  
\end{tikzpicture}
\end{center}

\begin{tabular}{cl}
    \resizebox{1.5ex}{!}{\tikz{ \fill (0,0) rectangle (1,1);}} 
  & the four summands of $M$, with $Z_0$ the top one in $\cZ^0$; \\
    \resizebox{1.5ex}{!}{\tikz{ \node[fill,rectangle,inner sep=0pt,minimum height=4pt,minimum width=4pt] at (0,0) {};
                                \node[draw=white, cross out,rotate=45,thick,minimum size=2.5pt,inner sep=0pt,outer sep=0pt] at (0,0) {}; }}
    and
    \resizebox{1.5ex}{!}{\tikz{ \node[fill,rectangle,inner sep=0pt,minimum height=4pt,minimum width=4pt] at (0,0) {};
                                \node[fill=white, rectangle, inner sep=0pt, minimum height=1pt, minimum width=2.5pt] at (0,0) {}; }}
  & positive ($\Sigma^{>0}M$) and negative suspensions ($\Sigma^{<0}M$), respectively; \\
    \resizebox{1.5ex}{!}{\tikz{ \fill[gray!30] (0,0) circle  (8pt); }}
and \resizebox{1.5ex}{!}{\tikz{ \fill[gray!95] (0,0) circle  (8pt); }}
& coaisle $\sY_M = (\Sigma^{\geq 0} M)\orth$ and aisle $\sX_M = (\Sigma^{<0} M)\orth$, respectively.
\end{tabular}

\bigskip

The corresponding co-t-structure $(\sA_M,\sB_M)$ is right adjacent in the sense of \cite{Bondarko} to the t-structure $(\sX_M,\sY_M)$, i.e.\ $\sB_M = \sX_M$ and $\sA_M \coloneqq {}\orth \sB_M = {}\orth \susp{M} = {}\orth(\Sigma^{\geq 0} M)$.

Recall how to obtain from the silting object $M$ a bounded t-structure $(\sX_M,\sY_M)$ and bounded co-t-structure $(\sA_M,\sB_M)$, using the bijections of K\"onig and Yang \cite{Koenig-Yang}:
\[ \begin{array}{l @{\quad\text{and}\quad} l}
\sX_M \coloneqq (\Sigma^{<0} M)\orth    = \susp{M}                 &   \sY_M \coloneqq (\Sigma^{\geq 0} M)\orth, \\
\sA_M \coloneqq {}\orth(\Sigma^{\geq0}M) = \cosusp{\Sigma^{-1} M}   &   \sB_M \coloneqq (\Sigma^{<0} M)\orth    = \susp{M}.
\end{array} \]


\newpage 

\appendix

\section{Notation, terminology and basic notions} \label{app:notation}

\noindent
In this section we collect some notation and basic terminology, which is mostly standard. We always work over an algebraically closed field $\kk$ and denote the dual of a vector space $V$ by $V^*$. Throughout, $\sD$ will be a $\kk$-linear triangulated category with suspension (otherwise know as shift or translation) functor $\Sigma\colon \sD \to \sD$.

For two objects $A,B \in \sD$, we use the shorthand $\Hom^i(A,B) = \Hom(A,\Sigma^i B)$ resembling Ext spaces in abelian categories, and $\hom(A,B) = \dim\Hom(A,B)$ for dimensions of homomorphism spaces.
We write 
\[ \Hom^{>0}(A,B) = \bigoplus_{i>0} \Hom(A,\Sigma^i B) \quad\text{and}\quad
   \Hom^\bullet(A,B) = \bigoplus_{i\in\IZ} \Sigma^{-i}\Hom(A,\Sigma^i B) \]
for aggregated homomorphism spaces (and similarly for obvious variants) and
for the homomorphism complex, a complex of vector spaces with zero differential.

\subsection{Properties of triangulated categories and their subcategories} \label{app:category-properties}

A $\kk$-linear triangulated category $\sD$ is said to be

\begin{description}[leftmargin=1em,font=\normalfont,style=sameline]
\item[\emph{algebraic}] if $\sD$ arises as the homotopy category of a $\kk$-linear differential graded category;
                        see \cite{Keller-dg}. Examples are bounded derived categories of $\kk$-linear abelian 
                        categories.
\item[\emph{Hom-finite}] if $\dim\Hom(D_1,D_2)<\infty$ for all objects $D_1,D_2\in\sD$. The bounded derived category
                        $\Db(\Lambda)$ of any finite-dimensional $\kk$-algebra $\Lambda$ is Hom-finite.
\item[\emph{Krull--Schmidt}] if every object of $\sD$ is isomorphic to a finite direct sum of objects all of whose 
                        endomorphism rings are local. In this case, the direct sum decomposition is unique up to 
                        isomorphism. Bounded derived categories of $\kk$-linear Hom-finite abelian categories are 
                        Krull--Schmidt; see \cite{Atiyah}.
%
\item[\emph{indecomposable}] if for every decomposition $\sD\cong\sD_1\oplus\sD_2$ with triangulated categories $\sD_1$
                      and $\sD_2$ either $\sD_1\cong0$ or $\sD_2\cong0$. The derived category of a finite-dimensional 
                      algebra is indecomposable if (and only if) the associated Gabriel quiver is connected.
\item[in possession of \emph{Serre duality}] if there is an equivalence $\SSS\colon\sD\isom\sD$ with
                        $\Hom(D_1,D_2)\cong\Hom(D_2,\SSS D_1)^*$, bifunctorially in $D_1,D_2\in\sD$. Such an 
                        autoequivalence is canonical and unique, if it exists, and called the \emph{Serre functor}
                        of $\sD$. 
\end{description}

\noindent
The existence of a Serre functor is equivalent to the existence of Auslander--Reiten triangles; see \cite[\S I.2]{Reiten-vandenBergh}. If $\Lambda$ is a finite-dimensional $\kk$-algebra, then $\Db(\Lambda)$ has Serre duality if and only if $\Lambda$ has finite global dimension; in this case, the Auslander--Reiten translation is given by the cosuspended Serre functor: 
$\tau = \Sigma^{-1}\SSS$.

We conclude that $\Db(\LLambda)$ is algebraic, Hom-finite, Krull--Schmidt and indecomposable for all choices of $r,n,m$. It has Serre duality if and only if $n>r$, which we always assume in this article.

\subsection{Subcategories of triangulated categories} \label{app:subcategory-types}

\noindent
Let $\sC$ be a collection of objects of $\sD$, regarded as a full subcategory. We recall the following terminology:

\medskip
\noindent
\begin{longtable}{@{} p{0.13\textwidth} @{} p{0.87\textwidth} @{}}
$\sC\orth$,      & the \emph{right orthogonal} to $\sC$, the full subcategory of $D\in\sD$ with $\Hom(\sC,D)=0$, \\
${}\orth\sC$ ,   & the \emph{left orthogonal} to $\sC$, the full subcategory of $D\in\sD$ with $\Hom(D,\sC)=0$. \\
                 & If $\sC$ is closed under suspensions and cosuspensions, then $\sC\orth$ and $\orth\sC$ are
                  triangulated subcategories of $\sD$. \listskip
$\thick{}{\sC}$, & the \emph{thick subcategory generated by $\sC$}, the smallest triangulated  subcategory of $\sD$ containing $\sC$ which is also closed under taking direct summands. \listskip
$\susp(\sC)$
and
$\cosusp(\sC)$,  & the \emph{(co-)suspended subcategory generated by $\sC$}, the smallest full subcategory of 
                   $\sD$ containing $\sC$ which is closed under (co-)suspension, extensions and taking direct 
                   summands. \listskip
$\add(\sC)$,     & the \emph{additive subcategory} of $\sD$ containing $\sC$, the smallest full subcategory of
                   $\sD$ containing $\sC$ which is closed under finite coproducts and direct summands. \listskip
$\ind{\sC}$,     & the set of \emph{indecomposable} objects of $\sC$, up to isomorphism. \listskip
$\clext{\sC}$,   & the smallest full subcategory of $\sD$ containing $\sC$ that is closed under extensions, i.e. if $\tri{C'}{C}{C''}$ is a triangle with $C',C'' \in \sC$ then $C \in \sC$. \\
\end{longtable}
\medskip

\noindent
The \emph{ordered} extension closure of a pair of subcategories $(\sC_1,\sC_2)$ of $\sD$ is defined as
 \[ \sC_1 * \sC_2 \coloneqq \add\{D \in \sD \mid \tri{C_1}{D}{C_2} \text{ for } C_1\in \sC_2 \text{ and } C_2 \in \sC_2\} .\] 
This operation is associative and $\sC$ is extension closed in $\sD$ if and only if $\sC * \sC \subseteq \sC$.

\subsection{Approximations and adjoints} \label{app:approximations}

\noindent
For this section only, suppose $\sD$ is an additive category and $\sC$ a full subcategory of $\sD$.

Recall that $\sC$ is called \emph{right admissible} in $\sD$ if the inclusion functor $\sC \hookrightarrow \sD$ admits a right adjoint. Analogously for \emph{left admissible}. A subcategory $\sC$ is called \emph{admissible} if it is both left and right admissible.

Often, one does not need admissibility but only approximate admissibility. 
A \emph{right $\sC$-approximation} of an object $D\in\sD$ is a morphism $C\to D$ with $C\in\sC$ such that the induced maps $\Hom(C',C)\to\Hom(C',D)$ are surjective for all $C'\in\sC$. 
A morphism $f\colon C\to D$ is called a \emph{minimal right $\sC$-approximation} if $fg=f$ is only possible for isomorphisms $g\colon C\to C$. Dually for \emph{(minimal) left $\sC$-approximations}.
We say $\sC$ is
\begin{itemize}
\item \emph{contravariantly finite in $\sD$} if all objects of $\sD$ have right $\sC$-approximations;
\item \emph{covariantly finite in $\sD$}     if all objects of $\sD$ have left $\sC$-approximations;
\item \emph{functorially finite in $\sD$}    if it is contravariantly finite and covariantly finite in $\sD$.
\end{itemize}
Note that in the case that $\sD$ is a Hom-finite, $\kk$-linear, Krull--Schmidt category, the existence of a $\sC$-approximation guarantees the existence (and uniqueness, up to isomorphism) of a minimal $\sC$-approximation.

Sometimes, right $\sC$-approximations are called \emph{$\sC$-precovers} and left $\sC$-approximations are called \emph{$\sC$-preenvelopes}. If for all $D\in\sD$ the induced map $\Hom(C',C) \to \Hom(C',D)$ above were bijective instead of surjective, then $\sC$ would be even right admissible. In this sense, the morphism $C\to D$ `approximates' the (possibly nonexistent) right adjoint to the inclusion functor.

For Krull--Schmidt triangulated categories $\sD$, these concepts coincide:

\begin{proposition}[{\cite[Proposition 1.3]{Keller-Vossieck}}] \label{prop:KV}
Let $\sD$ be a Krull--Schmidt triangulated category and let $\sC \subset \sD$ a suspended subcategory. Then $\sC$ is contravariantly finite in $\sD$ if and only if $\sC$ is right admissible. Dually for covariantly finite cosuspended subcategories.
\end{proposition}

Thus, a thick subcategory $\sC$ of $\sD$ is functorially finite if and only if it is admissible.
Functorial finiteness can often be deduced from Hom-finiteness. More precisely, let
\[ H_D \coloneqq \{ C\in\ind{\sC} \mid \Hom(D,C)\neq0 \} , \qquad
   H^D \coloneqq \{ C\in\ind{\sC} \mid \Hom(C,D)\neq0 \} . \]

\begin{lemma} \label{lem:functorially-finite}
Let $\sD$ be a Hom-finite, Krull--Schmidt 
category with a subcategory $\sC$. If the set $H_D$ is finite for all $D\in\ind{\sD}$, then $\sC$ is covariantly finite in $\sD$. Dually, if $H^D$ is finite for all $D\in \ind{\sD}$, then $\sC$ is contravariantly finite in $\sD$.
\end{lemma}

\begin{proof}
For  $D\in \ind{\sD}$, the direct sum $\bigoplus_{C\in H_D} C \otimes \Hom(D,C)^*$ is a well-defined object of $\sD$ by the assumption on $H_D$. Hence the natural morphism $D \to \bigoplus_{C\in H_D} C \otimes \Hom(D,C)^*$ is a (not necessarily minimal) left $\sC$-approximation of $D$. Therefore, indecomposable objects of $\sD$ have left $\sC$-approximations; as $\sD$ is Krull--Schmidt, all objects of $\sD$ do and $\sC$ is covariantly finite in $\sD$.
Dually for contravariant finiteness.
\end{proof}

\begin{corollary} \label{cor:functorially-finite}
Let $\sD$ be a Hom-finite, Krull--Schmidt 
category with a subcategory $\sC$ containing only finitely many indecomposable objects. Then $\sC$ is functorially finite in $\sD$.
\end{corollary}

\subsection{Silting subcategories} \label{sec:generating}

Silting objects are a generalisation of tilting objects, which were introduced in \cite{Keller-Vossieck}. However, we follow the terminology of \cite{AI}. Note that all subcategories are assumed to be additive and closed under isomorphisms.

Let $\sM$ be a subcategory of a triangulated category $\sD$. 
\begin{itemize}
\item $\sM$ is called a \emph{partial silting subcategory} if $\Hom^{>0}(\sM,\sM)=0$.
\item $\sM$ is called a \emph{silting subcategory} if it is partial silting and $\thick{\sD}{\sM}=\sD$.
\item An object $D\in\sD$ is called a \emph{silting object} if $\add(D)$ is a silting subcategory.
\item Two silting objects $D,D'\in\sD$ are \emph{equivalent} if and only if $\add(D) = \add(D')$.
\end{itemize}

For reasonable categories, there is a strong connection between silting objects and silting subcategories; see \cite[Theorem.~2.27]{AI}:

\begin{lemma}
Let $\sD$ be a Hom-finite, Krull--Schmidt triangulated category. Then $\sD$ has a silting object if and only if $\sD$ has a silting subcategory and $K(\sD)$ is free of finite rank.
\end{lemma}

In particular, if a category $\sD$ as in the lemma has a silting object $D$, then 
\[ \rk K(\sD) = \#\{\text{isomorphism classes of indecomposable summands of } D \} \]
and in particular, the right-hand side is independent of the silting object.
We record two further easy observations:

\begin{lemma}
A partial silting subcategory is extension-closed.  
\end{lemma}

\begin{proof}
If $\sM\subset\sD$ is partial silting, then any extension $\trilabels{M'}{D}{M''}{}{}{e}$ with $M',M''\in\sM$ has $e\in\Hom(M'',\Sigma M')=0$, so that the extension is trivial. In other words, the extension closure $\sM*\sM$ is built from direct sums only.
\end{proof}

\begin{lemma} \label{lem:silting-ff}
If $\sD$ is a Hom-finite, Krull--Schmidt triangulated category with a silting object, then any (additive) subcategory $\sN$ of a silting subcategory $\sM$ is functorially finite in $\sM$. 
\end{lemma}

\begin{proof}
The existence of a silting object implies that $\sM$ and $\sN$ are each additively generated by finitely many objects. Now apply Corollary~\ref{cor:functorially-finite}.
\end{proof}

\subsection{Torsion pairs, t-structures and co-t-structures} \label{app:torsion}

We assume again that $\sD$ is a $\kk$-linear triangulated category. A pair $(\sX,\sY)$ of full subcategories closed under direct summands is called a \emph{torsion pair} if $\Hom(\sX,\sY) =0$ and $\sD = \sX*\sY$; see \cite{Iyama-Yoshino}.

Both $\sX$ and $\sY$ are then extension closed. By definition, for every $D\in\sD$ there is a triangle $\tri{X}{D}{Y}$ with $X\in \sX$ and $Y\in\sY$. The map $X\to D$ is a right $\sX$-approximation and $D\to Y$ is a left $\sY$-approximation, i.e.\ $\sX$ is contravariantly finite and $\sY$ is covariantly finite in $\sD$. The triangle is called the \emph{approximation triangle} of $D$. By abuse of terminology, we shall call $\sX$ the \emph{aisle} and $\sY$ the \emph{co-aisle} of the torsion pair. The abuse arises as this terminology is normally reserved for the case that $(\sX,\sY)$ is a t-structure (see below).

The torsion pair $(\sX,\sY)$ will be called \emph{bounded} if $\bigcup_{i\in \IZ} \Sigma^i \sX = \bigcup_{i \in \IZ} \Sigma^i \sY = \sD$. Torsion pairs appear in three important guises, namely $(\sX,\sY)$ is called a
\begin{itemize}
\item \emph{t-structure} \cite{BBD} if $\Sigma \sX \subseteq \sX$ ($\iff \Sigma^{-1} \sY \subseteq \sY$);
\item \emph{co-t-structure} \cite{Pauksztello} (also \emph{weight structure} \cite{Bondarko}) if $\Sigma^{-1} \sX \subseteq \sX$ ($\iff \Sigma \sY \subseteq \sY$);
\item \emph{stable t-structure} (also \emph{semi-orthogonal decomposition}) if $\Sigma \sX = \sX$ ($\iff \Sigma \sY = \sY$). 
\end{itemize}
For historical reasons, when the terminology `semi-orthogonal decomposition' is used the torsion pair is often written as $\sod{\sY,\sX}$. Furthermore, a t-structure is stable if and only if it is also a co-t-structure.

If $(\sX,\sY)$ is a t-structure then its \emph{heart} $\sH = \sX \cap \Sigma \sY$ is an abelian subcategory of $\sD$; see \cite[Theorem 1.3.6]{BBD}. 
A bounded t-structure is determined by its heart via $\sX = \susp \sH$ and $\sY = \cosusp \Sigma^{-1} \sH$; see, for example, \cite[Section 3]{Bridgeland}. 

If $(\sX, \sY)$ is a co-t-structure then its \emph{co-heart} $\sM = \sX \cap \Sigma^{-1} \sY$ is a partial silting subcategory of $\sD$; see, for instance, \cite[Corollary 5.9]{MSSS}. Note that, if $\sM$ is abelian then it is semisimple. A co-t-structure is bounded if and only if $\sM$ is a silting subcategory. Moreover, a bounded co-t-structure is determined by its co-heart (\cite[Proposition 2.23]{AI}):
\[ \label{page:co-t-structure}
   \sX = \cosusp \sM = \bigcup_{l\geq 0} \Sigma^{-l} \sM * \Sigma^{-l+1} \sM * \cdots * \sM, 
   \textrm{ and, }
   \sY = \susp \sM   = \bigcup_{l\geq 0} \sM * \Sigma \sM * \cdots * \Sigma^l \sM.
\]

\begin{remark}
If $(\sX,\sY)$ is a t-structure then the approximation triangle is functorial and called the \emph{truncation triangle}, with $X\to D$ being a right minimal $\sX$-approximation called the \emph{right truncation} and $D\to Y$ a left minimal $\sY$-approximation called the \emph{left truncation} of $\sD$. Another way to express this functoriality is: the inclusion $\sX\embed\sD$ has a right adjoint (given by $D\mapsto X$) and $\sY\embed\sD$ has a left adjoint. In particular, truncations are minimal approximations. We mention that `t-structure' is an abbreviation for `truncation structure'. 
\end{remark}

\subsection{K\"onig-Yang bijections} \label{app:correspondences}

The notions of silting subcategories, t-structures and co-t-structures for finite dimensional $\kk$-algebras are related by the following bijections of K\"onig and Yang. Before we state them, recall  an abelian category $\sA$ is called a \emph{length category} if it is both artinian and noetherian.

\begin{theorem}[{\cite[Theorem 6.1]{Koenig-Yang}}]
\label{thm:koenig-yang}
Let $\Lambda$ be a finite dimensional $\kk$-algebra. There are bijections between
\begin{enumerate}[label=(\roman*)]
\item equivalence classes of silting objects in $\sK^b(\proj \Lambda)$,
\item bounded t-structures in $\sD^b(\mod \Lambda)$ whose heart is a length category,
\item bounded co-t-structures in $\sK^b(\proj \Lambda)$.
\end{enumerate}
\end{theorem}

\noindent
Under these bijections, a silting subcategory $\sM\subset\sD$ is mapped to the
\[ \begin{array}{r l}
   \text{t-structure} & (\sX_\sM,\sY_\sM) \coloneqq (\Sigma^{<0}\sM)\orth,\Sigma^{\geq0}\sM)\orth)   = (\susp{\sM},\Sigma^{<0}\sM)\orth)    ; \\
\text{co-t-structure} & (\sA_\sM,\sB_\sM) \coloneqq {}\orth(\Sigma^{\geq0}\sM),\Sigma^{<0}\sM)\orth) = (\cosusp{\Sigma\inv\sM},\susp{\sM}) .
\end{array} \]

\subsection{Exceptional sequences and semi-orthogonal decompositions} \label{sec:exceptionals}

The notion of semi-orthogonal decomposition $\sD=\clext{\sC_1,\sC_2}$ is synonymous with that of a stable t-structure $(\sC_2,\sC_1)$, see \ref{app:torsion}, and leads to equivalences $\sC_1 \cong \sD/\sC_2$ and $\sC_2 \cong \sD/\sC_1$.
An admissible subcategory $\sC\subset\sD$ produces two semi-orthogonal decompositions $\sD = \clext{\sC,{}\orth\sC} = \clext{\sC\orth,\sC}$.

An object $E$ of a $\kk$-linear triangulated category $\sD$ is \emph{exceptional} if $\Hom(E,E)=\kk$ and $\Hom^{\neq0}(E,E)=0$, i.e.\ $E$ has the smallest possible graded endormorphism ring. Exceptional objects are characterised by the following property (which is used in the text): $\thick{\sD}{E} = \add(\Sigma^i E\mid i\in\IZ)$. 
Morever, the subcategory $\thick{\sD}{E}$ is then admissible by \cite[Theorem~3.2]{Bondal}. Hence an exceptional object $E$ leads to semi-orthogonal decompositions $\sD = \clext{\thick{\sD}{E}\orth,\thick{\sD}{E}}$.

An \emph{exceptional sequence} in $\sD$ is a tuple $(E_1,\ldots,E_t)$ of exceptional objects such that $\Hom^\bullet(E_i,E_j)=0$ for all $i>j$. 
The sequence is \emph{full} if $\thick{\sD}{E_1,\ldots,E_t}=\sD$ and \emph{strong} if $\Hom^\bullet(E_i,E_j)=\Hom(E_i,E_j)$, i.e.\ all homomorphisms occur in degree zero.
A full, strong exceptional sequence $(E_1,\ldots,E_t)$ gives rise to a tilting object $E_1\oplus\cdots\oplus E_t$. Similarly, a full exceptional sequence $(E_1,\ldots,E_t)$ with $\Hom^{>0}(E_i,E_j)=0$ for all $i,j$ gives rise to a silting object.

\section{The repetitive algebra and string modules} \label{app:strings}

\noindent
For a finite-dimensional algebra $\Lambda$, Happel showed in \cite{Happel} that there is a full embedding $F\colon\Db(\Lambda) \to \stmod{\rLambda}$, where $\stmod{\rLambda}$ denotes the stable module category of the repetitive algebra $\rLambda$, and $F$ is called the \emph{Happel functor}. A finite-dimensional algebra $\Lambda$ is gentle if and only if its repetitive algebra is special biserial (see \cite[Proposition]{Schroer}). For such an algebra, there is a convenient description of all the indecomposable objects of $\stmod{\rLambda}$ using string and band modules; see \cite{Schroer}. Since the algebras $\LLambda$ are gentle, this machinery applies. Moreover, only string modules occur; indeed it is this absence of band modules that is responsible for discreteness. Thus, we shall omit any further reference to band modules.

In this section we shall recall the construction of the repetitive algebra, the description of string modules and the maps between them. We then apply these results to the derived-discrete algebras $\LLambda$.

\subsection{The repetitive algebra}

The notion of a repetitive algebra was introduced by Hughes and Waschb\"usch in \cite{Hughes-Waschbuesch}. The standard references are \cite{Hughes-Waschbuesch,Ringel,Schroer}. The relations for $\LLambda$ are also recalled in \cite{BGS}. The following summary is based on \cite{Schroer}.

Let $Q=(Q_0,Q_1)$ be a finite, connected quiver with vertices $Q_0$ and arrows $Q_1$. A \emph{path} $p$ in $Q$ is a sequence of arrows $p=a_1 a_2 \cdots a_t$ with $s(a_{i+1}) = e(a_{i})$ for $1 \leq i < t$. The start of $p$, $s(p) = s(a_1)$ and the end of $p$, $e(p) = e(a_t)$. The path $p$ is said to have \emph{length} $t$. Note there is a trivial path of length $0$, $e_v$, corresponding to each vertex $v \in Q_0$.
The concatenation $p_1 p_2$ of paths $p_1$ and $p_2$ is defined if and only if $e(p_1) = s(p_2)$. A path $q$ is called a \emph{subpath} of a path $p$ if $p=p_1 q p_2$ for some (not necessarily non-trivial) paths $p_1$ and $p_2$.
Write $\paths$ for the set of paths of $Q$.
A \emph{relation} for $Q$ is a non-zero linear combination of paths of length at least $2$ which have the same starting points and end points. A \emph{zero-relation} is a relation of the form $p$ (sometimes written $p=0$). A \emph{commutativity relation} is a relation of the form $p - q$.

Now let $\rho$ be a set of zero- and commutativity relations for $Q$ and consider the path algebra arising from the bound quiver $\Lambda \coloneqq \kk Q/\clext{\rho}$. Two paths $p_1$ and $p_2$ in $Q$ are \emph{equivalent} if $p_1 = p'vp''$ and $p_2 = p'wp''$, where $v-w$ or $w-v$ is a commutativity relation in $\rho$. Note that this generates an equivalence relation on $\paths$; we denote the equivalence class of a path $p$ by $\bp$. A path $p$ in $Q$ is called a \emph{path} in $(Q,\rho)$ if for each $p' \in \bp$, $p'$ does not have a subpath belonging to $\rho$. A path $a_1 \cdots a_n$ is called \emph{maximal} if $b a_1 \cdots a_n$ and $a_1 \cdots a_n c$ are not paths in $(Q,\rho)$ for each $b$ and $c$ such that $e(b) = s(a_1)$ and  $e(a_n) = s(c)$.

The repetitive algebra $\rLambda \coloneqq \kk \rQ/\clext{\rrho}$, where $\rQ=(\rQ_0,\rQ_1)$ is specified by:
\begin{itemize}
\item the vertex set is given by $\rQ_0 \coloneqq \IZ \times Q_0$;
\item for each arrow $a\colon x \to y$ in $Q_1$ there is an arrow $(i,a)\colon (i,x) \to (i,y)$ in $\rQ_1$;
\item for each maximal path $p$ in $(Q,\rho)$, there is a \emph{connecting arrow} $\rp \colon (i,y) \to (i+1,x)$ in $\rQ_1$, where $s(p) = x$ and $e(p)=y$.
\end{itemize}
If $p$ is a path in $Q$, the corresponding path in $(i,Q)$ is denoted $(i,p)$. Let $p=p_1 p_2$ be a maximal path in $(Q,\rho)$. Then the path $(i,p_2)(i,\rp)(i+1,p_1)$ is called a \emph{full path} in $\rQ$. We now define the relations:
\begin{itemize}
\item $\rrho$ inherits the relations from $\rho$, i.e.\ for paths $p$, $p_1$ and $p_2$ in $Q$, if $p\in \rho$ (resp. $p_1 - p_2 \in \rho$) then $(i,p) \in \rrho$ (resp. $(i,p_1) - (i,p_2) \in \rrho$) for all $i \in \IZ$.
\item Let $p$ be a path that contains a connecting arrow. If $p$ is not a subpath of a full path then $p \in \rrho$.
\item Let $p=p_1 p_2 p_3$ and $q=q_1 q_2 q_3$ be maximal paths in $(Q,\rho)$ with $p_2 = q_2$. Then $(i,p_3)(i,\rp)(i+1,p_1) - (i,q_3)(i,\rq)(i+1,q_1) \in \rrho$ for all $i\in \IZ$.
\end{itemize}
Denote the set of paths in $(\rQ,\rrho)$  by $\rpaths$.


In $\rQ(r,n,m)$ there are $\IZ$ copies of each vertex in $Q_0(r,n,m)$, labelled $(i,x)$ for $x \in Q_0(r,n,m)$ and $i \in \IZ$. Likewise, there are $\IZ$ copies of each arrow in $Q_1(r,n,m)$, for each $i \in \IZ$ we have:
\begin{itemize}
\item the arrows $(i,a_j)\colon (i,j) \to (i,j+1)$ for $-m \leq j \leq -1$;
\item the arrows $(i,b_j)\colon (i,j) \to (i,j+1)$ for $0 \leq j \leq n-r$;
\item the arrows $(i,c_j)\colon (i,j) \to (i,j+1)$ for $n-r+1 \leq j \leq n-1$, where $n \equiv 0$;
\item the arrows $(i,x_j) \colon (i,j+1) \to (i+1,j)$ for $n-r+1 \leq j \leq n-1$, where $n \equiv 0$;
\item the arrows $(i,y) \colon (i,n-r+1) \to (i+1,-m)$.
\end{itemize}
In an abuse of notation, we write down only one copy of each vertex and arrow in the following shorthand version of the quiver $\rQ(r,n,m)$.

\[ \label{fig:repetitive_algebra} 
\centering \xymatrix@R=2ex{
     & & & & & \smxy{1} \ar[r]^-{b_1}     &
                \cdots \ar[r]^-{b_{n-r-1}} &
               \smxy{n-r} \ar@/^/[dr]^-{b_{n-r}}
\\
\smxy{-m} \ar[r]^-{a_{-m}}        &
\smxy{1-m} \ar[r]^-(0.4){a_{1-m}} &
  \cdots \ar[r]^-{a_{-2}}       &
 \smxy{-1} \ar[r]^-{a_{-1}}        &
\smxy{0} \ar@/^/[ur]^-{b_0} \ar@/^/[dr]^-{x_{n-1}}
 & & & & \smxy{n-r+1} \ar@/^/[dl]^-(0.6){c_{n-r+1}} \ar`d[dd] `/20pt[llllllll]^y [llllllll]
\\
 & & & & & \smxy{n-1} \ar@/^/[ul]^-{c_{n-1}} \ar@/^/[r]^-{x_{n-2}}  &
            \cdots \ar@/^/[l]^-{c_{n-2}}  \ar@/^/[r]^-{x_{n-r+2}} &
           \smxy{n-r+2} \ar@/^/[l]^-{c_{n-r+2}} \ar@/^/[ur]^-{x_{n-r+1}} & \\
 & & & & & & & &
} \]
Following the rules above, we can read off the following relations for $\rLLambda$; see \cite[Section~3]{BGS}. The degrees of the arrows should be inferred by the presence of the connecting arrows labelled $x$ and $y$ of degree $1$. We have the following relations:
\begin{itemize}
\item $c_k c_{k+1} = 0$ for $k = n-r, \ldots, n-1$, where $c_{n-r} \coloneqq b_{n-r}$ and $c_n \coloneqq b_0$;
\item $x_k x_{k-1} = 0$ for $k = n-r+2, \ldots, n-1$;
\item $y x_{n-1} = 0$ if $m =0$, and $a_{-1} x_{n-1} = 0$ if $m>0$;
\item $c_{n-r+1} x_{n-r+1} - y a_{-m} \cdots b_{n-r} = 0$ if $r > 1$;
\item $c_k x_k - x_{k-1} c_{k-1} = 0$ for $k=n-r+2, \ldots, n-1$ if $r>1$;
\item $x_{n-1}c_{n-1} - b_0 \cdots b_{n-r} y a_{-m} \cdots a_{-1} = 0$ if $r >1$, and in the case $r=1$ we have $y a_{-m} \cdots b_{n-1} - b_0 \cdots b_{n-1} y a_{-m} \cdots a_{-1} = 0$;
\item Any path starting at $(i,k)$ and ending at $(i+1,k+1)$, with $k \neq 0$ and $-m \leq k \leq n-r$, that contains $y$ as a subpath is zero.
\end{itemize}

\subsection{String modules} \label{sec:strings}

Let $\Lambda = \kk Q/\clext{\rho}$ be a special biserial algebra. We describe strings for the bound quiver $(Q,\rho)$, which give rise to string modules. The references are \cite{Butler-Ringel} and \cite{Wald-Waschbusch}. We remind the reader that all modules are right modules.

For each arrow $a\in Q_1$, introduce a formal inverse $\ia = a^{-1}$ with $s(\ia)=e(a)$ and $e(\ia)=s(a)$. For a path $p=a_1 \cdots a_n$ the inverse path $\ip = \ia_n \cdots \ia_1$.

A \emph{walk} $w$ of length $l>0$ in $(Q,\rho)$ is a sequence $w=w_1 \cdots w_l$, satisfying the usual concatenation requirements, where each $w_i$ is either an arrow or an inverse arrow. Formal inverses of walks are defined in the obvious way. Starting and ending vertices of walks and their inverses are defined analogously to those for paths.

A walk is called a \emph{string} if it contains neither subwalks of the form $a \ia$ or $\ia a$ for some $a \in Q_1$, nor a subwalk $v$ such that $v \in \rho$ or $\iv \in \rho$. We also define \textbf{two} \emph{strings of length zero}, namely, for each $x \in Q_0$ there are \emph{trivial strings} $1^+_x$ and $1^-_x$. We write $s(1^{\pm}_x) = e(1^{\pm}_x) = x$ and set $(1^{\pm}_x)^{-1} = 1^{\mp}_x$.

For technical reasons, in order to define composition of strings with trivial strings, we need to introduce \emph{string functions} $\sigma, \epsilon \colon Q_1 \to \{-1,1\}$ satisfying the following properties:
\begin{itemize}
\item If $a_1 \neq a_2 \in Q_1$ with $s(a_1) = s(a_2)$ then $\sigma(a_1) = -\sigma(a_2)$.
\item If $b_1 \neq b_2 \in Q_1$ with $e(b_1) = e(b_2)$ then $\epsilon(b_1) = - \epsilon(b_2)$.
\item If $a,b \in Q_1$ are such that $ab \notin \rho$ then $\sigma(b) = -\epsilon(a)$.
\end{itemize}
The choice of such string functions is completely arbitrary. An explicit algorithm for choosing such functions is given in \cite[p.~158]{Butler-Ringel}. The functions $\sigma$ and $\epsilon$ can be extended to strings as follows. If $a \in Q_1$, define $\sigma(\ia) = \epsilon(a)$ and $\epsilon(\ia)=\sigma(a)$. If $w = w_1 \cdots w_n$ is a string, define $\sigma(w) = \sigma(w_1)$ and $\epsilon(w) = \epsilon(w_n)$. Finally, for $x \in Q_0$ define $\sigma(1^\pm_x) = \mp 1$ and $\epsilon(1^\pm_x) = \pm 1$.

We are now able to define compositions of strings. For strings $v = v_1 \cdots v_m$ and $w = w_1 \cdots w_n$ of length at least $1$ this is done in the obvious way: the composition $vw$ is defined if $vw = v_1 \cdots v_m w_1 \cdots w_n$ is a string. However, if $w = 1^\pm_x$ then $vw$ is defined if $e(v) = x$ and $\epsilon(v) = \pm 1$. Analogously, if $v = 1^\pm_x$ then $vw$ is defined if $s(w) = x$ and $\sigma(w) = \mp 1$. Note that given arbitary strings $v$ and $w$ whose composition $vw$ is defined, we necessarily have $\sigma(w) = -\epsilon(v)$.
However, in the case of a special biserial algebra, this condition is not sufficient for a string to be defined.

Modulo the equivalence relation $w \sim \iw$, the strings form an indexing set for the so-called \emph{string modules} of $\Lambda$. We shall write $M(w)$ for the corresponding string module. We direct the reader to \cite[Section~3]{Butler-Ringel} for precise details on how to pass to a representation-theoretic description of the modules.

\begin{example} \label{ex:B}
Consider $\rLambda(2,3,1)$, where we relabel the arrows in the figure above
as $a=a_{-1}$, $b=b_0$, $c=b_1$, $d=c_2$, $x=x_2$ and $y=y$ to avoid cumbersome subscripts. Consider the string $(-2,x)(-1,\ic)(-1,\ib)(-1,x)(0,\ic)(0,\ib)$, which we write as $x \ic \ib x \ic \ib_0$ for short, with $\ib_0 = (0,\ib)$ to determine the `degrees' of each of the arrows. This can be represented pictorially by the diagram below.
\[
\xymatrix@!=2pt{
                         &        &                          &                                           &        &                       & \bullet \ar[dl]^-{(0,b)}  \\
                         &        &                          & \bullet \ar[dl]_-{(-1,b)} \ar[dr]^(0.3){(-1,x)} &       & \bullet \ar[dl]^-{(0,c)} & \\
\bullet \ar[dr]_-{(-2,x)} &        & \bullet \ar[dl]^-{(-1,c)} &                                            & \bullet &                        &  \\
                         & \bullet &                          &                                           &        &                       &
}
\]
In this picture, we read from left to right, direct arrows point downwards and to the right and inverse arrows point downwards and to the left.
\end{example}

\subsection{Irreducible maps between string modules and a linear order}

A complete description of the irreducible maps between string modules was obtained in \cite{Butler-Ringel}. Given a string $w$, the irreducible maps whose source is the string module $M(w)$ can be determined by modifying $w$ in a minimal way either on the left, or on the right.

We describe the algorithm that modifies $w$ on the left, i.e.\ that keeps the endpoint of $w$ fixed, to produce a new string $w[1]$. This yields an irreducible morphism $M(w) \to M(w[1])$.
\begin{enumerate}[label=(\arabic*)]
\item \textit{Adding a hook on the left:}
If there exists $a_0 \in Q_1$ such that $a_0 w$ is defined, then let $a_1\cdots a_n$ be the maximal direct string starting at $s(a_0)$. Then $w[1] \coloneqq \ia_n \cdots \ia_1 a_0 w$; the irreducible map is the natural inclusion.
\item \textit{Removing a cohook on the left:}
If there is no $a \in Q_1$ such that $aw$ is defined, then $w = v_1 \cdots v_{n-1} \bar{v}_n w'$ with $v_i \in Q_1$ and a string $w'$, where $v_1 \cdots v_{n-1}$ is a maximal direct substring at the beginning of $w$. Then $w[1] \coloneqq w'$; the irreducible map is the natural projection map.
\end{enumerate}
There is a dual algorithm, which adds a hook or removes a cohook on the right to output the string $[1]w$.  The inverse operations are written $w[-1]$ and $[-1]w$, respectively. For $n \in \IZ$ we define $w[n] = w[1]\cdots[1]$; similarly for $[n]w$.

We illustrate these concepts in the diagrams below; in the left-hand diagram, we add a hook, in the right-hand diagram, we remove a cohook.
\[
\xymatrix@!R=0.025pc@!C=0.005pc{
                   & &         &                        &                     &                                      &                    && \bullet \ar[dr]^-{v_1} &                    &                 &         &                             &&         &         &                   &&        \\
                   & &         &                        &                     & \ar[dl]_-{a_1} \bullet \ar[dr]^-{a_0} &                    &&                 & \bullet \ar@{.}[dr] &                 &         &                             &&         &         &                   &&         \\
\bullet \arrr^-{w} & & \bullet & \mapsto                & \ar@{.}[dl] \bullet &                                       & \bullet \arrr^-{w} &&  \bullet         &                     & \bullet \ar[dr]^-{v_n} &         & \ar[dl]^-{v_{n+1}} \bullet \arrr^-{w'} && \bullet & \mapsto & \bullet \arrr^-{w'} && \bullet \\
                   & &         & \ar[dl]_-{a_n} \bullet &                     &                                       &                    &&                  &                     &                 & \bullet &                             &&         &         &                    &&         \\
                   & & \bullet &                        &                     &                                       &                    &&                  &                     &                 &         &                             &&         &         &                    &&
}
\]
These operations give rise to AR sequences/triangles:
\[ \xymatrix@!=0.6pc{
                  & w[1] \ar[dr] & \\
w \ar[ur] \ar[dr] &              & **[r] [1]w[1] = [1](w[1]) =([1]w)[1] . \\
                  & [1]w \ar[ur] &
} \]

The process of adding a hook or removing a cohook determines a total order on strings ending at a given vertex whose modules lie in the same component of the AR quiver. This process can be generalised to produce a total order on all strings ending at a given vertex. This is the Gei\ss\ total order \cite{Geiss}, which we describe next.

Let $x \in Q_0$. There is a linear order on strings $w$ and $v$ in $(Q,\rho)$ such that $e(w) = e(v) = x$ and $\epsilon(w) = \epsilon(v) = t$ with $t \in \{-1,1\}$. Namely,
\[
v < w \iff
\left\{ \begin{array}{ll}
   \text{either} & w=w'v, \text{ where } w'=w'_1 \cdots w'_n \text{ with } w'_n \in Q_1; \\
   \text{or} & v = v'w, \text{ where } v'=v'_1 \cdots v'_m \text{ with } \bar{v}'_m \in Q_1;\\
   \text{or} & v=v'c, w=w'c \text{ with } w'_n \in Q_1 \text{ and } \bar{v}'_m \in Q_1,
\end{array} \right.
\]
where $w'$, $v'$ and $c$ are strings.
It may be useful to illustrate this definition with a picture. Below we indicate arrows of either direction by short wiggly lines, (sub)strings by long wiggly lines, a direct arrow points downwards and to the right, an inverse arrow points downwards and to the left.
\[ \xymatrix@!R=2pt@!C=5pt{
\text{Case 1: \quad} & w = & \bullet \arr^-{w'_1} & \bullet \ar@{.}[r] & \bullet \arr^-{w'_{n-1}} & \bullet \ar[dr]^-{w'_n} &                                   & & \\
                &     &                     &                    &                        &                         & \bullet \arrr^-{v}                & & \bullet \\
\text{Case 2: \quad} &     &                     &                    &                      &                         & \bullet \ar[dl]_-{v'_m} \arrr^-{w} & & \bullet \\
                & v = & \bullet \arr^-{v'_1} & \bullet \ar@{.}[r] & \bullet \arr^-{v'_{m-1}} & \bullet                 &                                   & &         \\
                & w = & \bullet \arr^-{w'_1} & \bullet \ar@{.}[r] & \bullet \arr^-{w'_{n-1}} & \bullet \ar[dr]^-{w'_n} &                                    & &         \\
\text{Case 3: \quad} &     &                     &                    &                      &                         & \bullet \ar[dl]_-{v'_m} \arrr^-{c}  & & \bullet \\
                & v = & \bullet \arr^-{v'_1} & \bullet \ar@{.}[r] & \bullet \arr^-{v'_{m-1}} & \bullet
}
\]

\begin{example} \label{ex:B2}
Consider $\rLambda(2,3,1)$, where we relabel the arrows in the figure on page~\pageref{fig:repetitive_algebra} as $a=a_{-1}$, $b=b_0$, $c=b_1$, $d=c_2$, $x=x_2$ and $y=y$ to avoid cumbersome subscripts. We have the following linear order on the strings ending at vertex $(0,-1)$:
\[ \resizebox{\textwidth}{!}{
$1_{(0,-1)} < \invd y_{-1} < a \invd y_{-1} < \cdots < c y a \invd y_{-1} < y_{-1} < c y_{-1} < \ix b c y_{-1} < \cdots < a b c \ix b c y_{-1} < b c y_{-1},$
}
\]
where we write $y_{-1} = (-1,y)$ for short, $1_{(0,-1)}$ denotes the trivial string at vertex $(0,-1)$, and we have only indicated the degree of final arrow arrow; the others can be deduced from this. All the inequalities above correspond to adding a hook, except for $cya \invd y_{-1} < y_{-1}$ and $abc \ix bcy_{-1} < bcy_{-1}$, which correspond to removing a cohook.

The AR triangle starting at $cy_{-1}$ is
\[
\xymatrix@!=0.6pc{
                          & \ix bcy_{-1}  \ar[dr] &  \\
cy_{-1}   \ar[ur] \ar[dr]  &                      & \ix bc_{-1}. \\
                          & c_{-1} \ar[ur]        &
}
\]
\end{example}

\begin{remark}
Note that $(i,a_{-1})$ and $(i,c_{n-1})$ are two arrows in $\rQ(r,n,m)$ ending at the vertex $(i,0)$. By definition of the string function $\epsilon \colon \rQ_1(r,n,m) \to \{-1,1\}$ we have $\epsilon \big( (i,a_{-1}) \big) = - \epsilon \big( (i, c_{n-1}) \big)$. Thus, $(i,a_{-1})$ and $(i,c_{n-1})$ are not comparable in the total order defined above (because their $\epsilon$ values differ).
\end{remark}

\subsection{Maps between string modules} \label{sec:maps}

It is straightforward to compute the maps between string modules. This was first observed in \cite{Crawley-Boevey} and later generalised in \cite{Crawley-Boevey2} and \cite{Krause}. We follow the neat exposition given in \cite[Section 2]{Schroer-Zimmermann}.

For a string $w$, define the set of \emph{factor strings}, $\Fac{w}$, to be the set of decompositions $w = def$ with $d,e,f \in \strings$ , where $d=d_1 \cdots d_n$ and $f=f_1 \cdots f_m$, in which we require $d$ to be trivial or $d_n \in Q_1^{-1}$ and $f$ to be trivial or $f_1 \in Q_1$.
Similarly, the set of \emph{substrings}, $\Sub{w}$, is the set of decompositions in which we require $d$ to be trivial or $d_n \in Q_1$ and $f$ to be trivial or $f_1 \in Q_1^{-1}$. A picture may be useful: on the left we illustrate a factor string decomposition and on the right, a substring decomposition.
\[
\xymatrix@!R=0.025pc@!C=0.005pc{
                                  &&         & \ar[dl]_-{d_n} \bullet \arrr^-{e} && \bullet \ar[dr]^-{f_1} &                                 &&         \\
\bullet \arrr^-{d_1 \cdots d_{n-1}} && \bullet &                                  &&                        & \bullet \arrr^-{f_2 \cdots f_m} && \bullet
}
\xymatrix@!R=0.025pc@!C=0.005pc{
\bullet \arrr^-{d_1 \cdots d_{n-1}} && \bullet \ar[dr]^-{d_n} &                    &&         & \ar[dl]_-{f_1} \bullet \arrr^-{f_2 \cdots f_m} && \bullet \\
                                  &&                        & \bullet \arrr^-{e} && \bullet &                                               &&
}
\]

A pair $((d_1,e_1,f_1),(d_2,e_2,f_2)) \in \Fac{v} \times \Sub{w}$ is called \emph{admissible} if $e_1 = e_2$ or $e_1 = \bar{e}_2$. Then the main results of \cite{Crawley-Boevey, Crawley-Boevey2} and \cite{Krause} assert:

\begin{theorem} \label{thm:maps}
Let $v,w \in \strings$ and suppose $M_v$ and $M_w$ are their corresponding string modules. Then
$
\hom(M_v,M_w) = \#\{\text{admissible pairs in } \Fac{v} \times \Sub{w}\}.
$
\end{theorem}

The following corollary is immediate.

\begin{corollary} \label{cor:maps}
Suppose $v$ and $w$ are strings such that $e(v) = e(w)$ and $\epsilon(v) = \epsilon(w)$, making $v$ and $w$ comparable in the Gei\ss\ total order. If $v \leq w$ then there is a non-zero morphism $M(v) \to M(w)$.
\end{corollary}

\subsection{Strings and maps for derived-discrete algebras}

Here we list some pertinent facts about strings and string modules for discrete derived categories from \cite{BGS}, and establish some additional routine but useful properties.

\begin{lemma}[\cite{BGS}] \label{lem:BGS-facts}
Denote the simple modules of $\LLambda$ by $S(i)$ for $-m \leq i < n$. In the coordinate system introduced in Properties~\ref{properties}, $Z^0_{0,0} = S(0)$. Then:
\begin{enumerate}[label=(\roman*)~]
\item If $m>0$ then $S(-1) = X^1_{0,0}$; in particular there is a simple module on the mouth of the $\cX$ component.
\item If $r<n$ then $S(n-r)$ lies on the mouth of the $\cY$ component.
\end{enumerate}
\end{lemma}

\noindent
The embedding $\mod{\LLambda} \hookrightarrow \mod{\rLLambda}$ maps simple modules $S(i) \mapsto \rS(0,i)$; the latter corresponds to the trivial string $1_{(0,i)}$. Since morphisms to and from a simple module cannot factor through projective-injective modules (recall that $\rLambda$ is a self-injective algebra), we obtain both $\stHom(\rS(0,i), X) = \Hom(\rS(0,i), X)$ and $\stHom(X,\rS(0,i)) = \Hom(X,\rS(0,i))$ for any $X \in \mod{\rLLambda}$.

\begin{lemma} \label{lem:appendix-mouth}
Let $A \in \ind{\Db(\LLambda)}$ with $r>1$ and let $i,k\in \IZ$, $0\leq k<r$. Then
\[ \begin{array}{ll}
  \Hom(X^k_{ii}, A) = \kk & \text{if } A \in \rayfrom{X^k_{ii}}   \cup \corayto{\SSS X^k_{ii}}     \cup \raythrough{Z^k_{ii}} ,\\[0.5ex]
  \Hom(Y^k_{ii}, A) = \kk & \text{if } A \in \corayfrom{Y^k_{ii}} \cup \rayto{\SSS Y^k_{ii}}       \cup \coraythrough{Z^k_{ii}} ,
\end{array} \]
and in all other cases the Hom spaces are zero. For $r=1$ the Hom-spaces are as above, except $\Hom(X^0_{ii},X^0_{i,i+m}) = \kk^2$.
\end{lemma}

The other two statements of Lemma~\ref{lem:mouth_hammocks} follow from these by Serre duality.

\begin{proof}
\Case{$m > 0$} 
Since, by \cite[Theorem~B]{BGS}, the action of $\tau$ and $\Sigma$ together is transitive on the set of objects at the mouths of the $\cX$ components, by Lemma~\ref{lem:BGS-facts}, we may assume that $X^k_{ii} = \rS(0,-1)$.
Note that the chain of morphisms in Properties~\ref{properties}$(5)$ corresponds to the totally ordered set of strings ending at $1_{(0,-1)}$ in the Gei\ss\ total order. By Corollary~\ref{cor:maps}, it follows that each object in this totally ordered set admits a morphism from $\rS(0,-1)$. This totally ordered set is shown in Example~\ref{ex:B2} for the algebra $\rLambda(2,3,1)$.

Now it remains to show that these are the only objects admitting morphisms from $\rS(0,-1)$. Let $w$ be a string that admits a substring decomposition $def \in \Sub{w}$ with $e = 1_{(0,-1)}$ or $\bar{e} = 1_{(0,-1)}$. We claim that $w = de (= d)$ or $w = ef (= f)$. This is clear since there is only one arrow ending at $(0,-1)$, namely $(0,a_{-2})$ when $m > 1$ and $(-1,y)$ when $m=1$. These strings (or their inverses) are precisely the strings listed in Gei\ss\ total order in Properties~\ref{properties}$(5)$. Therefore, by Theorem~\ref{thm:maps}, these are precisely the indecomposable objects admitting a morphism from $\rS(0,-1)$.  Reading this off gives $\rayfrom{S(0,1)}$ and $\corayto{\SSS S(0,1)}$ in the $\cX$ components, and $\raythrough{Z^1_{0,0}}$ in the $\cZ$ component.

The Hom-hammock of objects admitting morphisms from $\rS(0,n-r)$, which is in the $\cY$ component for any $m$, can be obtained in an analogous fashion.

\Case{$m=0$} 
We use an embedding  $\Db(\Lambda(r,n,0)) \hookrightarrow \Db(\Lambda(r,n,1))$.
By \cite[Lemma~3.1]{BGS}, the indecomposable projective $P(-1)$ lies on the mouth of the $\cX$ component.
Applying Lemma~\ref{lem:functorially-finite} to $\sP = \thick{\Db(\Lambda(r,n,1))}{P(-1)}$ yields $\sP\orth \simeq \Db(\Lambda(r,n,0))$ as in the proof of Proposition~\ref{prop:embedding-A}. Computing Hom-hammocks in $\sP\orth$ by the case $m>0$ gets the claim.
\end{proof}

\begin{remark}
In the `extended ray' of strings ending at $1_{(0,-1)}$, to obtain the part of this linearly ordered set corresponding to $\corayto{\SSS S(0,-1)}$, we consider the inverse strings of those ending at $1_{(0,-1)}$ with the direct arrow $a_{(0,-2)}$ (for $m>1$) or $y_{-1}$ (for $m=1$). We thus obtain strings starting with the corresponding inverse arrow, which gives the coray.
\end{remark}

\noindent
The next fact is used in particular for Proposition~\ref{prop:embedding-A}, the classification of silting objects.

\begin{lemma} \label{lem:proj-position}
The projective $P(n-r)\in\cZ$ in the AR quiver of $\Db(\LLambda)$.
\end{lemma}

\begin{proof}
The simple module $S(n-r+1) \in \mod{\Lambda}$ corresponds to  the trivial string $1_{(0,n-r+1)}$. One can show by direct computation that for any $n \in \IZ$ both $[n]1_{(0,n-r+1)}$ and $1_{(0,n-r+1)}[n]$ exist. This means that $1_{(0,n-r+1)}$ sits in a $\IZ A_{\infty}^{\infty}$-component of the AR quiver, for otherwise, eventually one of $[n]1_{(0,n-r+1)}$ or $1_{(0,n-r+1)}[n]$ would not be defined.  The projective $P(n-r)$ is represented by the string $\ib_{n-r}$, which is given by $1_{(0,n-r+1)}[-1]$, and hence lies in the same component as $S(n-r+1)$, i.e.\ $P(n-r)\in\cZ$.
\end{proof}

\FloatBarrier

\addtocontents{toc}{\protect{\setcounter{tocdepth}{-1}}}

\bigskip
\noindent
\resizebox{\textwidth}{!}{{Email: \texttt{broomhead@math.uni-hannover.de, david.pauksztello@manchester.ac.uk, david.ploog@uni-due.de}}}

\end{document}